\documentclass[12pt, showkeys]{amsart}%
\usepackage{color}
\usepackage{graphicx,subfigure}
\usepackage{subfigure}
\usepackage{subfigure}
\usepackage{amssymb}
\usepackage{amsmath,amsxtra}
\usepackage{multirow}
\usepackage{enumerate}
\usepackage{epstopdf}
\usepackage{cite,stmaryrd,txfonts}
\usepackage{tikz}
\usepackage{pgf}
\usetikzlibrary{snakes}
\usetikzlibrary{angles,calc,patterns}
\usetikzlibrary{shapes,arrows,automata}
\usepackage{bbm}
\usepackage{dsfont}
\usepackage{accents}
 \usepackage{float}

\usepackage{epic}
\usepackage{empheq}
\usepackage{cases}
\usepackage{diagbox}
\def\dx{\,{\rm dx}}

\newtheorem{theorem}{Theorem}[section]
\newtheorem{remark}[theorem]{Remark}
\newtheorem{proposition}[theorem]{Proposition}
\newtheorem{lemma}[theorem]{Lemma}
\newtheorem{corollary}[theorem]{Corollary}
\newtheorem{definition}[theorem]{Definition}

\newtheorem{assumption}[theorem]{Assumption}

\newcounter{mnote}
\let\oldmarginpar\marginpar
\renewcommand\marginpar[1]{\-\oldmarginpar[\raggedleft\footnotesize #1]
  {\raggedright\footnotesize #1}}

\usepackage{footnote}
\usepackage{booktabs}

\usepackage{rotating}

\numberwithin{equation}{section}

\usepackage{cite}

\usepackage[colorlinks,linkcolor=black,
anchorcolor=black,
citecolor=black]{hyperref}
\usepackage[shortlabels]{enumitem}
\setlist[enumerate]{nosep}
\usepackage{color, soul}
\soulregister\cite7
\usepackage{setspace}

\def\rn{\boldsymbol{\mathbb{R}^{\boldsymbol{\mathsf{n}}}}}

\def\Hdiv{H({\rm div})}
\def\HdivT{H({\rm div},T)}
\def\HdivOmega{H({\rm div},\Omega)}
\def\hlk{H\Lambda^k} 

\def\uV{\undertilde{V}}

\def\uL{\undertilde{L}}

\def\usigma{\undertilde{\sigma}}
\def\utau{\undertilde{\tau}}

\def\ualpha{\undertilde{\alpha}}

\def\ua{\undertilde{a}}
\def\ub{\undertilde{b}}
\def\uf{\undertilde{f}}

\def\ux{\undertilde{x}}

\def\curl{{\rm curl}}

\def\dv{{\rm div}}

\def\omp{\ominus^\perp}
\def\opp{\oplus^\perp}

\def\od{\mathbf{d}}
\def\oT{\mathbf{T}}

\def\odelta{\boldsymbol{\delta}}
\def\okappa{\boldsymbol{\kappa}}

\def\icr{{\sf icr}}

\def\hf{\boldsymbol{\mathfrak{H}}}

\def\R{\mathcal{R}}
\def\N{\mathcal{N}}

\usepackage[mathscr]{eucal}

\def\xA{\boldsymbol{\mathbf{A}}}
\def\xB{\boldsymbol{\mathbf{B}}}
\def\xD{\boldsymbol{\mathbf{D}}}
\def\xH{{\boldsymbol{\mathbf{H}}}}
\def\xT{\boldsymbol{\mathbf{T}}}
\def\xX{{\boldsymbol{\mathbf{X}}}}

\def\xv{\boldsymbol{\mathbf{v}}}
\def\xw{\boldsymbol{\mathbf{w}}}

\def\yD{\boldsymbol{\mathbbm{D}}}
\def\yE{\boldsymbol{\mathbbm{E}}}
\def\yS{\boldsymbol{\mathbbm{S}}}
\def\yT{\boldsymbol{\mathbbm{T}}}
\def\yY{{\boldsymbol{\mathbbm{Y}}}}

\def\yb{\boldsymbol{\mathbbm{b}}}
\def\yw{\boldsymbol{\mathbbm{w}}}

\def\zE{\boldsymbol{\mathsf{E}}}
\def\zS{\boldsymbol{\mathsf{S}}}
\def\zZ{{\boldsymbol{\mathsf{Z}}}}

\def\fomega{\boldsymbol{\omega}}
\def\fmu{\boldsymbol{\mu}}

\def\fzeta{\boldsymbol{\zeta}}
\def\feta{\boldsymbol{\eta}}

\def\fsigma{\boldsymbol{\sigma}}
\def\ftau{\boldsymbol{\tau}}

\def\fvartheta{\boldsymbol{\vartheta}}
\def\fvarsigma{\boldsymbol{\varsigma}}

\def\fvarphi{\boldsymbol{\varphi}}
\def\fpsi{\boldsymbol{\psi}}

\def\fiota{\mathbf{\iota}}
\def\fchi{\mathbf{\chi}}

\def\fvarrho{\mathbf{\varrho}}
\def\fnu{\mathbf{\nu}}
\def\fxi{\mathbf{\xi}}
\def\fvarpi{\mathbf{\varpi}}

\def\fW{\boldsymbol{W}}

\def\fL{\boldsymbol{L}}

\def\ff{\boldsymbol{f}}

\def\ixalpha{\boldsymbol{\alpha}} 

\begin{document}

\title{Nonconforming finite element spaces for $H\Lambda^k$ in $\mathbb{R}^n$}

\author{Shuo Zhang}
\address{LSEC, Institute of Computational Mathematics and Scientific/Engineering Computing, Academy of Mathematics and System Sciences, Chinese Academy of Sciences, Beijing 100190; University of Chinese Academy of Sciences, Beijing, 100049; People's Republic of China}
\email{szhang@lsec.cc.ac.cn}

\thanks{The research is partially supported by NSFC (12271512, 11871465).}

\begin{abstract}
This paper constructs a unified family of nonconforming finite element spaces for $H\Lambda^k$ in $\rn$ ($0\leqslant k\leqslant n$, $n\geqslant 1$). The spaces employ piecewise Whitney forms as shape functions, and include the lowest-degree Crouzeix-Raviart element space for $H\Lambda^0$. Optimal approximations and uniform discrete Poincar\'e inequalities are presented. Based on the newly constructed finite element spaces, discrete de Rham complexes with commutative diagrams, and the discrete Helmholtz decomposition and Hodge decomposition for piecewise constant spaces are established. All discrete operators involved are local, acting cell-wise. A framework of nonconforming finite element exterior calculus is then established, and is naturally connected to the classical conforming one. The cooperation of conforming and nonconforming finite element spaces leads to new discretization schemes of the Hodge Laplace problem. 

The new finite element spaces are constructed by a novel approach that seeks to mimic the dual connections between adjoint operators; novel construction methods and basic estimations are presented. Although the new spaces do not fit Ciarlet's finite element definition, they admit locally supported basis functions each spanning at most two adjacent cells, which makes the computation of the local stiffness matrices and the assembling of the global stiffness matrix implementable by following the standard procedure. Some numerical experiments are given to show the implementability and the performance of the new kind of spaces. 

\end{abstract}



\maketitle

\tableofcontents

\section{Introduction}

Conforming finite elements for exterior differential forms have been extensively studied, based on which conforming finite element exterior calculus has been well established; we refer to, e.g., \cite{Arnold.D;Falk.R;Winther.R2010bams,Arnold.D2018feec,Arnold.D;Falk.R;Winther.R2006acta,Boffi.D;Brezzi.F;Fortin.M2013,Hiptmair.R2002acta} and the references therein for details. Naturally, the research has now reached a point where extension is appropriate to nonconforming methods. Well-designed nonconforming methods can possess many characteristics that conforming ones lack, with the (lowest-degree) Crouzeix-Raviart element \cite{Crouzeix.M;Raviart.P1973} being a typical example. The Crouzeix-Raviart element, originally designed for $H^1$ which is equivalent to $H\Lambda^0$ for $0$-forms, is among the most widely used finite elements. It can be distinguished from conforming ones with kinds of practically crucial properties, including, e.g.,
\begin{itemize}
\item Different from conforming interpolators discussed in \cite{Christiansen.S;Winther.R2008smoothed,Clement.P1975,Gawlik.E;Holst.M;Licht.M2021,Licht.M2019mcweakly,Licht.M2019mc,Ern.A;Guermond.J2017M2AN,Scott.R;Zhang.S1990intp,Falk.R;Winther.R2014}, the Crouzeix-Raviart element admits a cell-wise defined\footnote{Here and in the sequel, by ``locally defined" or ``cell-wise defined", we mean if two functions $u$ and $v$ are equal on a cell $T$, then their respective interpolations $\mathbb{I}u$ and $\mathbb{I}v$ are equal on $T$.} stable interpolator which works for functions in $H^1$ without using the inter-cell regularization, smoothing or averaging techniques. 
\item In the construction of Helmholtz orthogonal decomposition of piecewise constants, which cannot be established when restricted to conforming element spaces, the lowest-degree Crouzeix-Raviart element plays an irreplaceable role~\cite{Arnold.D;Falk.R1989,monk1991mixed}. 
\item Applied to the computation of Laplacian eigenvalues, the lowest-degree Crouzeix-Raviart element scheme may yield asymptotic lower bounds to the exact eigenvalues~\cite{Armentano.M;Duran.R2004asymptotic}, which differs essentially from conforming ones.
\end{itemize}
These properties may indicate the potential theoretical and practical significance of nonconforming methods compared to conforming ones. This paper investigates nonconforming finite element discretizations for general exterior differential forms, and particularly, generalizes the Crouzeix-Raviart element for $H\Lambda^0$ to a unified family for $H\Lambda^k$ for $0\leqslant k\leqslant n$ in $\rn$ by a novel approach. Nonconforming finite element exterior calculus can then be established based on these spaces.
~\\

Attempts to generalize the Crouzeix-Raviart elements have been devoted to the $\Hdiv$ problems \cite{Arbogast.T;Correa.M2016,Shi.D;Pei.L2008low,Quan.Q;Ji.X;Zhang.S2022}. Following directly from Crouzeix-Raviart element, these elements all use the integral of the normal components as nodal parameters. For these elements, the crucial property of the Crouzeix-Raviart element, namely cell-wise defined nodal interpolator, cannot be validated for functions with only $\Hdiv$ regularity, nor can an associated discrete Helmholtz decomposition be established. Further, if we try to embed such an $H(\dv)$ element into a discretized de Rham complex, which is a crucial issue for the discretization of exterior differential operators, the continuity restriction for the corresponding $H^1$ finite element is the evaluation at vertices. As well known, the continuity of the evaluation at vertices is neither sufficient nor necessary for a finite element to work for $H^1$ problems, and the weak continuity condition for these $\Hdiv$ elements is not as reasonable as the original Crouzeix-Raviart element. It is suggested in \cite{Bringmann.P;Ketteler.J;Schedensack.M2024} that vector Crouzeix-Raviart element can be used for $H(\curl)$ in three dimension; though, the same obstacles can be come across. 

Different from existing attempts, instead of establishing the space by imposing local continuity primally, the main ingredient of the new approach is to reveal and mimic the relationship between adjoint operators, inspired by a new interpretation of the Crouzeix-Raviart element. Actually, beyond being a {\bf consequence}, the well-known integration by part formula, on the lowest-degree Crouzeix-Raviart element space $V^{\rm CR}_h$ and the lowest-degree Raviart-Thomas element space $\uV{}^{\rm RT}_{h0}$ on a grid $\mathcal{G}_h$,
\begin{equation}\label{eq:greencrrt}
\sum_{T\in\mathcal{G}_h}\int_T\nabla v_h\utau{}_h+\int_Tv_h\dv\utau{}_h=0,\ \ \mbox{for}\ v_h\in V^{\rm CR}_h\ \mbox{and}\ \utau{}_h\in \uV{}^{\rm RT}_{h0}
\end{equation}
also serves as a {\bf sufficient} condition for a piecewise linear polynomial function to belong to $V^{\rm CR}_h$, in accordance with the adjoint relation between $(\dv,H_0(\dv))$ and $(\nabla,H^1)$. Namely, $V^{\rm CR}_h$ can be equivalently figured out as
\begin{equation}
V^{\rm CR}_h=\left\{v_h\ \mbox{is\ piecewise\ linear,\ such\ that}\  \sum_{T\in\mathcal{G}_h}\int_T\nabla v_h\utau{}_h+\int_Tv_h\dv\utau{}_h=0\ \forall\,\utau{}_h\in \uV{}^{\rm RT}_{h0}\right\}.
\end{equation}
This observation then hints quite a natural approach to construct a finite element space by constructing discrete adjoint relationships. By the aid of the existing conforming Whitney forms, in this paper, the methodology can be directly applied to design a family of nonconforming finite element spaces for $H\Lambda^k$ with piecewise Whitney forms. 
~\\

The approach of inheriting the adjoint relationship can actually lead to natural advantages. Several properties emerge naturally from the construction of the finite element spaces. A basic one is the consistency property, which follows directly. Then, cell-wise defined global interpolators can be constructed for functions in $\hlk$ with no extra regularity needed; the interpolators are stable in broken $\hlk$ norm and provide optimal approximation to all functions in $\hlk$. Combined with the global interpolators, these newly constructed spaces are connected by piecewise operations of $\od^k$ to form nonconforming finite element de Rham complexes, as well as commutative diagrams with the de Rham Hilbert complexes. Further, the Helmholtz and Hodge decompositions of the piecewise constant $k$-forms follow from the discrete adjoint relation. It is worth noting that the Poincar\'e-Leftschetz duality Theorem \ref{thm:displd} can be reconstructed by the respective discrete harmonic spaces by conforming and nonconforming finite element spaces. With the structural properties given in Section \ref{sec:ncdcx}, a framework of nonconforming finite element exterior calculus is established, and is naturally linked to the classical conforming one by the discrete complex duality \eqref{eq:complexduality} and the discrete Poincar\'e-Lefschetz duality.

On the other hand, we have to remark that, in contrast to the conforming Whitney forms, the nonconforming finite element spaces defined in this paper may not correspond to a ``finite element"(triple) in Ciarlet's sense~\cite{Ciarlet.P1978book}. Therefore, some basic features of the finite element methods cannot be dealt with in standard ways. Two main obstacles are: first, it is not any longer straightforward to figure out the basis functions of the global finite element spaces, and second, it is difficult, if not impossible, to follow the standard procedure to prove the uniform discrete Poincar\'e inequalities. In this paper, we develop nonstandard approaches to circumvent the obstacles. For every newly designed finite element space, we prove the existence of a set of basis functions which each is supported on no more than two cells, and the relevant numerical scheme can be implemented by the standard routine for the finite element in Ciarlet's sense. Some numerical experiments are provided to verify the implementability of the new finite element functions. We also prove that the constant of the discrete Poincar\'e inequality of a newly designed finite element space is asymptotically equal to that of an associated conforming Whitney form space which is proved uniformly bounded; it then follows that the discrete Poincar\'e inequality holds uniformly for the new spaces. 

Since nonconforming finite element spaces are constructed for $(\od^k,H\Lambda^k)$ and particularly discrete Hodge decompositions are constructed accordingly, new discretization schemes can be developed. Meanwhile, dual structures can be further investigated with more applications. We investigate the dual roles of conforming and nonconforming spaces by constructing some new finite element schemes for the Hodge Laplace problem with nonconforming spaces. The two finite element spaces connect with each other within their respective discretization schemes through classical mixed formulations, and their roles are complementary within the discretization scheme of a new mixed formulation.
~\\

The remainder of the paper is organized as follows. In the remaining part of this section, we collect some preliminaries and notations. In Section \ref{sec:newfem}, we use the two-dimensional $\Hdiv$ problem for instance to illustrate the main features of the new type of finite element spaces, including the construction of the new space, the locally-supported basis functions, the basic error estimation by cell-wise defined interpolators, and numerical experiments for the implementability of the new finite element functions. In Section \ref{sec:ncdcx}, a family of nonconforming finite element spaces are constructed for $H\Lambda^k$ in $\rn$, $0\leqslant k\leqslant n$, with the Crouzeix-Raviart element space being the one for $H\Lambda^0$. Optimal approximation and uniform Poincar\'e inequalities are established. Based on these finite element spaces, theory of nonconforming finite element exterior calculus is constructed, including the Helmholtz/Hodge decomposition for piecewise constant $k$-forms, the discrete Poincar\'e-Lefschetz duality,  the discrete de Rham complex and commutative diagrams. Then in Section \ref{sec:dishl}, the newly-designed nonconforming spaces are used for the discretization of the Hodge Laplace problem. The correspondent and complementary connections between the conforming and nonconforming spaces are investigated with classical and new mixed formulations. Finally, in Section \ref{sec:conc}, some conclusions and discussions are given. 
~\\

\paragraph{\bf Preliminaries and Notations} In the sequel of the paper, we use $\N$ and $\R$ to denote the null space and the range of certain operators. Namely, for example, $\N(\oT,\xD)$ denotes $\left\{\xv\in\xD:\oT\xv=0\right\}$, and $\R(\oT,\xD)$ denotes $\left\{\oT\xv:\xv\in\xD\right\}.$ For a Hilbert space $\xH$, we use the notations $\opp_\xH$ and $\omp_\xH$ to denote the orthogonal summation and orthogonal difference; namely, for two spaces $\xA$ and $\xB$ in $\xH$, the presentation $\xA\opp_\xH\xB$ implies that $\xA$ and $\xB$ are orthogonal in $\xH$, and evaluates as the direct summation of $\xA$ and $\xB$; for $\xA\subset\xB\subset \xH$, $\xB\omp_\xH\xA$ evaluates as the orthogonal complementation of $\xA$ in $\xB$. The subscript $\xH$ can occasionally be dropped. 

For $\Omega$ a domain and $T\subset \Omega$, we use $E_T^\Omega:L^1(T)\to L^1(\Omega)$ for the extension operator defined by $E_T^\Omega v=v$ on $T$ and $E_T^\Omega v=0$ elsewhere. For $V_T\subset L^1(T)$, we use $E_T^\Omega V_T$ for short of $\R(E_T^\Omega,V_T)$.

We use $\od^k$ and $\odelta_k$ for the exterior {\it differential} and {\it codifferential} operators on $\Lambda^k$. $\odelta_k=(-1)^{kn}\star\od^{n-k}\star$, $\star$ being the Hodge star operator. Denote, on the domain $\Xi$,
$$
H\Lambda^k(\Xi):=\left\{\fomega\in L^2\Lambda^k(\Xi):\od^k\fomega\in L^2\Lambda^{k+1}(\Xi)\right\},\ \ \ 0\leqslant k\leqslant n-1,
$$
and by $H_0\Lambda^k(\Xi)$ the closure of $\mathcal{C}_0^\infty\Lambda^k(\Xi)$ in $H\Lambda^k(\Xi)$. Denote 
$$
H^*\Lambda^k(\Xi):=\left\{\fmu\in L^2\Lambda^k(\Xi):\odelta_k\fmu\in L^2\Lambda^{k-1}(\Xi)\right\},\ \ \ 1\leqslant k\leqslant n,
$$
and $H^*_0\Lambda^k(\Xi)$ the closure of $\mathcal{C}_0^\infty\Lambda^k(\Xi)$ in $H^*\Lambda^k(\Xi)$. $\Xi$ can occasionally be dropped. The spaces of harmonic forms are $\hf\Lambda^k:=\N(\od^k,H\Lambda^k)\omp\R(\od^{k-1},H\Lambda^{k-1})$, $\hf_0\Lambda^k:=\N(\od^k,H_0\Lambda^k)\omp\R(\od^{k-1},H_0\Lambda^{k-1})$, $\hf^*\Lambda^k:=\N(\odelta_k,H^*\Lambda^k)\omp\R(\odelta_{k+1},H^*\Lambda^{k+1})$, and $\hf^*_0\Lambda^k:=\N(\odelta_k,H^*_0\Lambda^k)\omp\R(\odelta_{k+1},H^*_0\Lambda^{k+1})$. As the Helmholtz decompositions hold that
$$
\N(\od^k,H\Lambda^k)\opp \R(\odelta_{k+1},H^*_0\Lambda^{k+1})=L^2\Lambda^k=\R(\od^{k-1},H\Lambda^{k-1})\opp\N(\odelta_k,H^*_0\Lambda^k),
$$
it follows that $\hf\Lambda^k=\hf^*_0\Lambda^k$ and $\hf_0\Lambda^k=\hf^*\Lambda^k$. This is the Poincar\'e-Lefschetz duality(cf.  \cite[Section 4.5.5]{Arnold.D2018feec}) which links two dual complexes connected by $\od^k$ and $\odelta_k$, respectively.

The space of Whitney forms is denoted as (\cite{Arnold.D;Falk.R;Winther.R2006acta,Arnold.D;Falk.R;Winther.R2010bams,Arnold.D2018feec})  $\displaystyle\mathcal{P}^-_1\Lambda^k=\mathcal{P}_0\Lambda^k+\okappa(\mathcal{P}_0\Lambda^{k+1}),$ where the Koszul operator $\okappa$ is $\displaystyle\okappa(\dx^{\ixalpha_1}\wedge\dots\wedge\dx^{\ixalpha_k}):=\sum_{j=1}^k(-1)^{j+1}x^{\ixalpha_j}\dx^{\ixalpha_1}\wedge\dots\wedge\dx^{\ixalpha_{j-1}}\wedge\dx^{\ixalpha_{j+1}}\wedge\dots\wedge\dx^{\ixalpha_k}$ for 
$\displaystyle\ixalpha\in\mathbb{IX}_{k,n}:=\left\{\ixalpha=(\ixalpha_1,\dots,\ixalpha_k)\in\mathbb{N}^k:1\leqslant \ixalpha_1<\ixalpha_2<\dots<\ixalpha_k\leqslant n,\ \mathbb{N}\ \mbox{the\ set\ of\ integers}\right\}$, the set of $k$-indices, $k\leqslant n$. Note that $\mathcal{P}^-_1\Lambda^0=\mathcal{P}_1\Lambda^0$ and $\mathcal{P}^-_1\Lambda^n=\mathcal{P}_0\Lambda^n$. Denote the Whitney forms associated with the operator $\odelta_k$ by
$\displaystyle\mathcal{P}^{*,-}_1\Lambda^k:=\star(\mathcal{P}^-_1\Lambda^{n-k}).$ Note that 
\begin{equation}\label{eq:n=r=c}
\N(\od^k,\mathcal{P}^-_1\Lambda^k)=\R(\od^{k-1},\mathcal{P}^-_1\Lambda^{k-1})=\mathcal{P}_0\Lambda^k=\R(\odelta_{k+1},\mathcal{P}^{*,-}_1\Lambda^{k+1})=\N(\odelta_k,\mathcal{P}^{*,-}_1\Lambda^k).
\end{equation}

Denote, on a simplicial subdivision $\mathcal{G}_h$ of $\Omega$, $0\leqslant k\leqslant n$,
\begin{equation} 
\displaystyle\mathcal{P}^-_1\Lambda^k(\mathcal{G}_h):=\bigoplus_{T\in\mathcal{G}_h}E_T^\Omega \mathcal{P}^-_1\Lambda^k(T), \ \mbox{and}\ \ \displaystyle\mathcal{P}^{*,-}_1\Lambda^k(\mathcal{G}_h):=\bigoplus_{T\in\mathcal{G}_h}E_T^\Omega\mathcal{P}^{*,-}_1\Lambda^k(T).
\end{equation}
Here and in the sequel, the subscript $``\cdot_h"$ denotes mesh dependence. In particular, an operator with the subscript $``\cdot_h"$ indicates that the operation is performed cell by cell. 

The conforming finite element spaces with Whitney forms are
$\fW_h\Lambda^k:=\mathcal{P}^-_1(\mathcal{G}_h)\cap H\Lambda^k$, $\fW_{h0}\Lambda^k:=\mathcal{P}^-_1(\mathcal{G}_h)\cap H_0\Lambda^k$, $\fW^*_h\Lambda^k:=\mathcal{P}^{*,-}_1(\mathcal{G}_h)\cap H^*\Lambda^k$, and $\fW^*_{h0}\Lambda^k:=\mathcal{P}^{*,-}_1(\mathcal{G}_h)\cap H^*_0\Lambda^k.$ Note that the spaces defined this way are respectively identical to the finite element spaces with piecewise Whitney forms defined by the continuity of the nodal parameters~\cite{Arnold.D2018feec}. Denote the spaces of discrete harmonic forms by 
$\hf_h\Lambda^k:=\mathcal{N}(\od^k,\mathbf{W}_h\Lambda^k)\omp \mathcal{R}(\od^{k-1},\mathbf{W}_h\Lambda^{k-1})$, $\hf_{h0}\Lambda^k:=\mathcal{N}(\od^k,\mathbf{W}_{h0}\Lambda^k)\omp \mathcal{R}(\od^{k-1},\mathbf{W}_{h0}\Lambda^{k-1})$, $\hf^*_h\Lambda^k:=\mathcal{N}(\odelta_k,\mathbf{W}_h^*\Lambda^k)\omp \mathcal{R}(\odelta_{k+1},\mathbf{W}_h^*\Lambda^{k+1})$, and $\hf^*_{h0}\Lambda^k:=\mathcal{N}(\odelta_k,\mathbf{W}_{h0}^*\Lambda^k)\omp \mathcal{R}(\odelta_{k+1},\mathbf{W}_{h0}^*\Lambda^{k+1})$.

Given $T$ a simplex, denote, associated with $T$, $\tilde{x}^j=x^j-c_j$, where $c_j$ is a constant such that $\int_T\tilde{x}^j=0$, and $\okappa_T$, a Koszul operator on $T$, by for $\ixalpha\in\mathbb{IX}_{k,n}$ 
$$
\okappa_T(\dx^{\ixalpha_1}\wedge\dots\wedge\dx^{\ixalpha_k}):=\sum_{j=1}^k(-1)^{(j+1)}\tilde{x}^{\ixalpha_j}\dx^{\ixalpha_1}\wedge\dots\dx^{\ixalpha_{j-1}}\wedge\dx^{\ixalpha_{j+1}}\wedge\dots\wedge\dx^{\ixalpha_k}.
$$
Then $\od^{k-1}\okappa_T(\dx^{\ixalpha_1}\wedge\dots\dx^{\ixalpha_k})=k\dx^{\ixalpha_1}\wedge\dots\dx^{\ixalpha_k}. $ By the aid of $\okappa_T$, we can rewrite the Whitney forms as $\mathcal{P}^-_1\Lambda^k(T)=\mathcal{P}_0\Lambda^k(T)\opp\okappa_T(\mathcal{P}_0\Lambda^{k+1}(T))$, {orthogonal\ in}\ $L^2\Lambda^k(T)$. We further use $\okappa_h$ to denote the operation of $\okappa_T$ cell by cell. Denote $\okappa^\delta:=\star\circ\okappa\circ\star$, $\okappa^\delta_T:=\star\circ\okappa_T\circ\star$, and $\okappa^\delta_h:=\star\circ\okappa_h\circ\star$.

\section{A nonconforming $H(\dv)$ finite element space}
\label{sec:newfem}

In this section, we use the two-dimensional $\Hdiv$ problem for instance to illustrate the main features of the new type of finite element spaces studied in this paper. 

Let $\Omega\subset\mathbb{R}^2$ denote a polygon. As usual, we use $\nabla$ and $\dv$ to denote the gradient operator and divergence operator, respectively, and we use $H^1(\Omega)$, $H^1_0(\Omega)$, $H(\dv,\Omega)$, $H_0(\dv,\Omega)$, $L^2(\Omega)$ and $L^2_0(\Omega)$ to denote certain Sobolev (Lebesgue) spaces. For here, we denote vector-valued quantities by undertilde $``\undertilde{\cdot}"$. We use $(\cdot,\cdot)$ with subscripts to represent $L^2$ inner product. 

For this planar domain, we specifically use $\mathcal{T}_h$ for a shape-regular subdivision of $\Omega$ with mesh size $h$ that consists of triangles, such that $\overline\Omega=\cup_{T\in\mathcal{T}_h}\overline T$ and every boundary vertex is connected to at least one interior vertex. Denote by $\mathcal{E}_h$, $\mathcal{E}_h^i$, $\mathcal{E}_h^b$, $\mathcal{X}_h$, $\mathcal{X}_h^i$ and $\mathcal{X}_h^b$ the set of edges, interior edges, boundary edges, vertices, interior vertices and boundary vertices, respectively. We use $\mathbf{n}$ for the outward unit normal vector with respect to a triangle.

Let $\mathbb{V}^1_h$ denote the continuous piecewise linear element space, and $\uV^{\rm RT}_h$ denote the Raviart-Thomas \cite{Raviart.P;Thomas.J1977} element space of lowest degree on $\mathcal{T}_h$. Denote $\mathbb{V}^1_{h0}:=\mathbb{V}^1_h\cap H^1_0(\Omega)$ and $\uV^{\rm RT}_{h0}:=\uV^{\rm RT}_h\cap H_0(\dv,\Omega)$.  On a triangle $T$, denote the space of the lowest-degree Raviart-Thomas shape functions by $\mathbb{RT}(T):={\rm span}\left\{\ualpha+\beta\ux:\ualpha\in\mathbb{R}^2,\beta\in\mathbb{R}\right\}$. Then
\begin{equation}\label{eq:nabladivdual}
\mathcal{R}(\dv,\mathbb{RT}(T))=\mathbb{R}=\mathcal{N}(\nabla,P_1(T)),\ \ \ \mbox{and}\  \ \mathcal{N}(\dv,\mathbb{RT}(T))=\mathbb{R}^2=\mathcal{R}(\nabla,P_1(T)).
\end{equation}
Denote $\displaystyle\mathbb{RT}(\mathcal{T}_h):=\bigoplus_{T\in\mathcal{T}_h}E_T^\Omega\mathbb{RT}(T)$. We define the nonconforming finite element spaces
\begin{equation}\label{eq:2drtabc}
\mathbb{RT}^{\rm nc}_h:=\left\{\utau{}_h\in \mathbb{RT}(\mathcal{T}_h):\sum_{T\in\mathcal{T}_h}(\utau{}_h,\nabla \xv_h)_T+(\dv\utau{}_h,\xv_h)_T=0,\ \forall\,\xv_h\in \mathbb{V}^1_{h0}\right\},
\end{equation}
and
\begin{equation}\label{eq:2drt0abc}
\mathbb{RT}^{\rm nc}_{h0}:=\left\{\utau{}_h\in \mathbb{RT}(\mathcal{T}_h):\sum_{T\in\mathcal{T}_h}(\utau{}_h,\nabla \xv_h)_T+(\dv\utau{}_h,\xv_h)_T=0,\ \forall\,\xv_h\in \mathbb{V}^1_h\right\}.
\end{equation}

Note that $\mathbb{RT}^{\rm nc}_h$ does not confirm to Ciarlet's finite element definition. In Section \ref{sec:locbrtabc}, we will present sets of locally supported basis functions for each of $\mathbb{RT}^{\rm nc}_h$ and $\mathbb{RT}^{\rm nc}_{h0}$ for their implementability. In Section \ref{subsec:intrt}, we establish a cell-wise defined projective interpolator for $\Hdiv$, and prove optimal approximation and stability properties of $\mathbb{RT}^{\rm nc}_h$ and $\mathbb{RT}^{\rm nc}_{h0}$ directly without the aid of the classical Raviart-Thomas element.

\subsection{Locally supported global basis functions of $\mathbb{RT}^{\rm nc}_{h0}$ and $\mathbb{RT}^{\rm nc}_h$}
\label{sec:locbrtabc}

\subsubsection{Structures of $\mathbb{RT}(T)$ on a triangle $T$ and $\mathbb{RT}(\mathcal{T}_h)$ on $\mathcal{T}_h$}
For a cell $T\in\mathcal{T}_h$, we use $a_i$ (located at $\ua{}_i$) and  $e_i$ for the vertices and opposite edges, $h_i$ being the height on $e_i$, $i=1:3$. Let $\lambda_i$ be the barycentric coordinates. Let $|e_i|$ and $|h_i|$ denote the length of $e_i$ and $h_i$, respectively, and let $S$ denote the area; cf. Figure \ref{label:threebasis}. 
\begin{figure}[htbp]
\begin{tikzpicture}[scale=0.6]

\path 	coordinate (a1) at (4,3)
coordinate (a2) at (0,0)
coordinate (a3) at (5,0)
coordinate (b1) at ($ (0,0)!.5!(5,0) $)
coordinate (b2) at ($ (4,3)!.5!(5,0) $)
coordinate (b3) at ($ (0,0)!.5!(4,3) $)
coordinate (O) at ($ (a1)!.6!(b1) $)
coordinate (c1) at ($ (b1)!.5cm!-90:(a3) $)
coordinate (c2) at ($ (b2)!.5cm!-90:(a1) $)
coordinate (c3) at ($ (b3)!.5cm!-90:(a2) $)
coordinate (d1) at ($ (b1)!1cm!0:(a3) $)
coordinate (d2) at ($ (b2)!1cm!0:(a1) $)
coordinate (d3) at ($ (b3)!1cm!0:(a2) $);

\draw[line width=.4pt]  (a1) -- (a2) -- (a3) -- cycle;

		\draw [->]    (b1) -- (c1);
		\draw [->]    (b2) -- (c2);
		\draw [->]    (b3) -- (c3);
		
		\draw [->]    (b1) -- (d1);
		\draw [->]    (b2) -- (d2);
		\draw [->]    (b3) -- (d3);

		\draw[fill] (a1) circle [radius=0.02];
		\draw[fill] (a2) circle [radius=0.02];
		\draw[fill] (a3) circle [radius=0.02];
		
		\node[above] at (a1) {$a_1$};
		\node[left] at (a2) {$a_2$};
		\node[right] at (a3) {$a_3$};

		\node[below] at (c1) {$-1/|e_1|$};
		
		\node[right] at (c2) {$+1/|e_2|$};
		
		\node[left] at ($(a2)!0.67!8:(a1)$) {$+1/|e_3|$};

		\node[above right] at (b1) {$e_1$};
		\node[left] at (b2) {$e_2$};
		\node[below] at ($(a2)!0.53!(a1)$) {$e_3$};
		
		\node at(3,-0.35){$\mathbf{n}_1$};
		\node at(4.85,1.95){$\mathbf{n}_2$};
		\node at(1.5,1.55){$\mathbf{n}_3$};

		\node at (O) {$\undertilde{b}_T^{a_1}$};
	
\end{tikzpicture}
\hspace{-0.4in}
\begin{tikzpicture}[scale=0.6]

\path 	coordinate (a1) at (4,3)
coordinate (a2) at (0,0)
coordinate (a3) at (5,0)
coordinate (b1) at ($ (0,0)!.5!(5,0) $)
coordinate (b2) at ($ (4,3)!.5!(5,0) $)
coordinate (b3) at ($ (0,0)!.5!(4,3) $)
coordinate (O) at ($ (a1)!.6!(b1) $)
coordinate (c1) at ($ (b1)!.5cm!-90:(a3) $)
coordinate (c2) at ($ (b2)!.5cm!-90:(a1) $)
coordinate (c3) at ($ (b3)!.5cm!-90:(a2) $)
coordinate (d1) at ($ (b1)!1cm!0:(a3) $)
coordinate (d2) at ($ (b2)!1cm!0:(a1) $)
coordinate (d3) at ($ (b3)!1cm!0:(a2) $);

\draw[line width=.4pt]  (a1) -- (a2) -- (a3) -- cycle;

		\draw [->]    (b1) -- (c1);
		\draw [->]    (b2) -- (c2);
		\draw [->]    (b3) -- (c3);
		
		\draw [->]    (b1) -- (d1);
		\draw [->]    (b2) -- (d2);
		\draw [->]    (b3) -- (d3);

		\draw[fill] (a1) circle [radius=0.02];
		\draw[fill] (a2) circle [radius=0.02];
		\draw[fill] (a3) circle [radius=0.02];
		
		\node[above] at (a1) {$a_1$};
		\node[left] at (a2) {$a_2$};
		\node[right] at (a3) {$a_3$};

		\node[below] at (c1) {$+1/|e_1|$};
		
		\node at (c2) {$\ \ \ \ \ -1/|e_2|$};
		
		\node[left] at ($(a2)!0.67!8:(a1)$) {$+1/|e_3|$};

		\node[above right] at (b1) {$e_1$};
		\node[left] at (b2) {$e_2$};
		\node[below] at ($(a2)!0.53!(a1)$) {$e_3$};

		\node at (O) {$\undertilde{b}_T^{a_2}$};
	
\end{tikzpicture}
\hspace{-0.3in}
\begin{tikzpicture}[scale=0.6]

\path 	coordinate (a1) at (4,3)
coordinate (a2) at (0,0)
coordinate (a3) at (5,0)
coordinate (b1) at ($ (0,0)!.5!(5,0) $)
coordinate (b2) at ($ (4,3)!.5!(5,0) $)
coordinate (b3) at ($ (0,0)!.5!(4,3) $)
coordinate (O) at ($ (a1)!.6!(b1) $)
coordinate (c1) at ($ (b1)!.5cm!-90:(a3) $)
coordinate (c2) at ($ (b2)!.5cm!-90:(a1) $)
coordinate (c3) at ($ (b3)!.5cm!-90:(a2) $)
coordinate (d1) at ($ (b1)!1cm!0:(a3) $)
coordinate (d2) at ($ (b2)!1cm!0:(a1) $)
coordinate (d3) at ($ (b3)!1cm!0:(a2) $);

\draw[line width=.4pt]  (a1) -- (a2) -- (a3) -- cycle;

		\draw [->]    (b1) -- (c1);
		\draw [->]    (b2) -- (c2);
		\draw [->]    (b3) -- (c3);
		
		\draw [->]    (b1) -- (d1);
		\draw [->]    (b2) -- (d2);
		\draw [->]    (b3) -- (d3);

		\draw[fill] (a1) circle [radius=0.02];
		\draw[fill] (a2) circle [radius=0.02];
		\draw[fill] (a3) circle [radius=0.02];
		
		\node[above] at (a1) {$a_1$};
		\node[left] at (a2) {$a_2$};
		\node[right] at (a3) {$a_3$};

		\node[below] at (c1) {$+1/|e_1|$};
		
		\node[right] at (c2) {$+1/|e_2|$};
		
		\node[left] at ($(a2)!0.67!8:(a1)$) {$-1/|e_3|$};

		\node[above right] at (b1) {$e_1$};
		\node[left] at (b2) {$e_2$};
		\node[below] at ($(a2)!0.53!(a1)$) {$e_3$};

		\node at (O) {$\undertilde{b}_T^{a_3}$};
	
\end{tikzpicture}
\includegraphics[width=0.32\textwidth]{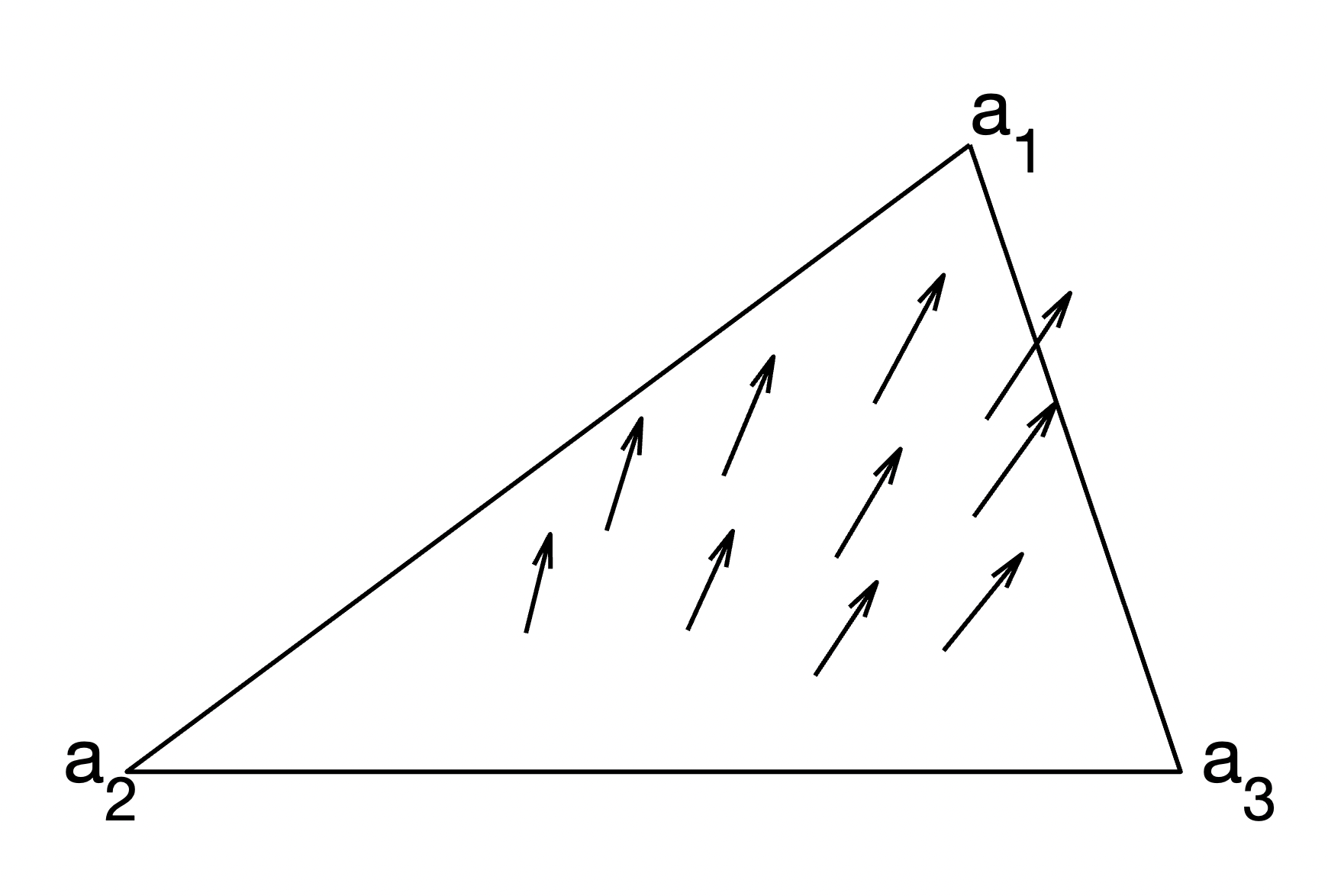}
\includegraphics[width=0.32\textwidth]{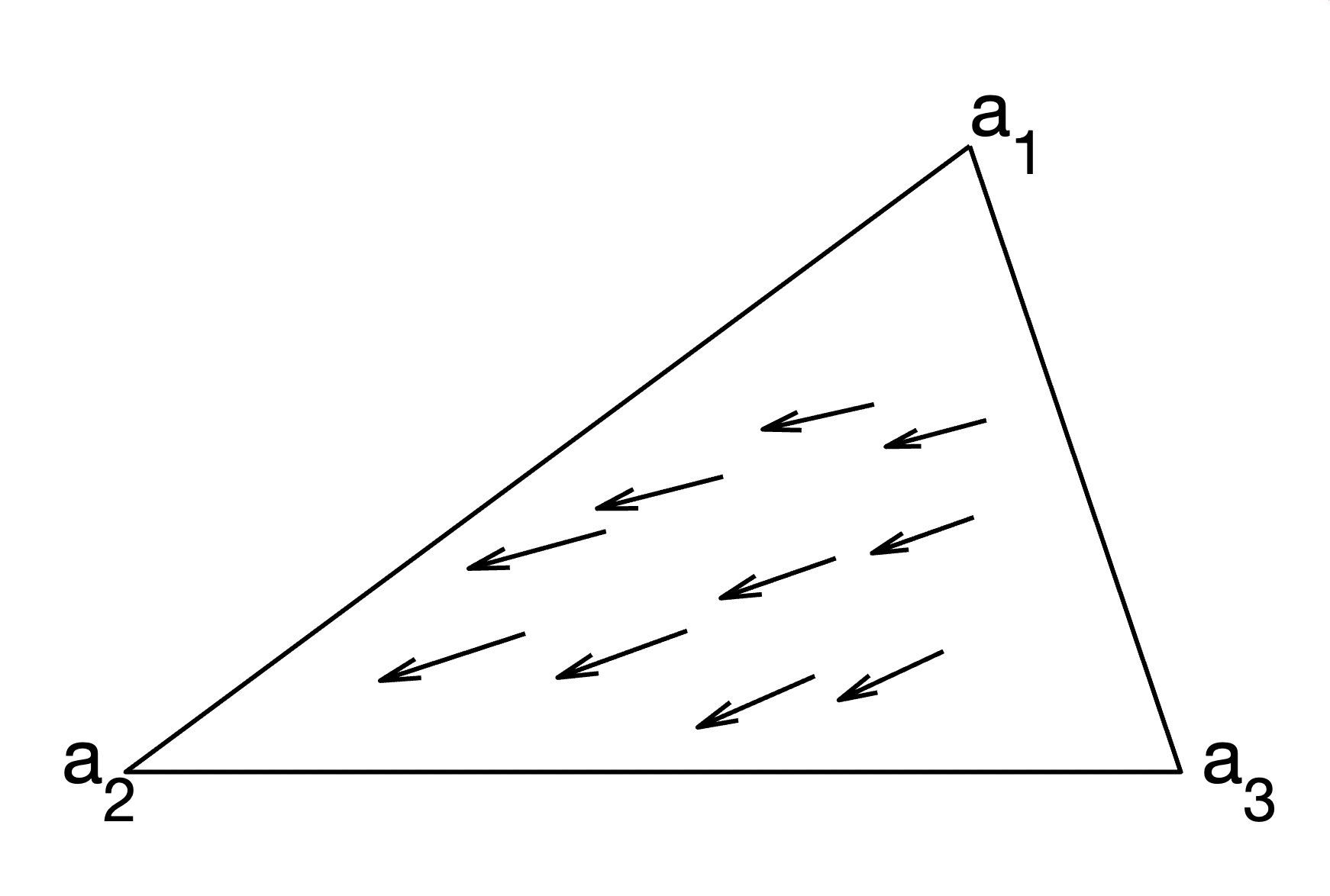}
\includegraphics[width=0.32\textwidth]{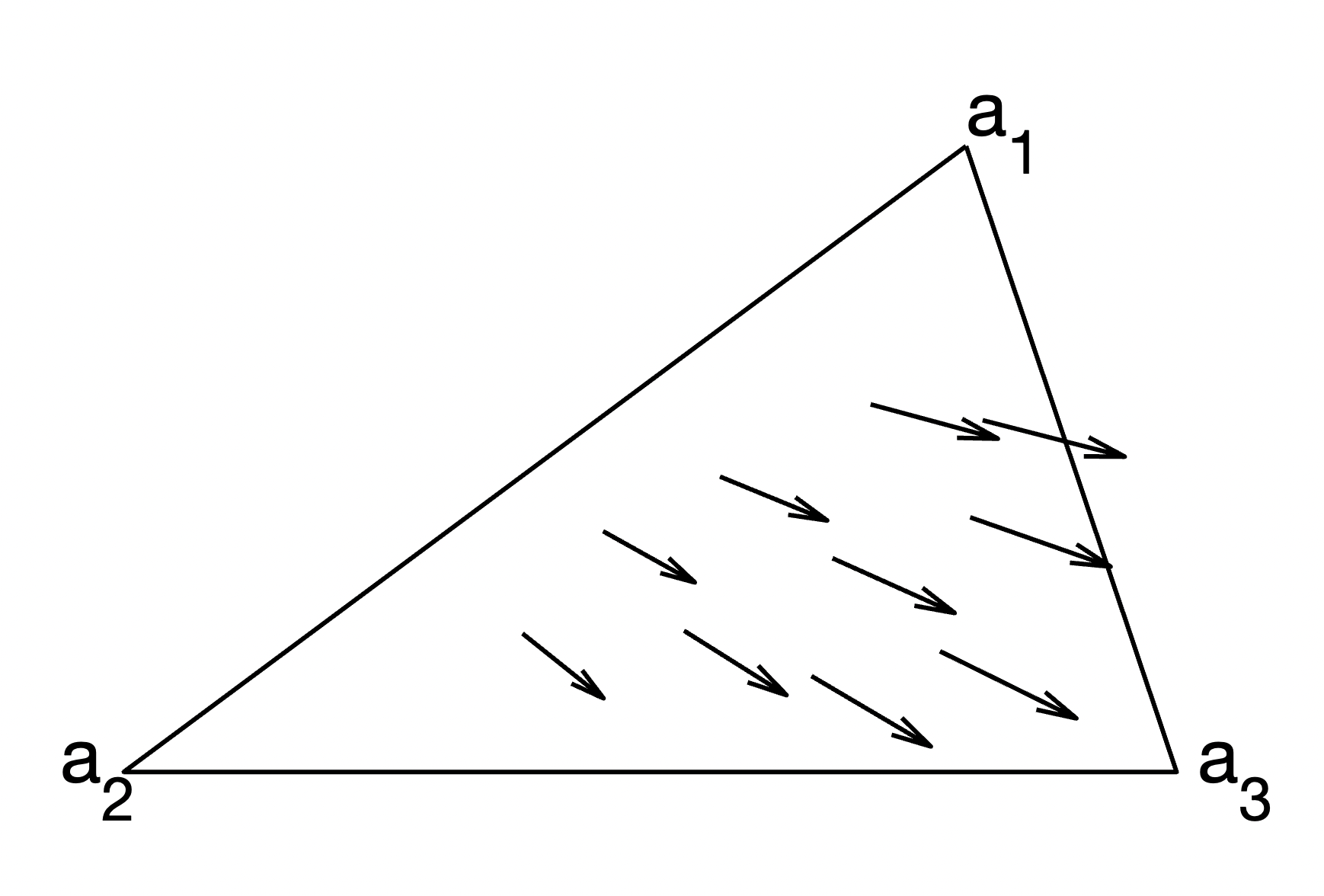}
\caption{Illustration of the three basis functions of $\mathbb{RT}(T)$ on a cell $T$. We pay particular attention to the sign of the outward normal component at  every edge.}\label{label:threebasis}
\end{figure}

Denote
\begin{equation}\label{eq:localbi}
\undertilde{b}_T^{a_i}:=\frac{1}{2S}(\undertilde{x}+\undertilde{a}_i-\undertilde{a}_j-\undertilde{a}_k),\ \ i=1,2,3,\ \left\{i,j,k\right\}=\left\{1,2,3\right\}.
\end{equation}
Then, $\left\{\undertilde{b}_T^{a_i},i=1,2,3\right\}$ form a basis of $\mathbb{RT}(T)$. Particularly, $\undertilde{b}_T^{a_i}\cdot\mathbf{n}_j|_{e_j}=(1-2\delta_{ij})/|e_j|$, and 
\begin{equation}\label{eq:dualityblambda}
(\undertilde{b}_T^{a_i},\nabla\lambda_j)_T+(\dv\undertilde{b}_T^{a_i},\lambda_j)_T=\delta_{ij},\ \ 1\leqslant i,j\leqslant 3. 
\end{equation}
The identities \eqref{eq:nabladivdual} confirm the existence of a basis of $\mathbb{RT}(T)$ that satisfies the dual relation \eqref{eq:dualityblambda}, and \eqref{eq:localbi} further gives the precise formulation of them. See Figure \ref{label:threebasis} for the illustrations and profiles of the local basis functions. Then 
$$
\displaystyle \mathbb{RT}(\mathcal{T}_h) = \bigoplus_{T\in\mathcal{T}_h}E_T^\Omega\mathbb{RT}(T)
=\bigoplus_{T\in\mathcal{T}_h}\bigoplus_{M\in \mathcal{X}_h\cap\partial T}{\rm span}\left\{E_T^\Omega\ub{}_T^M\right\}=\bigoplus_{M\in\mathcal{X}_h}\bigoplus_{\partial T\ni M}{\rm span}\left\{E_T^\Omega\ub{}_T^M\right\}.
$$

\subsubsection{Two types of basis functions in $\mathbb{RT}^{\rm nc}_{h0}$ and $\mathbb{RT}^{\rm nc}_h$}
\label{sec:rtabcrtabc0}

For $M\in\mathcal{X}_h$, denote by $\psi_{M}$ the basis function of $\mathbb{V}^1_h$ such that $\psi_M(M)=1$ and $\psi_M$ vanishes on other vertices. We can rewrite \eqref{eq:dualityblambda} to the lemma below. 
\begin{lemma}\label{lem:localorth}
For $M,M'\in\mathcal{X}_h$ and $T,T'\in\mathcal{T}_h$, such that $M\in\partial T$ and $M'\in\partial T'$, with
$$
\delta_{MM'}\ \mbox{denoting}\ \left\{\begin{array}{ll}1,&M=M'\\ 0,&M\neq M'\end{array}\right.\ \ \mbox{and}\ \  \delta_{TT'}\ \mbox{denoting}\ \left\{\begin{array}{ll}1,&T=T'\\ 0,&T\neq T'\end{array}\right.,
$$
it holds that
$$
(E_T^\Omega\ub{}_T^M,\nabla \psi_{M'})_{T'}+(\dv E_T^\Omega\ub{}_T^M,\psi_{M'})_{T'}=\delta_{MM'}\delta_{TT'}.
$$
\end{lemma}

Denote, for $M\in\mathcal{X}_h$,
\begin{equation}
\mathcal{B}_M:=\left\{\utau{}_h\in \bigoplus_{\partial T\ni M}{\rm span}\left\{E_T^\Omega\ub{}_T^M\right\}:\sum_{T\in\mathcal{T}_h}(\utau{}_h,\nabla \psi_M)_T+(\dv\utau{}_h,\psi_M)_T=0\right\},
\end{equation}
and
\begin{equation}
\mathcal{C}_M:=\bigoplus_{\partial T\ni M}{\rm span}\left\{E_T^\Omega\ub{}_T^M\right\}.
\end{equation}
Then $\mathcal{B}_M\subset\mathcal{C}_M$. We present the structures of $\mathbb{RT}^{\rm nc}_{h0}$ and $\mathbb{RT}^{\rm nc}_h$ in the lemma below.
\begin{lemma}\label{lem:basisrtabc}
\begin{enumerate}
\item If $M\neq N\in\mathcal{X}_h$, $\mathcal{C}_M\cap\mathcal{C}_N=\{0\}$;
\item 
$\displaystyle\mathbb{RT}^{\rm nc}_{h0}=\bigoplus_{M\in\mathcal{X}_h}\mathcal{B}_M;$
\item $\displaystyle \mathbb{RT}^{\rm nc}_h=\left[\bigoplus_{M\in\mathcal{X}_h^b}\mathcal{C}_M\right]\oplus\left[\bigoplus_{M\in\mathcal{X}_h^i}\mathcal{B}_M\right].$
\end{enumerate}
\end{lemma}
\begin{proof}
~~The first item follows directly by definition. 
For the second, by \eqref{eq:2drt0abc} and Lemma \ref{lem:localorth},
{
\begin{multline*}
\mathbb{RT}^{\rm nc}_{h0}=\left\{\utau{}_h\in \bigoplus_{M\in\mathcal{X}_h}\bigoplus_{\partial T\ni M}{\rm span}\left\{E_T^\Omega\ub{}_T^M\right\}: \sum_{T\in\mathcal{T}_h}(\utau{}_h,\nabla \psi_N)_T+(\dv\utau{}_h,\psi_N)_T=0,\ \forall\, N\in\mathcal{X}_h\right\}
\\
= \bigoplus_{M\in\mathcal{X}_h}\left\{\utau{}_h\in \bigoplus_{\partial T\ni M}{\rm span}\left\{E_T^\Omega\ub{}_T^M\right\}:\sum_{T\in\mathcal{T}_h,\partial T\ni M}(\utau{}_h,\nabla \psi_M)_T+(\dv\utau{}_h,\psi_M)_T=0\right\},
\end{multline*}
}
and the second item follows. For the third,
\begin{multline*}
\mathbb{RT}^{\rm nc}_h=\left\{\utau{}_h\in \bigoplus_{M\in\mathcal{X}_h}\bigoplus_{\partial T\ni M}{\rm span}\left\{E_T^\Omega\ub{}_T^M\right\}: \sum_{T\in\mathcal{T}_h}(\utau{}_h,\nabla \psi_N)_T+(\dv\utau{}_h,\psi_N)_T=0,\ \forall\, N\in\mathcal{X}_{h0}\right\}
\\
= \left\{\utau{}_h\in \bigoplus_{M\in\mathcal{X}_h^i}\bigoplus_{\partial T\ni M}{\rm span}\left\{E_T^\Omega\ub{}_T^M\right\}:\sum_{T\in\mathcal{T}_h}(\utau{}_h,\nabla \psi_N)_T+(\dv\utau{}_h,\psi_N)_T=0,\ \forall\, N\in\mathcal{X}_{h0}\right\} 
\\
\bigoplus \left\{\utau{}_h\in \bigoplus_{M\in\mathcal{X}_h^b}\bigoplus_{\partial T\ni M}{\rm span}\left\{E_T^\Omega\ub{}_T^M\right\}\right\}
\\
= \left[\bigoplus_{M\in\mathcal{X}_h^i}\left\{\utau{}_h\in \bigoplus_{\partial T\ni M}{\rm span}\left\{E_T^\Omega\ub{}_T^M\right\}:\sum_{T\in\mathcal{T}_h}(\utau{}_h,\nabla \psi_M)_T+(\dv\utau{}_h,\psi_M)_T=0\right\}\right] 
\\
\bigoplus  \left[\bigoplus_{M\in\mathcal{X}_h^b}\bigoplus_{\partial T\ni M}{\rm span}\left\{E_T^\Omega\ub{}_T^M\right\}\right].
\end{multline*}
The third item follows. This completes the proof. 
\end{proof}

\subsubsection{Profiles of $\mathcal{B}_M$ and $\mathcal{C}_M$}

\begin{lemma}
Given a vertex $M$ that is shared by $m$ triangles, $\dim(\mathcal{B}_M)=m-1$. A minimal support for a function in $\mathcal{B}_M$ is a combination of two cells.
\end{lemma}

\begin{proof}
~~The support of $\psi_M$ consists of $m$ triangles. Denote by $T_i$, $1\leqslant i\leqslant m$, the $m$ triangles that share $M$. The basis functions in $\mathcal{B}_M$ then take the form $\displaystyle\sum_{i=1}^m\gamma_i\undertilde{b}_{T_i}^M$, satisfying 
\begin{equation}\label{eq:basisBM}
\sum_{i=1}^m\left[(\gamma_i\undertilde{b}_{T_i}^M,\nabla(\psi_M|_{T_i}))_{T_i}+(\gamma_i\dv\undertilde{b}_{T_i}^M,\psi_M|_{T_i})_{T_i}\right]=0. 
\end{equation}
By \eqref{eq:dualityblambda}, this equation admits $(m-1)$ linearly independent solutions, and every corresponding function can be supported on two cells. Particularly, we assign the two cells to be adjacent. Figure \ref{fig:globalbasis} illustrates the profile of a basis function. 
\begin{figure}[htbp]
\begin{center}
\includegraphics[width=0.45\textwidth]{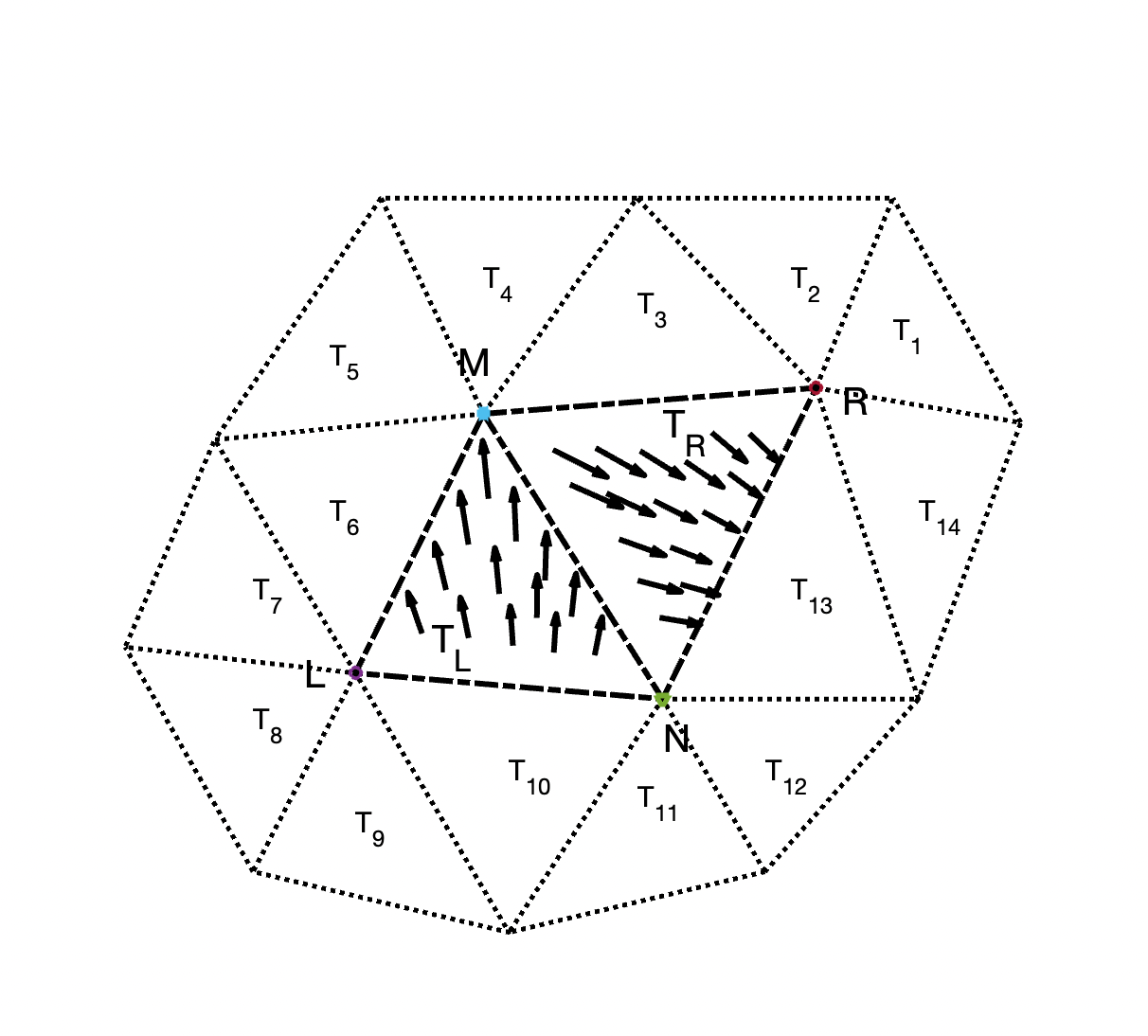}
\end{center}
\caption{Profile of a {\bf global} basis functions in $\mathcal{B}_M$, supported on two adjacent cells. }\label{fig:globalbasis}
\end{figure}

The function as illustrated in Figure \ref{fig:globalbasis}, denoted by $\utau$, is
$$
\utau=\ub{}_{T_L}^M\ \mbox{on}\ T_L,\ \ \utau=-\ub{}_{T_R}^M\ \mbox{on}\ T_R,\ \mbox{and}\ \utau=0,\ \mbox{elsewhere}.
$$
By \eqref{eq:dualityblambda}, on $T_L$, 
$\displaystyle(\utau,\nabla\psi_M)_{T_L}+(\dv\utau,\psi_M)_{T_L}=1$, $\displaystyle(\utau,\nabla\psi_L)_{T_L}+(\dv\utau,\psi_L)_{T_L}=0,$ and $\displaystyle(\utau,\nabla\psi_N)_{T_L}+(\dv\utau,\psi_N)_{T_L}=0$; on $T_R$, $\displaystyle(\utau,\nabla\psi_M)_{T_R}+(\dv\utau,\psi_M)_{T_R}=-1$, $\displaystyle(\utau,\nabla\psi_R)_{T_R}+(\dv\utau,\psi_R)_{T_R}=0$, and $\displaystyle(\utau,\nabla\psi_N)_{T_R}+(\dv\utau,\psi_N)_{T_R}=0.$ Then $\utau$ satisfies \eqref{eq:basisBM}. As $\utau$ vanishes on other cells, we can obtain $\displaystyle\sum_{T\in\mathcal{T}_h}(\utau,\nabla\psi)_T+(\dv\utau,\psi)_T=0$ for all $\psi\in \mathbb{V}^1_h$, thus $\utau\in\mathbb{RT}^{\rm nc}_{h0}$. 

\begin{figure}[htbp]
	\centering
	\includegraphics[width=\textwidth]{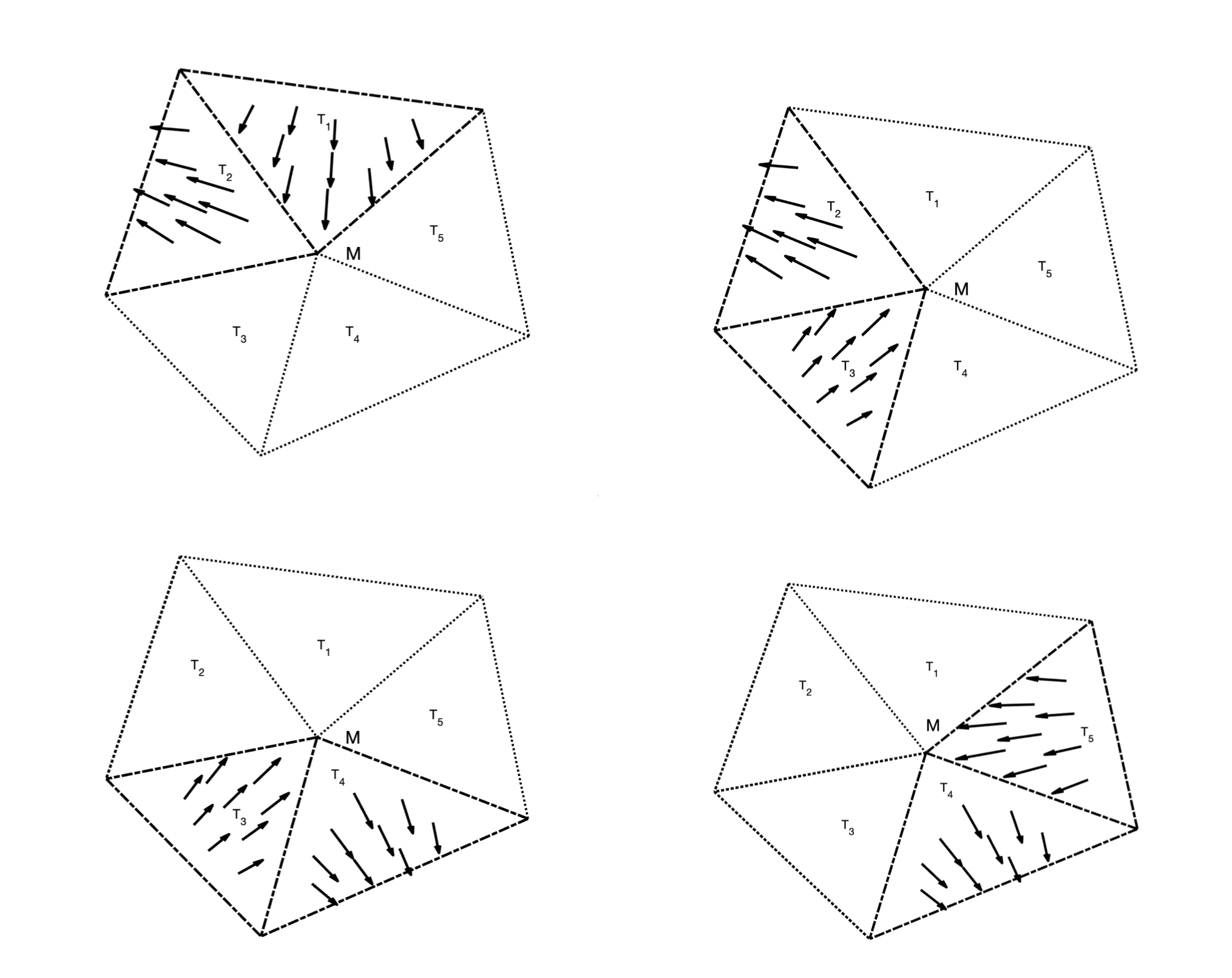}
	\caption{Profiles of linearly independent basis functions of $\mathcal{B}_M$, $M\in\mathcal{X}_h^i$.}
	\vspace{-0.2cm}
	\label{fig:basisM}
\end{figure}

\begin{figure}[htbp]
\includegraphics[width=0.48\textwidth]{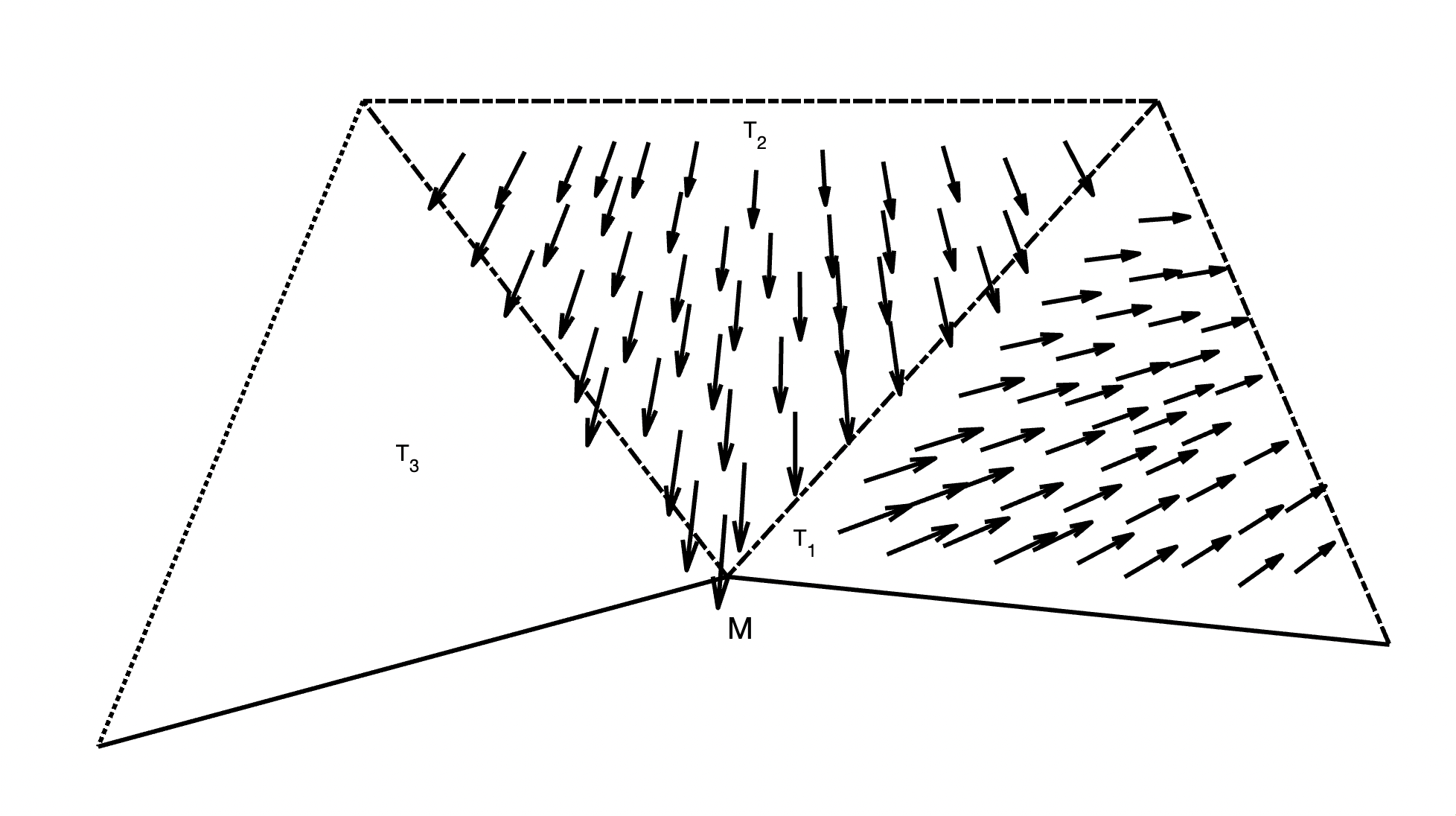}
\includegraphics[width=0.48\textwidth]{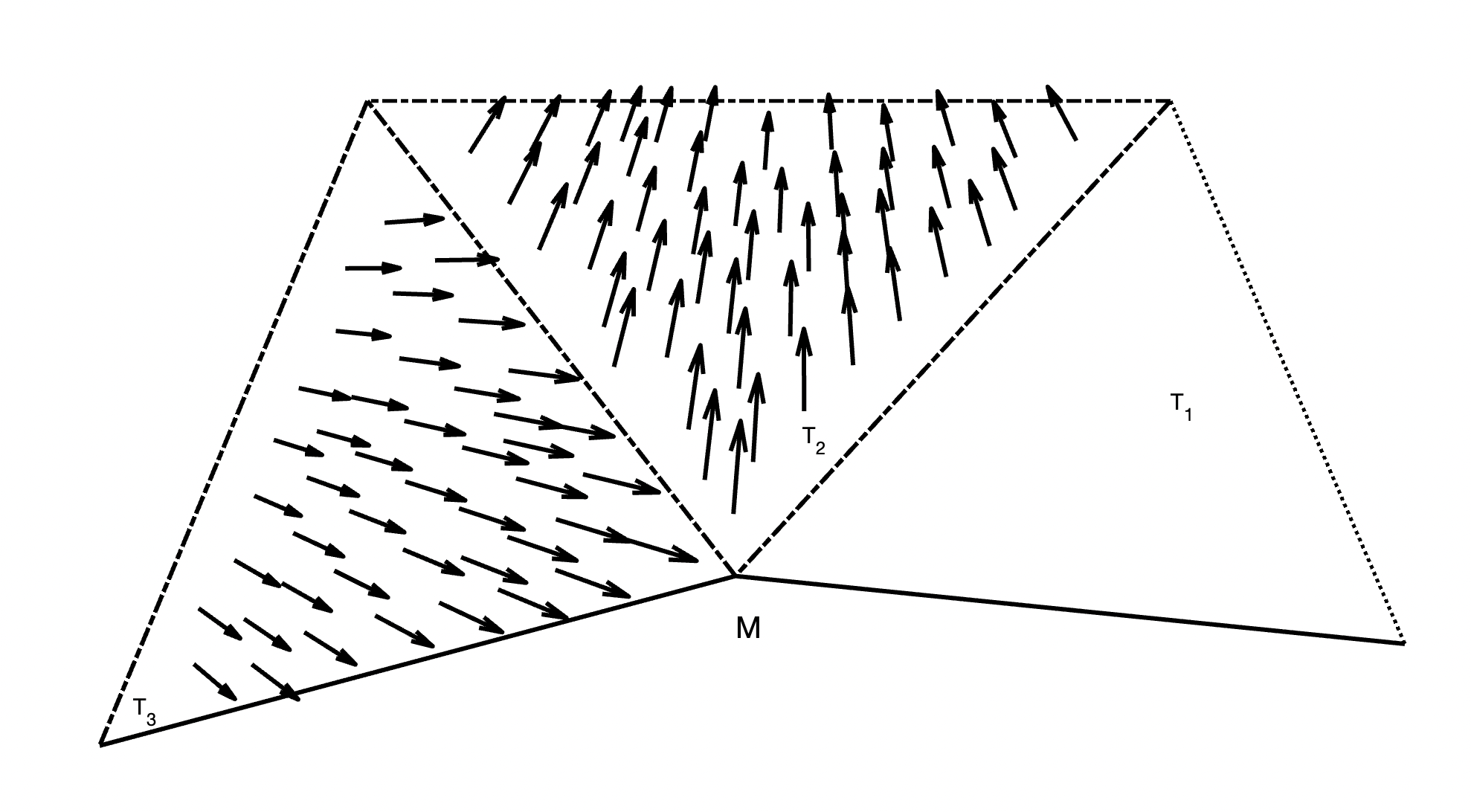}
\caption{Illustration of global basis functions of $\mathbb{RT}^{\rm nc}_{h0}$ based on a boundary vertex $M$}\label{fig:basisMbdy}
\end{figure}

According to the profile of Figure \ref{fig:globalbasis}, a set of linearly independent basis functions of $\mathcal{B}_M$ can be given in Figure \ref{fig:basisM}, where $M$ is an interior vertex, and in Figure \ref{fig:basisMbdy}, where $M$ is a boundary vertex. This completes the proof. 
\end{proof}
~
\begin{lemma}\label{lem:strccm}
Given a vertex $A$ that is shared by $m$ triangles, $\dim(\mathcal{C}_A)=m$. A minimal support for a function in $\mathcal{C}_A$ is one cell.
\end{lemma}
The proof of Lemma \ref{lem:strccm} is straightforward. We refer to Figure \ref{fig:bdybasissep} for an illustration. 

\begin{figure}[htbp]
\includegraphics[width=0.32\textwidth]{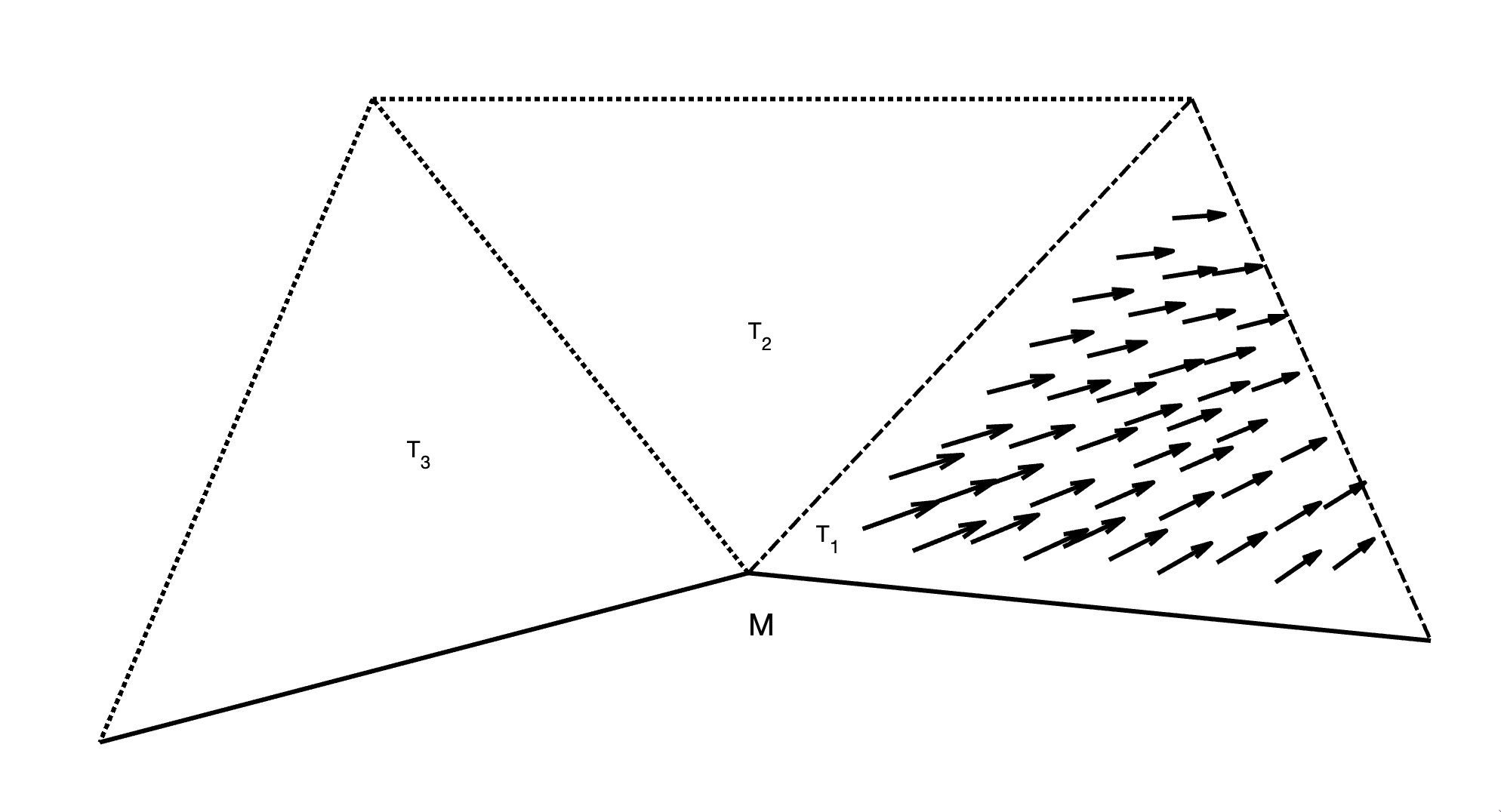}
\includegraphics[width=0.32\textwidth]{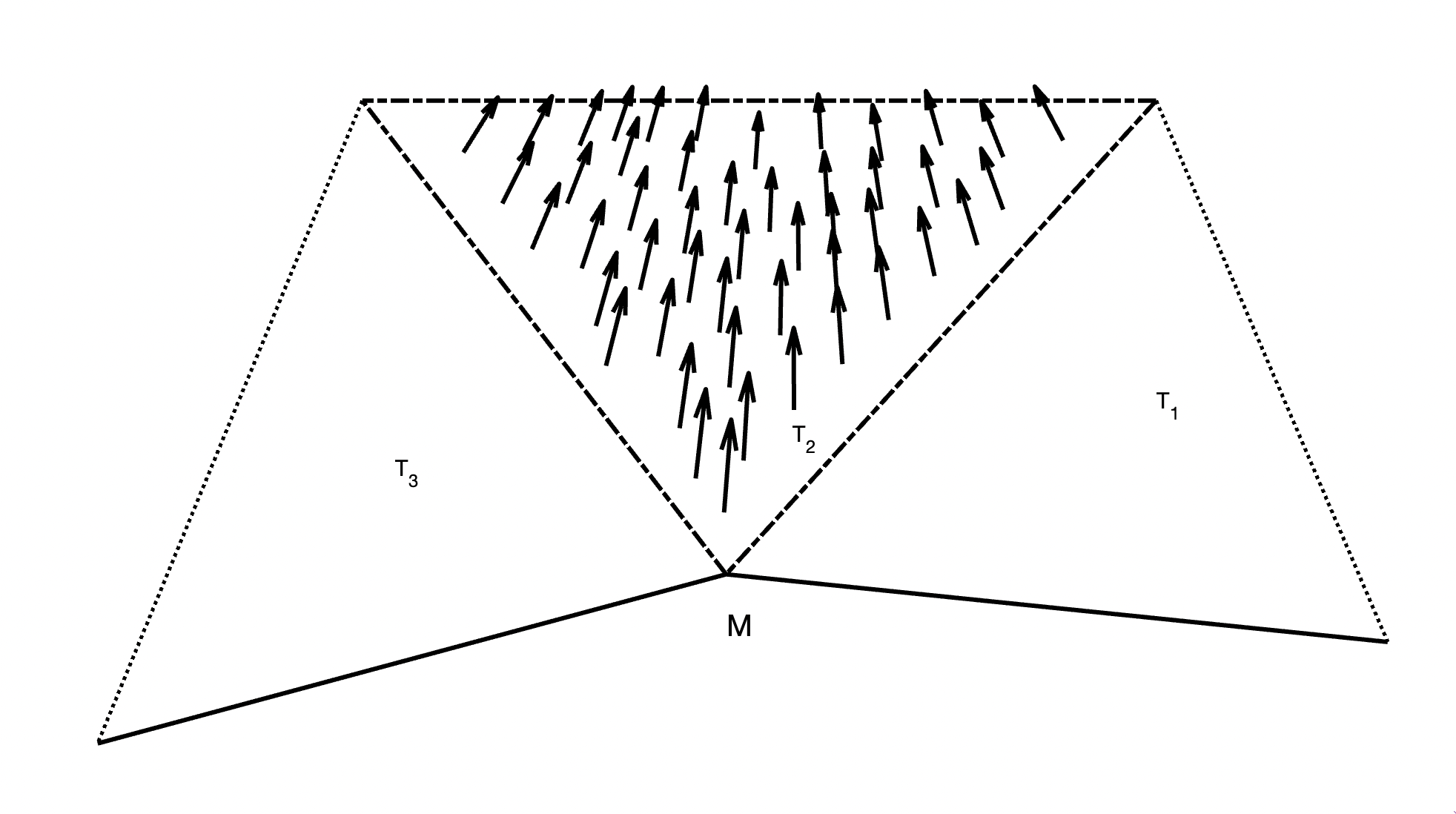}
\includegraphics[width=0.32\textwidth]{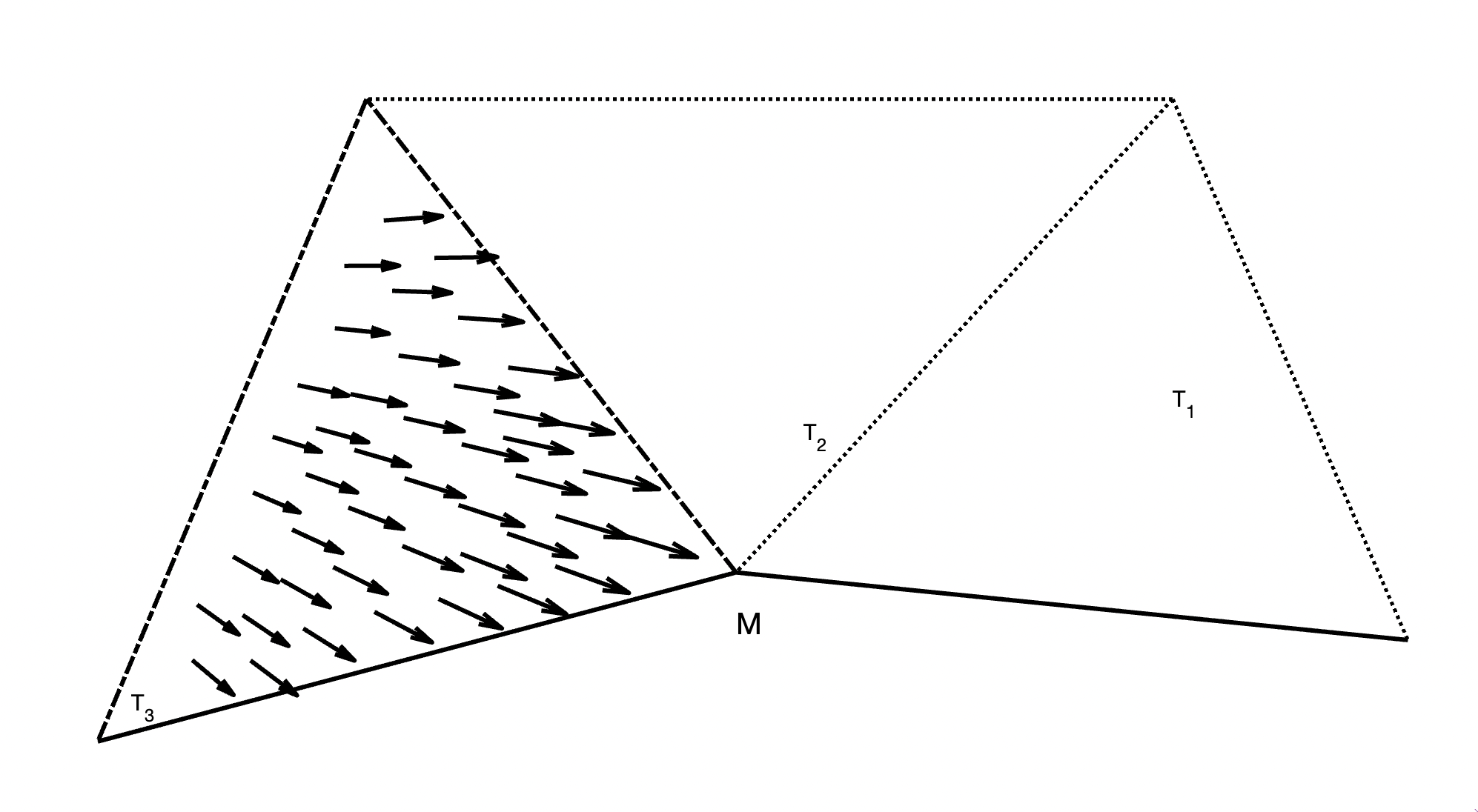}
\caption{The local basis functions associated with a boundary vertex $M$ can work as global basis functions of $\mathbb{RT}^{\rm nc}_h$.}\label{fig:bdybasissep}
\end{figure}

\begin{remark}
For $\mathbb{RT}^{\rm nc}_{h0}$, the total amount of the locally supported basis functions is
$$
\sum_{M\in\mathcal{X}_h}\left[\#\left\{T\in\mathcal{T}_h:\partial T\ni M\right\}-1\right]=3\#\left[T\in\mathcal{T}_h\right]-\#\left[M\in\mathcal{X}_h\right]=\dim(\mathbb{RT}^{\rm nc}_{h0}).
$$
For $\mathbb{RT}^{\rm nc}_h$, the total amount of the locally supported basis functions is
\begin{multline*}
\sum_{M\in\mathcal{X}_h^b}\left[\#\left\{T\in\mathcal{T}_h:\partial T\ni M\right\}\right]+\sum_{M\in\mathcal{X}_h^i}\left[\#\left\{T\in\mathcal{T}_h:\partial T\ni M\right\}-1\right]
\\
=3\#\left[T\in\mathcal{T}_h\right]-\#\left[M\in\mathcal{X}_h^i\right]=\dim(\mathbb{RT}^{\rm nc}_h).
\end{multline*}

In any case, $T$ is covered by the supports of no more than $\widetilde m+6$ basis functions, where $\widetilde m$ is the number of cells that has at least one vertex in common with $T$. The generation of a local stiffness matrix is a local operation, and the assembling of global stiffness matrices can be done by following the standard routine for finite elements of Ciarlet-type. 
\end{remark}

Based on the specific profiles of the basis functions, we conclude this subsection by rephrasing Lemma \ref{lem:basisrtabc} as the theorem below. 
\begin{theorem}\label{thm:localbasisrtabc}
The space $\mathbb{RT}^{\rm nc}_{h0}$ admits a set of linear independent basis functions, which are belonging to $\bigoplus_{M\in\mathcal{X}_h}\mathcal{B}_M$ and each supported on two adjacent triangles. 

The space $\mathbb{RT}^{\rm nc}_h$ admits a set of linear independent basis functions; they consist of two types of functions, Type I and Type II. The functions of Type I are belonging to $\bigoplus_{M\in\mathcal{X}_h^i}\mathcal{B}_M$ and each supported on two adjacent triangles, and the functions of Type II are belonging to $\bigoplus_{M\in\mathcal{X}_h^b}\mathcal{C}_M$ and each supported on one triangle.
\end{theorem}


%
%
\subsection{Approximation and stability}
\label{subsec:intrt}


%
\subsubsection{Locally-defined projective interpolator for $\Hdiv$}

Given a triangle $T$, define the cell-wise interpolator 
\begin{equation}\label{eq:irtlocal}
\mathbb{I}^{\rm RT}_T: \HdivT\to\mathbb{RT}(T)
\end{equation}
such that
\begin{equation}\label{eq:adpro}
(\mathbb{I}^{\rm RT}_T\utau,\nabla v)_T+(\dv\mathbb{I}^{\rm RT}_T\utau,v)_T=(\utau,\nabla v)_T+(\dv\utau,v)_T,\ \forall\,v\in P_1(T).  
\end{equation}
By \eqref{eq:dualityblambda}, $\displaystyle \mathbb{I}^{\rm RT}_T\utau=\sum_{i=1}^3 \left[(\utau,\nabla \lambda_i)_T+(\dv\utau,\lambda_i)_T\right]\ub{}_T^{a_i}$, and $\mathbb{I}^{\rm RT}_T \usigma=\usigma$ for $\usigma\in \mathbb{RT}(T)$. 

\begin{remark}
The Crouzeix-Raviart element interpolator $\mathbb{I}^{\rm CR}_T:H^1(T)\to P_1(T)$, defined such that $\int_e\mathbb{I}^{\rm CR}_Tv=\int_e v$, satisfies the condition
$
(\mathbb{I}^{\rm CR}_Tv,\dv\utau)_T+(\nabla \mathbb{I}^{\rm CR}_Tv,\utau)_T=(v,\dv\utau)_T+(\nabla v,\utau)_T,\ \forall\,\utau\in \mathbb{RT}(T).  
$
\end{remark}

On the triangulation $\mathcal{T}_h$, define the global interpolator by
\begin{equation}\label{eq:globalintp}
\mathbb{I}^{\rm RT}_h:\bigoplus_{T\in\mathcal{T}_h}E_T^\Omega H(\dv,T)\to \mathbb{RT}(\mathcal{T}_h),\ \ \ (\mathbb{I}^{\rm RT}_h\utau{}_h)_T=\mathbb{I}^{\rm RT}_T(\utau{}_h|_T),\ \forall\,T\in\mathcal{T}_h.
\end{equation}
\begin{lemma}\label{lem:intdivinrt}
$\displaystyle\R(\mathbb{I}^{\rm RT}_h, H(\dv,\Omega))\subset \mathbb{RT}^{\rm nc}_h$ and $\displaystyle \R(\mathbb{I}^{\rm RT}_h, H_0(\dv,\Omega))\subset \mathbb{RT}^{\rm nc}_{h0}$.
\end{lemma}
\begin{proof} 
~~Given $\usigma\in H(\dv,\Omega)$, $(\usigma,\nabla v_h)+(\dv\usigma,v_h)=0$ for any $v_h\in \mathbb{V}^1_{h0}$. Thus for any $v_h\in \mathbb{V}^1_{h0}$,
$\displaystyle \sum_{T\in\mathcal{T}_h}(\nabla v_h,\mathbb{I}^{\rm RT}_h\usigma)_T+(v_h,\dv\mathbb{I}^{\rm RT}_h\usigma)_T=\sum_{T\in\mathcal{T}_h}(\nabla v_h,\usigma)_T+(v_h,\dv\usigma)_T=0$. Namely $\mathbb{I}^{\rm RT}_h\usigma\in\mathbb{RT}^{\rm nc}_h$, and thus $\R(\mathbb{I}^{\rm RT}_h, H(\dv,\Omega))\subset \mathbb{RT}^{\rm nc}_h$. Similarly $\displaystyle \R(\mathbb{I}^{\rm RT}_h, H_0(\dv,\Omega))\subset \mathbb{RT}^{\rm nc}_{h0}$. This completes the proof. 
\end{proof}
\begin{remark}
Different from most existing interpolators, $\mathbb{I}^{\rm RT}_h\usigma$ is not defined in the form of $\sum l_i(\usigma)\utau{}_i$, where $\utau{}_i$ is each a global basis function of $\mathbb{RT}^{\rm nc}_h$, and $l_i$ is each a functional on $\usigma$. Indeed, according to theory of \cite{Zeng.H;Zhang.C;Zhang.S2023existence}, as the global basis functions of $\mathbb{RT}^{\rm nc}_h$ may not be locally linearly independent, interpolator defined as $\sum l_i(\usigma)\utau{}_i$ with $l_i$ depends on the local information of $\usigma$ cannot be projective. 
\end{remark}

%
\subsubsection{Approximation and stability}

\begin{lemma}\label{lem:intdiv}
With a constant $C$ depending on the shape regularity of $T$,
\begin{enumerate}[(a)]
\item stabilities:
$$\|\dv \mathbb{I}^{\rm RT}_T\usigma\|_{0,T}\leqslant \|\dv\usigma\|_{0,T},\ \ \mbox{and}\quad \displaystyle \|\mathbb{I}^{\rm RT}_T\usigma\|_{\dv,T}\leqslant C\|\usigma\|_{\dv,T};$$ 
\item optimal approximation:
$$
\displaystyle\|\dv(\usigma-\mathbb{I}^{\rm RT}_T \usigma)\|_{0,T}=\inf_{\utau\in \mathbb{RT}(T)}\|\dv(\usigma-\utau)\|_{0,T},\ \ \
\mbox{and}\quad\displaystyle\|\usigma-\mathbb{I}^{\rm RT}_T \usigma\|_{\dv,T}\leqslant C \inf_{\utau\in \mathbb{RT}(T)}\|\usigma-\utau\|_{\dv,T}. 
$$
\end{enumerate}
\end{lemma}
\begin{proof}
~~Evidently, $\dv\mathbb{I}^{\rm RT}_T\usigma$ is the $L^2(T)$ projection of $\dv\usigma$ onto piecewise constant space; therefore, $\|\dv\mathbb{I}^{\rm RT}_T\usigma\|_{0,T}\leqslant \|\dv\usigma\|_{0,\Omega}$. Now we use $\mathbf{P}^0_T$ for the $L^2(T)$ projection to constant, and $\undertilde{\mathbf{P}}{}^0_T:=(\mathbf{P}^0_T)^2$. Then for any $\utau\in \mathbb{RT}(T)$, we have by Poincar\'e inequality, $\|\utau-\undertilde{\mathbf{P}}{}^0_T\utau\|_{0,T}\leqslant Ch_T\|\nabla \utau\|_{0,T}=Ch_T/\sqrt{2}\|\dv \utau\|_{0,T}.$ Meanwhile, for any $v\in P_1(T)$, $\|v-\mathbf{P}^0_T v\|_{0,T}\leqslant Ch_T\|\nabla v\|_{0,T}$. For any $v\in P_1(T)$, $\left(\mathbb{I}^{\rm RT}_T\usigma,\nabla v\right)_T+(\dv\mathbb{I}^{\rm RT}_T\usigma, v)_T=(\usigma,\nabla v)_T+(\dv\usigma,v)_T$, and $(\dv\mathbb{I}^{\rm RT}_T\usigma, \mathbf{P}^0_Tv)_T=(\dv\usigma,\mathbf{P}^0_Tv)_T$. Therefore, 
$\left(\undertilde{\mathbf{P}}{}^0_T\mathbb{I}^{\rm RT}_T\usigma,\nabla v\right)_T=(\usigma,\nabla (v-\mathbf{P}^0_Tv))_T+(\dv\usigma,v-\mathbf{P}^0_Tv)_T.$ It follows that $\|\undertilde{\mathbf{P}}{}^0_T\mathbb{I}^{\rm RT}_T\usigma\|_{0,T}\leqslant C(\|\usigma\|_{0,T}+h_T\|\dv\usigma\|_{0,T}).$ Further $\|\mathbb{I}^{\rm RT}_T\usigma\|_{0,T}\leqslant C(\|\usigma\|_{0,T}+h_T\|\dv\usigma\|_{0,T}),$ and $\|\mathbb{I}^{\rm RT}_T\usigma\|_{\dv,T}\leqslant C\|\usigma\|_{\dv,T}.$

The optimal approximation follows then from the stability and by the standard procedure. 
\end{proof}

Moreover, as the global interpolator is defined completely piecewise, global stabilities hold and
\begin{equation}\label{eq:irtoptgl}
\|\usigma-\mathbb{I}^{\rm RT}_h\usigma\|_{\dv_h}\leqslant C\inf_{\utau{}_h\in \mathbb{RT}(\mathcal{T}_h)}\|\usigma-\utau{}_h\|_{\dv_h},
\end{equation}
where $C$ depends on the regularity of the triangulation only.

Further, the Poincar\'e inequalities hold for $\mathbb{RT}^{\rm nc}_h$ and $\mathbb{RT}^{\rm nc}_{h0}$.

\begin{lemma}\label{lem:stabrtabc}
Given $\utau{}_h\in \mathbb{RT}^{\rm nc}_h$, there is $\usigma{}_h\in \mathbb{RT}^{\rm nc}_h$, such that 
$\dv_h\usigma{}_h=\dv_h\utau{}_h$, and $\|\usigma{}_h\|_{0,\Omega}\leqslant C\|\dv_h\utau{}_h\|_{0,\Omega}$.
\end{lemma}
\begin{proof}
~~Note that $\dv_h\utau{}_h$ is piecewise constant, and there exists a $\utau\in H(\dv,\Omega)$, such that $\dv\utau=\dv_h\utau{}_h$, and $\|\utau\|_{\dv,\Omega}\leqslant C\|\dv\utau\|_{0,\Omega}$. Set $\usigma{}_h=\mathbb{I}^{\rm RT}_h\utau$, then $\dv_h\usigma{}_h=\dv\utau$, and $\|\usigma{}_h\|_{\dv_h}\leqslant C\|\utau\|_{\dv}\leqslant C\|\dv_h\utau{}_h\|_{0,\Omega}$. This completes the proof. 
\end{proof}

\begin{remark}
Evidently, $\mathbb{RT}^{\rm nc}_h\supset \uV^{\rm RT}_h$, and thus the approximation and stability properties of $\mathbb{RT}^{\rm nc}_h$ follow. Though, we present direct proofs of them by the aid of the interpolator. In some sense, both the two properties established here are optimal.  
\end{remark}



%
%
\subsection{Discretization of the variational problems}

\subsubsection{Discretization of the $H(\dv)$ elliptic problem} 

We consider the problem: given $\uf\in\uL^2(\Omega)$, find $\usigma\in H(\dv,\Omega)$, such that 
\begin{equation}\label{eq:evpd}
(\dv\usigma,\dv\utau)+(\usigma,\utau)=(\uf,\utau),\ \ \forall\,\utau\in H(\dv,\Omega). 
\end{equation}
It follows that $\dv\usigma\in H^1_0(\Omega)$, and $\uf=-\nabla\dv\usigma+\usigma$.

We here consider the discretization of \eqref{eq:evpd}: to find $\usigma{}_h\in \mathbb{RT}^{\rm nc}_h$, such that 
\begin{equation}\label{eq:dispro}
(\dv_h\usigma{}_h,\dv_h\utau{}_h)+(\usigma{}_h,\utau{}_h)=(\uf,\utau{}_h),\ \ \forall\,\utau{}_h\in \mathbb{RT}^{\rm nc}_h. 
\end{equation}

Immediately \eqref{eq:evpd} and \eqref{eq:dispro} are well-posed. Denote $\|\utau{}_h\|_{\dv_h}:=(\|\utau{}_h\|_0^2+\|\dv_h\utau{}_h\|_0^2)^{1/2}$. 

\begin{theorem}\label{thm:basicest}
Let $\usigma$ and $\usigma{}_h$ be the solutions of \eqref{eq:evpd} and \eqref{eq:dispro}, respectively. Then
\begin{equation}\label{eq:basicest}
\|\usigma-\usigma_h\|_{\dv_h}\leqslant 2\inf_{\utau{}_h\in \mathbb{RT}^{\rm nc}_h}\|\usigma-\utau{}_h\|_{\dv_h}+\inf_{v_h\in\mathbb{V}^1_{h0}}\|\dv\usigma-v_h\|_{1,\Omega}. 
\end{equation}
\end{theorem}

\begin{proof}
~~By Strang's lemma (cf. \cite{Ciarlet.P1978book}),
$$
\|\usigma-\usigma_h\|_{\dv_h}\leqslant 2\inf_{\utau{}_h\in \mathbb{RT}^{\rm nc}_h}\|\usigma-\utau{}_h\|_{\dv_h}+\sup_{\utau{}_h\in \mathbb{RT}^{\rm nc}_h}\frac{(\dv\usigma,\dv_h\utau{}_h)+(\nabla\dv\usigma,\utau{}_h)}{\|\utau{}_h\|_{\dv_h}}.
$$
For any $v_h\in \mathbb{V}^1_{h0}$,
$$(\dv\usigma,\dv_h\utau{}_h)+(\nabla\dv\usigma,\utau{}_h)
=(\dv\usigma-v_h,\dv_h\utau{}_h)+(\nabla(\dv\usigma-v_h),\utau{}_h)\leqslant \|\dv\usigma-v_h\|_{1,\Omega}\|\utau{}_h\|_{\dv_h}.$$
Then \eqref{eq:basicest} follows.
\end{proof}

By the abstract estimation, the precise convergence order can be figured out with respect to the assumption on the regularity of the solution.

\subsubsection{Discretization of the Darcy problem} 
We consider the problem: given $f\in L^2(\Omega)$, find $(u,\usigma)\in L^2(\Omega)\times H(\dv,\Omega)$, such that 
\begin{equation}\label{eq:Poissondual}
\left\{
\begin{array}{llll}
(\usigma,\utau)&+(u,\dv\utau)&=0&\forall\,\utau\in \HdivOmega,
\\
(\dv\usigma,v)&&=(f,v)&\forall\,v\in L^2(\Omega).
\end{array}
\right.
\end{equation}

The discretization is to find $(u_h,\usigma{}_h)\in \mathcal{P}_0(\mathcal{T}_h)\times \mathbb{RT}^{\rm nc}_h$, such that 
\begin{equation}\label{eq:Poissondualdis}
\left\{
\begin{array}{llll}
(\usigma{}_h,\utau{}_h)&+(u_h,\dv_h\utau{}_h)&=0&\forall\,\utau{}_h\in \mathbb{RT}^{\rm nc}_h,
\\
(\dv_h\usigma{}_h,v_h)&&=(f,v_h)&\forall\,v_h\in \mathcal{P}_0(\mathcal{T}_h).
\end{array}
\right.
\end{equation}

Here $\mathcal{P}_0(\mathcal{T}_h)$ is the space of piecewise constant functions. Evidently, \eqref{eq:Poissondualdis} is well-posed. 

\begin{theorem}
Let $(u,\usigma)$ and $(u_h,\usigma{}_h)$ be the solutions of \eqref{eq:Poissondual} and \eqref{eq:Poissondualdis}, respectively. Then
$$
\|u-u_h\|_{0,\Omega}+\|\usigma-\usigma{}_h\|_{\dv_h}
\leqslant C\left[\inf_{\substack{v_h\in \mathcal{P}_0(\mathcal{T}_h) \\ \utau{}_h\in \mathbb{RT}^{\rm nc}_h}}(\|u-v_h\|_{0,\Omega}+\|\usigma-\utau{}_h\|_{\dv_h})+\inf_{s_h\in \mathbb{V}^1_{h0}}\|u-s_h\|_{1,\Omega}\right].
$$
\end{theorem}

\begin{proof}
~~By the Strang lemma for saddle point problem (cf., e.g., \cite[Proposition5.5.6]{Boffi.D;Brezzi.F;Fortin.M2013}), 
$$
\|u-u_h\|_{0,\Omega}+\|\usigma-\usigma{}_h\|_{\dv_h}
\leqslant C\left[\inf_{\substack{v_h\in \mathcal{P}_0(\mathcal{T}_h) \\ \utau{}_h\in \mathbb{RT}^{\rm nc}_h}}(\|u-v_h\|_{0,\Omega}+\|\usigma-\utau{}_h\|_{\dv_h})+\sup_{\utau{}_h\in \mathbb{RT}^{\rm nc}_h}\frac{(\usigma,\utau{}_h)+(u,\dv_h\utau{}_h)}{\|\utau{}_h\|_{\dv_h}}\right].
$$
Note that $\usigma=\nabla u$, and we have, for any $s_h\in \mathbb{V}^1_{h0}$, 
$$
(\usigma,\utau{}_h)+(u,\dv_h\utau{}_h)=(\nabla u-\nabla s_h,\utau{}_h)+(u-s_h,\dv_h\utau{}_h)\leqslant \|u-s_h\|_{1,\Omega}\|\utau{}_h\|_{\dv_h}.
$$
It follows then
$$
\|u-u_h\|_{0,\Omega}+\|\usigma-\usigma{}_h\|_{\dv_h}
\leqslant C\left[\inf_{\substack{v_h\in \mathcal{P}_0(\mathcal{T}_h) \\ \utau{}_h\in \mathbb{RT}^{\rm nc}_h}}(\|u-v_h\|_{0,\Omega}+\|\usigma-\utau{}_h\|_{\dv_h})+\inf_{s_h\in \mathbb{V}^1_{h0}}\|u-s_h\|_{1,\Omega}\right].
$$
This completes the proof. 
\end{proof}


%
%
\subsection{Numerical experiments}
We show the implementability of $\mathbb{RT}^{\rm nc}_h$ and its difference from the classical Raviart-Thomas element by two series of experiments. 

\subsubsection{Implementability of the space $\mathbb{RT}^{\rm nc}_h$}

Firstly, we use $\mathbb{RT}^{\rm nc}_h$ to solve numerically the boundary value problems \eqref{eq:evpd} and \eqref{eq:Poissondual}. We use the unit square $(0,1)^2$ as the computation domain, and we choose properly the source terms, such that 
\begin{itemize}
\item for \eqref{eq:evpd}, the exact solution is 
$$
\usigma=(-2\cos(\pi x)\sin(\pi y),\sin(\pi x)\cos(\pi y))^\top;
$$
\item for \eqref{eq:Poissondual}, the exact solution is 
$$
u=\sin(\pi x)\sin(\pi y),\ \  \mbox{and}\ \ \usigma=\nabla u.
$$
\end{itemize}
We construct two series of triangulations, being crisscross (cf. Figure \ref{fig:reguandirregutri}, left) and irregular (cf. Figure \ref{fig:reguandirregutri}, right), respectively. The computational results are recorded in Figures \ref{fig:convprimal} and \ref{fig:convmix}. 

\begin{figure}[htbp]
\includegraphics[width=0.45\textwidth]{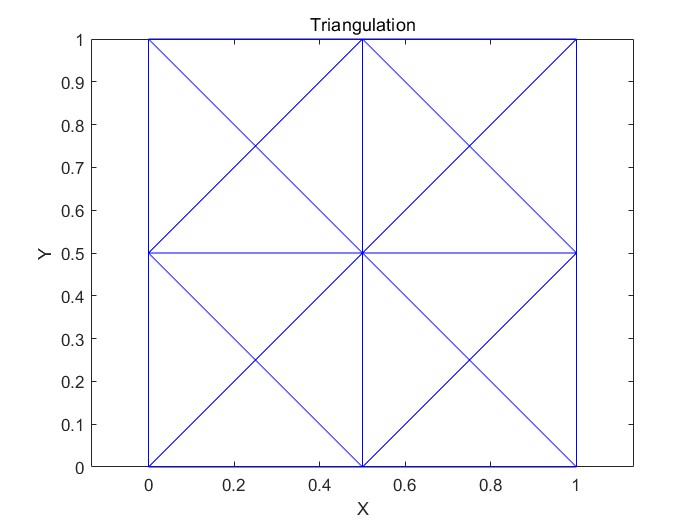}
\includegraphics[width=0.45\textwidth]{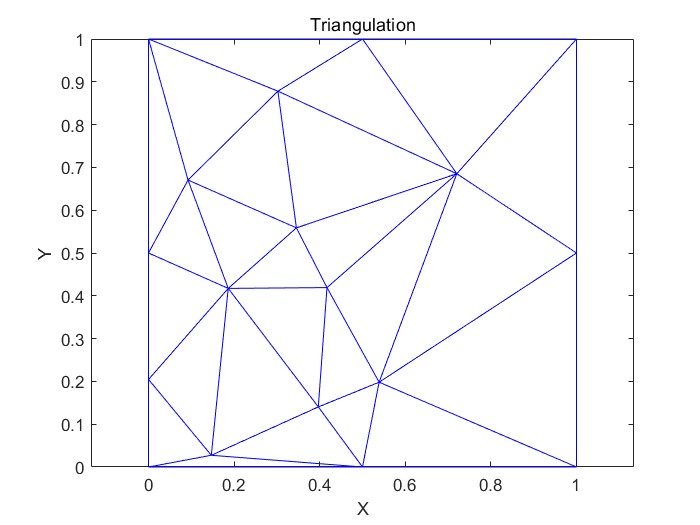}
\caption{The initial triangulation of two series of triangulations. Left: crisscross; right: irregular.}\label{fig:reguandirregutri}
\end{figure}

\begin{figure}[htbp]
\includegraphics[width=0.45\textwidth]{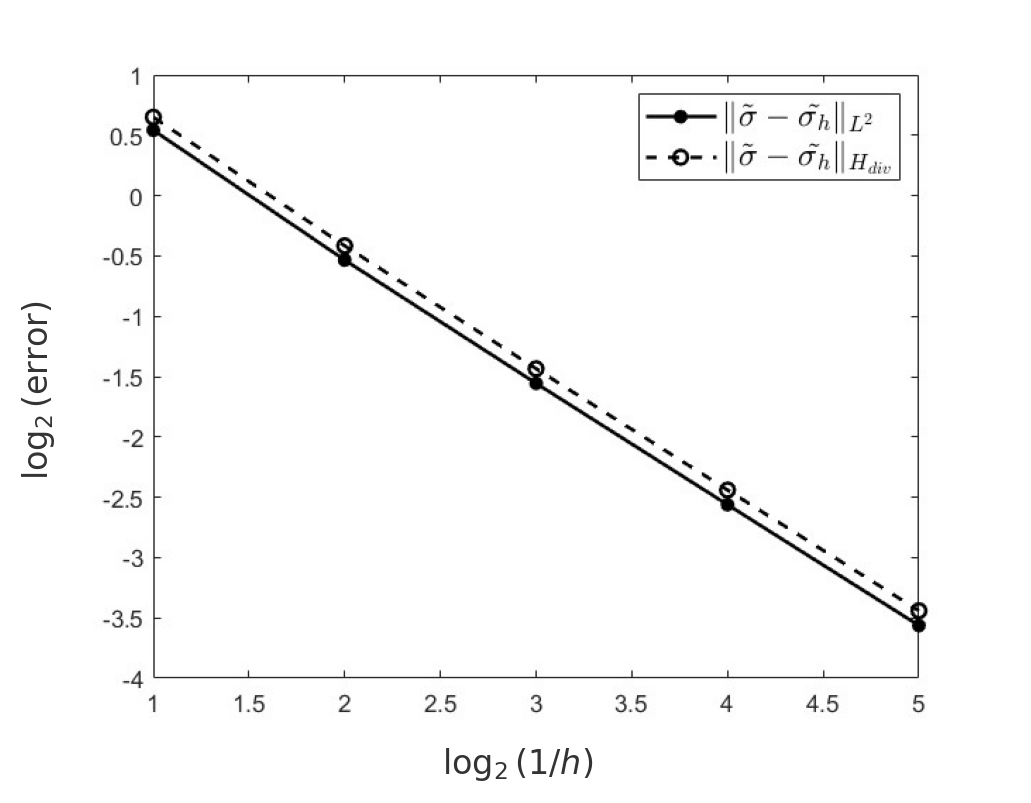}
\includegraphics[width=0.45\textwidth]{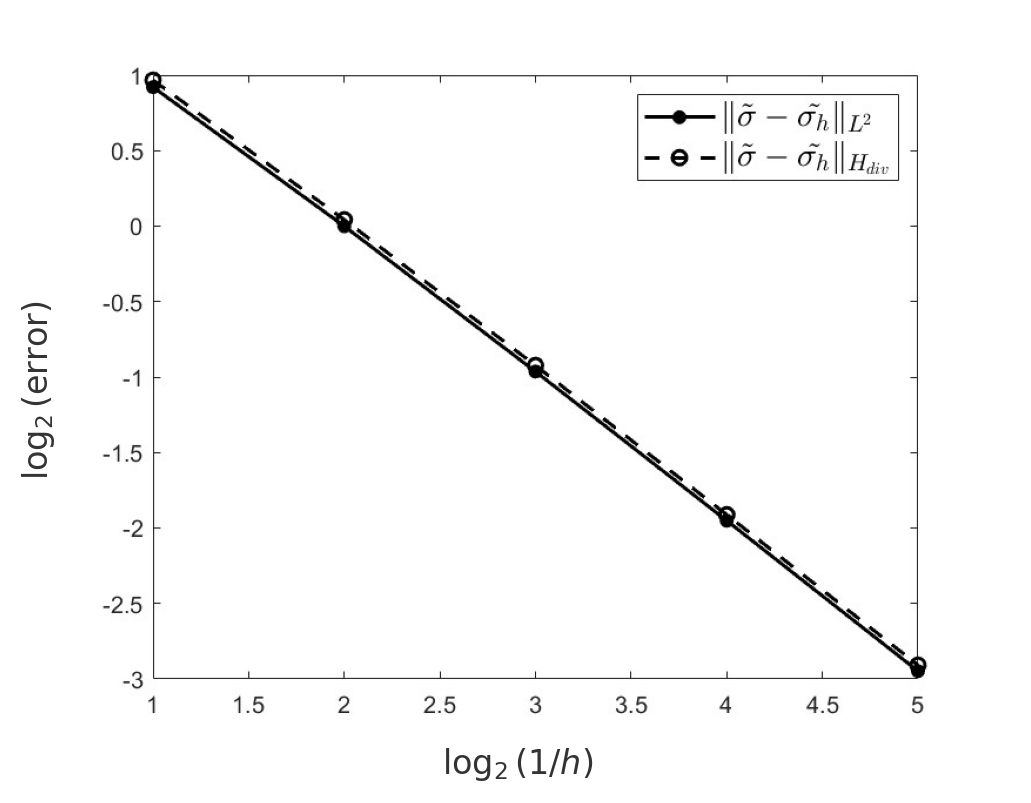}
\caption{Convergence process for \eqref{eq:evpd}. Left: on crisscross triangulations; right: on irregular triangulations.}\label{fig:convprimal}
\end{figure}

\begin{figure}[htbp]
\includegraphics[width=0.45\textwidth]{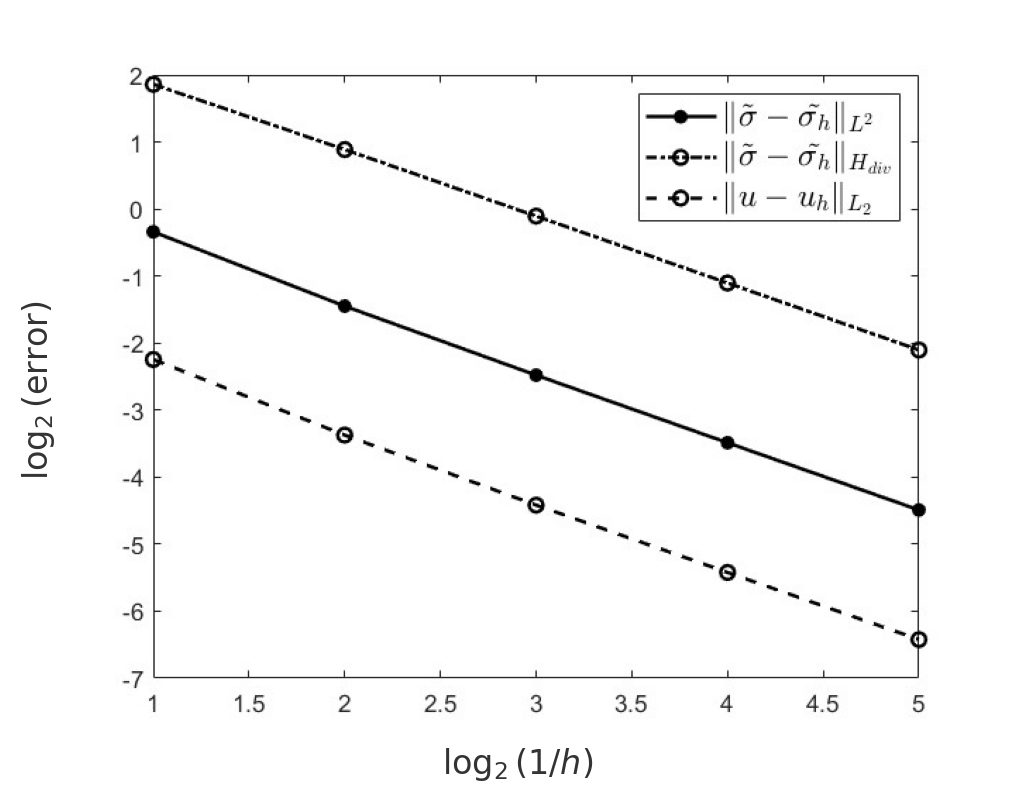}
\includegraphics[width=0.45\textwidth]{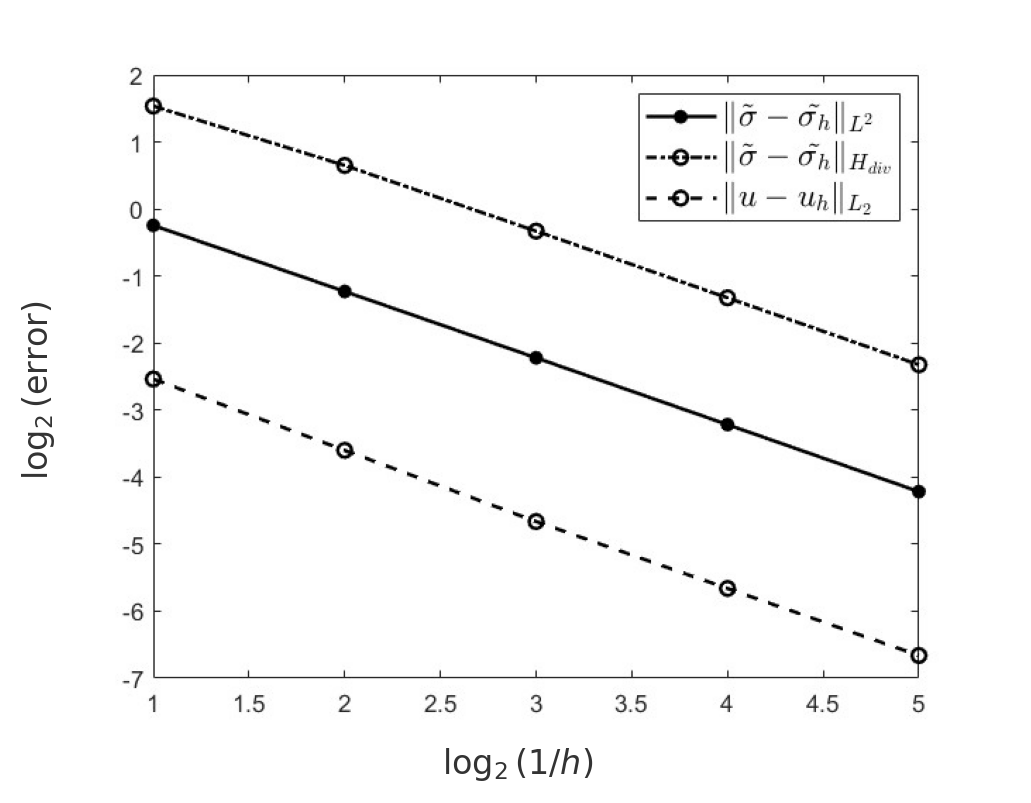}
\caption{Convergence process for \eqref{eq:Poissondual}. Left: on crisscross triangulations; right: on irregular triangulations.}\label{fig:convmix}
\end{figure}

On the two series of triangulations, we also use $\mathbb{RT}^{\rm nc}_h\times \mathcal{P}_0(\mathcal{T}_h)$ to solve the eigenvalue problem of \eqref{eq:Poissondual}, which is to find $(\lambda, u,\usigma)\in\mathbb{R}\times L^2(\Omega)\times H(\dv,\Omega)$, such that 
\begin{equation}\label{eq:Poissondualevp}
\left\{
\begin{array}{llll}
(\usigma,\utau)&+(u,\dv\utau)&=0&\forall\,\utau\in \HdivOmega,
\\
(\dv\usigma,v)&&=\lambda\ \pi^2\ (u,v)&\forall\,v\in L^2(\Omega).
\end{array}
\right.
\end{equation}
Note that on unit square, the eigenvalues of \eqref{eq:Poissondualevp} take the values $m^2+n^2$, $m,n\in\mathbb{N}^+$. Here we separate the effect of $\pi^2$ so that the results are easy to read. The respective eigenvalue problems of \eqref{eq:evpd} and \eqref{eq:Poissondual} are essentially equivalent to each other. The 10 smallest computed eigenvalues of \eqref{eq:Poissondualevp} on each series of grids are recorded in Tables \ref{tab:nrtevcc} and \ref{tab:nrtevunstr}. In the tables, we use ``L" to denote the level of each grid, and use $\searrow$/$\nearrow$ to denote the decreasing/increasing trend of the computed eigenvalues as the grids are refined and refined. The computed eigenvalues converge to the exact eigenvalues nicely. Moreover, it can be seen that the $\mathbb{RT}^{\rm nc}_h$ scheme for \eqref{eq:Poissondualevp} provides upper bounds to the exact eigenvalues. 
{\small
\begin{table}[!ht]
    \small
    \centering
    \begin{tabular}{|c|c|c|c|c|c|c|c|c|c|c|}
    \hline
    L&$\lambda_h^1$&$\lambda_h^2$&$\lambda_h^3$&$\lambda_h^4$&$\lambda_h^5$&$\lambda_h^6$&$\lambda_h^7$&$\lambda_h^8$&$\lambda_h^9$&$\lambda_h^{10}$\\
    \hline
        1 & 2.619  & 9.727  & 9.727  & 9.727  & 19.123  & 29.181  & 29.181  & 29.181  & 29.181  & 29.181  \\ \hline
        2 & 2.128  & 5.982  & 5.982  & 10.477  & 14.547  & 14.547  & 20.650  & 20.650  & 32.039  & 38.907  \\ \hline
        3 & 2.031  & 5.223  & 5.223  & 8.511  & 11.009  & 11.009  & 14.480  & 14.480  & 20.137  & 20.137  \\ \hline
        4 & 2.008  & 5.055  & 5.055  & 8.122  & 10.242  & 10.242  & 13.345  & 13.345  & 17.739  & 17.739  \\ \hline
        5 & 2.002  & 5.014  & 5.014  & 8.030  & 10.060  & 10.060  & 13.085  & 13.085  & 17.182  & 17.182  \\ \hline
         & $\searrow$ & $\searrow$ & $\searrow$ & $\searrow$ & $\searrow$ & $\searrow$ & $\searrow$ & $\searrow$ & $\searrow$ & $\searrow$ \\ \hline
    \end{tabular}
    \caption{Computed eigenvalues by $\mathbb{RT}^{\rm nc}_h$ scheme on crisscross grids.}
    \label{tab:nrtevcc}
\end{table}
}

\begin{table}[!ht]
    \small
    \centering
    \begin{tabular}{|c|c|c|c|c|c|c|c|c|c|c|}
 \hline
    L&$\lambda_h^1$&$\lambda_h^2$&$\lambda_h^3$&$\lambda_h^4$&$\lambda_h^5$&$\lambda_h^6$&$\lambda_h^7$&$\lambda_h^8$&$\lambda_h^9$&$\lambda_h^{10}$\\
        \hline
        1 & 2.474  & 6.921  & 7.575  & 12.150  & 17.167  & 20.663  & 21.687  & 23.653  & 25.825  & 26.359  \\ \hline
        2 & 2.123  & 5.636  & 5.770  & 9.850  & 12.660  & 13.141  & 18.222  & 18.582  & 24.122  & 25.578  \\ \hline
        3 & 2.031  & 5.164  & 5.199  & 8.474  & 10.712  & 10.778  & 14.307  & 14.357  & 18.985  & 19.218  \\ \hline
        4 & 2.008  & 5.041  & 5.050  & 8.119  & 10.182  & 10.195  & 13.325  & 13.336  & 17.520  & 17.565  \\ \hline
        5 & 2.002  & 5.010  & 5.013  & 8.030  & 10.046  & 10.049  & 13.081  & 13.084  & 17.132  & 17.142  \\ \hline
         & $\searrow$  & $\searrow$  & $\searrow$  & $\searrow$  & $\searrow$  & $\searrow$  & $\searrow$  & $\searrow$  & $\searrow$  & $\searrow$ \\ \hline
    \end{tabular}
    \caption{Computed eigenvalues by $\mathbb{RT}^{\rm nc}_h$ scheme on irregular grids.}
    \label{tab:nrtevunstr}
\end{table}

\subsubsection{Comparison with the classical Raviart-Thomas element}

\begin{figure}[htbp]
\includegraphics[width=0.32\textwidth]{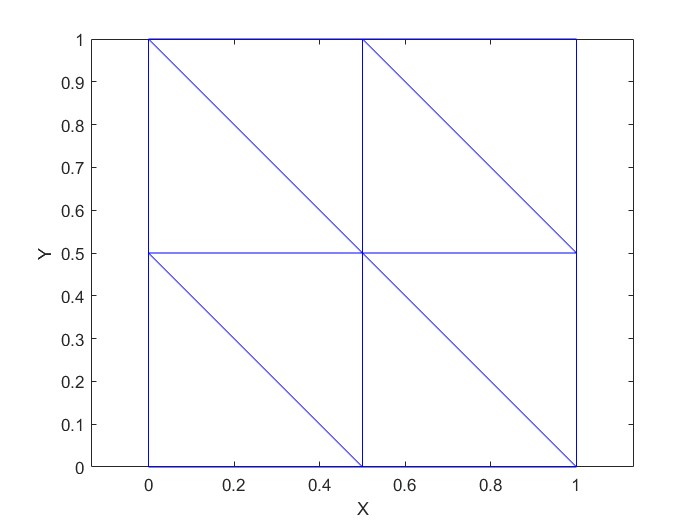}
\includegraphics[width=0.32\textwidth]{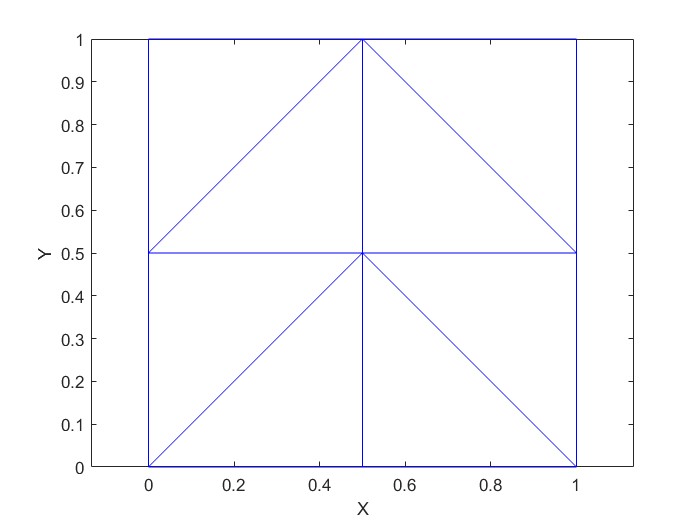}
\includegraphics[width=0.32\textwidth]{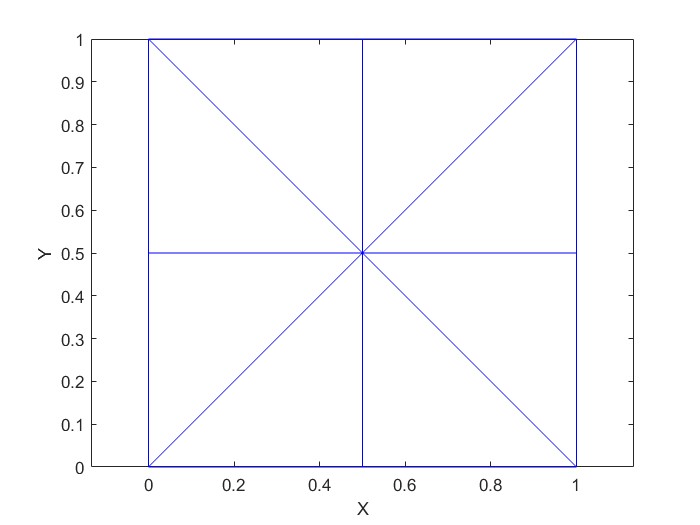}
\caption{The initial triangulations. Left: regular; middle: fish bone; right: union Jack.}\label{fig:threeregtri}
\end{figure}

We here show the experiments of solving the eigenvalue problem \eqref{eq:Poissondualevp} with the classical Raviart-Thomas element scheme on the crisscross triangulation, the regular triangulation (cf. Figure \ref{fig:threeregtri}, left), the fish-bone triangulation (cf. Figure \ref{fig:threeregtri}, middle), and the union Jack triangulation (cf. Figure \ref{fig:threeregtri}, right). The 10 smallest computed eigenvalues on each series of grids are recorded in Tables \ref{tab:rtevcc}, \ref{tab:rtevreg}, \ref{tab:rtevfish} and \ref{tab:rtevuj}. It can be seen that the classical (lowest-degree) Raviart-Thomas element might provide upper or lower bound to different eigenvalues, sensitive to the grid as well.\footnote{We do not think it is now found for the first time that the classical (lowest-degree) Raviart-Thomas element scheme cannot be expected to provide a certain bounds to the exact eigenvalues, though we do not find a referred literature.} Numerical experiments on these special triangulations which are easy to check are included.

\begin{table}[!ht]
    \small
    \centering
    \begin{tabular}{|c|c|c|c|c|c|c|c|c|c|c|}
 \hline
    L&$\lambda_h^1$&$\lambda_h^2$&$\lambda_h^3$&$\lambda_h^4$&$\lambda_h^5$&$\lambda_h^6$&$\lambda_h^7$&$\lambda_h^8$&$\lambda_h^9$&$\lambda_h^{10}$\\
    \hline
        1 & 1.858  & 4.158  & 4.158  & 8.254  & 9.727  & 12.042  & 12.042  & 12.733  & 14.590  & 14.590  \\ \hline
        2 & 1.965  & 4.893  & 4.893  & 7.431  & 9.850  & 9.850  & 11.731  & 11.731  & 14.847  & 15.317  \\ \hline
        3 & 1.991  & 4.975  & 4.975  & 7.862  & 9.986  & 9.986  & 12.712  & 12.712  & 17.071  & 17.071  \\ \hline
        4 & 1.998  & 4.994  & 4.994  & 7.966  & 9.998  & 9.998  & 12.929  & 12.929  & 17.024  & 17.024  \\ \hline
        5 & 1.999  & 4.998  & 4.998  & 7.991  & 9.999  & 9.999  & 12.982  & 12.982  & 17.006  & 17.006  \\ \hline
         & $\nearrow$ & $\nearrow$ & $\nearrow$ & $\nearrow$ & $\nearrow$ & $\nearrow$ & $\nearrow$ & $\nearrow$ & $\searrow$ & $\searrow$ \\ \hline
    \end{tabular}
    \caption{Computed eigenvalues by the classical Raviart-Thomas element scheme on crisscross grids.}
    \label{tab:rtevcc}
\end{table}

\begin{table}[!ht]
    \small
    \centering
    \begin{tabular}{|c|c|c|c|c|c|c|c|c|c|c|}
 \hline
    L&$\lambda_h^1$&$\lambda_h^2$&$\lambda_h^3$&$\lambda_h^4$&$\lambda_h^5$&$\lambda_h^6$&$\lambda_h^7$&$\lambda_h^8$&$\lambda_h^9$&$\lambda_h^{10}$\\
    \hline
        1 & 2.110  & 3.542  & 4.863  & 9.727  & 9.727  & 12.021  & 13.453  & 14.590  & ---  & ---  \\ \hline
        2 & 2.032  & 4.834  & 5.096  & 8.077  & 8.957  & 9.414  & 11.107  & 11.377  & 12.242  & 14.729  \\ \hline
        3 & 2.008  & 4.964  & 5.026  & 8.119  & 9.798  & 9.815  & 12.896  & 13.422  & 16.153  & 16.196  \\ \hline
        4 & 2.002  & 4.991  & 5.007  & 8.033  & 9.951  & 9.952  & 12.983  & 13.113  & 16.791  & 16.799  \\ \hline
        5 & 2.001  & 4.998  & 5.002  & 8.009  & 9.988  & 9.988  & 12.996  & 13.029  & 16.947  & 16.950  \\ \hline
         & $\searrow$ & $\nearrow$ & $\searrow$ & $\searrow$ & $\nearrow$ & $\nearrow$ & $\nearrow$ & $\searrow$ & $\nearrow$ & $\nearrow$ \\ \hline
    \end{tabular}
    \caption{Computed eigenvalues by the classical Raviart-Thomas element scheme on regular grids.}
    \label{tab:rtevreg}
\end{table}

\begin{table}[!ht]
    \small
    \centering
    \begin{tabular}{|c|c|c|c|c|c|c|c|c|c|c|}
 \hline
    L&$\lambda_h^1$&$\lambda_h^2$&$\lambda_h^3$&$\lambda_h^4$&$\lambda_h^5$&$\lambda_h^6$&$\lambda_h^7$&$\lambda_h^8$&$\lambda_h^9$&$\lambda_h^{10}$\\
    \hline
        1 & 2.084  & 4.127  & 4.127  & 9.727  & 9.727  & 12.895  & 12.895  & 14.590  & ---  & ---  \\ \hline
        2 & 2.032  & 4.943  & 4.959  & 8.337  & 8.881  & 8.989  & 11.359  & 11.501  & 12.716  & 13.188  \\ \hline
        3 & 2.008  & 4.993  & 4.995  & 8.126  & 9.788  & 9.800  & 13.153  & 13.166  & 16.107  & 16.159  \\ \hline
        4 & 2.002  & 4.999  & 4.999  & 8.034  & 9.950  & 9.951  & 13.047  & 13.048  & 16.790  & 16.794  \\ \hline
        5 & 2.001  & 5.000  & 5.000  & 8.009  & 9.988  & 9.988  & 13.012  & 13.012  & 16.948  & 16.948  \\ \hline
         & $\searrow$ & $\nearrow$ & $\nearrow$ & $\searrow$ & $\nearrow$ & $\nearrow$ & $\searrow$ & $\searrow$ & $\nearrow$ & $\nearrow$ \\ \hline
    \end{tabular}
    \caption{Computed eigenvalues by the classical Raviart-Thomas element scheme on fish-bone grids.}
    \label{tab:rtevfish}
\end{table}

\begin{table}[!ht]
    \small
    \centering
    \begin{tabular}{|c|c|c|c|c|c|c|c|c|c|c|}
 \hline
    L&$\lambda_h^1$&$\lambda_h^2$&$\lambda_h^3$&$\lambda_h^4$&$\lambda_h^5$&$\lambda_h^6$&$\lambda_h^7$&$\lambda_h^8$&$\lambda_h^9$&$\lambda_h^{10}$\\
    \hline
        1 & 2.432  & 4.127  & 4.127  & 7.295  & 9.727  & 12.895  & 12.895  & 14.590  & ---  & ---  \\ \hline
        2 & 2.030  & 4.925  & 4.925  & 8.315  & 9.727  & 9.727  & 11.501  & 11.501  & 13.497  & 13.497  \\ \hline
        3 & 2.008  & 4.993  & 4.993  & 8.120  & 9.786  & 9.786  & 13.133  & 13.133  & 16.097  & 16.097  \\ \hline
        4 & 2.002  & 4.999  & 4.999  & 8.033  & 9.950  & 9.950  & 13.047  & 13.047  & 16.789  & 16.789  \\ \hline
        5 & 2.001  & 5.000  & 5.000  & 8.009  & 9.988  & 9.988  & 13.012  & 13.012  & 16.948  & 16.948  \\ \hline
         & $\searrow$ & $\nearrow$ & $\nearrow$ & $\searrow$ & $\nearrow$ & $\nearrow$ & $\searrow$ & $\searrow$ & $\nearrow$ & $\nearrow$ \\ \hline
    \end{tabular}
    \caption{Computed eigenvalues by the classical Raviart-Thomas element scheme on union Jack grids.}
    \label{tab:rtevuj}
\end{table}

The $\mathbb{RT}^{\rm nc}_h$ scheme for \eqref{eq:Poissondualevp} is further carried out on the regular triangulation, fish-bone triangulation and the union Jack triangulation, and the 10 smallest computed eigenvalues on each series of grids are recorded in Tables \ref{tab:nrtevreg}, \ref{tab:nrtevfish} and \ref{tab:nrtevuj}. It can be seen that, in all these experiments, again, the $\mathbb{RT}^{\rm nc}_h$ scheme for \eqref{eq:Poissondualevp} provides upper bounds for all the eigenvalues. The robustness is improved with $\mathbb{RT}^{\rm nc}_h$. This will be further investigated in future. 

\begin{table}[!ht]
    \small
    \centering
    \begin{tabular}{|c|c|c|c|c|c|c|c|c|c|c|}
 \hline
    L&$\lambda_h^1$&$\lambda_h^2$&$\lambda_h^3$&$\lambda_h^4$&$\lambda_h^5$&$\lambda_h^6$&$\lambda_h^7$&$\lambda_h^8$&$\lambda_h^9$&$\lambda_h^{10}$\\
    \hline
        1 & 3.648  & 14.590  & 14.590  & 14.590  & 14.590  & 14.590  & 14.590  & 14.590  & ---  & ---  \\ \hline
        2 & 2.396  & 6.748  & 8.210  & 13.339  & 19.454  & 21.970  & 23.399  & 33.381  & 36.189  & 58.361  \\ \hline
        3 & 2.095  & 5.414  & 5.692  & 9.432  & 12.082  & 12.343  & 15.678  & 18.242  & 23.299  & 23.656  \\ \hline
        4 & 2.024  & 5.102  & 5.166  & 8.372  & 10.494  & 10.510  & 13.684  & 14.246  & 18.387  & 18.430  \\ \hline
        5 & 2.006  & 5.026  & 5.041  & 8.094  & 10.122  & 10.123  & 13.173  & 13.306  & 17.335  & 17.344  \\ \hline
         & $\searrow$ & $\searrow$ & $\searrow$ & $\searrow$ & $\searrow$ & $\searrow$ & $\searrow$ & $\searrow$ & $\searrow$ & $\searrow$ \\ \hline
    \end{tabular}
    \caption{Computed eigenvalues by $\mathbb{RT}^{\rm nc}_h$ scheme on regular grids.}
    \label{tab:nrtevreg}
\end{table}

\begin{table}[!ht]
    \small
    \centering
    \begin{tabular}{|c|c|c|c|c|c|c|c|c|c|c|}
 \hline
    L&$\lambda_h^1$&$\lambda_h^2$&$\lambda_h^3$&$\lambda_h^4$&$\lambda_h^5$&$\lambda_h^6$&$\lambda_h^7$&$\lambda_h^8$&$\lambda_h^9$&$\lambda_h^{10}$\\
    \hline
        1 & 3.648  & 14.590  & 14.590  & 14.590  & 14.590  & 14.590  & 14.590  & 14.590  & ---  & ---  \\ \hline
        2 & 2.395  & 7.247  & 7.455  & 14.590  & 17.639  & 20.437  & 26.875  & 32.313  & 36.332  & 58.361  \\ \hline
        3 & 2.095  & 5.537  & 5.552  & 9.559  & 11.969  & 12.131  & 16.941  & 17.131  & 22.453  & 23.322  \\ \hline
        4 & 2.024  & 5.133  & 5.134  & 8.380  & 10.485  & 10.497  & 13.960  & 13.973  & 18.334  & 18.398  \\ \hline
        5 & 2.006  & 5.033  & 5.033  & 8.094  & 10.121  & 10.122  & 13.239  & 13.240  & 17.334  & 17.339  \\ \hline
         & $\searrow$ & $\searrow$ & $\searrow$ & $\searrow$ & $\searrow$ & $\searrow$ & $\searrow$ & $\searrow$ & $\searrow$ & $\searrow$ \\ \hline
    \end{tabular}
    \caption{Computed eigenvalues by $\mathbb{RT}^{\rm nc}_h$ scheme on fish-bone grids.}
    \label{tab:nrtevfish}
\end{table}

\begin{table}[!ht]
    \small
    \centering
    \begin{tabular}{|c|c|c|c|c|c|c|c|c|c|c|}
 \hline
    L&$\lambda_h^1$&$\lambda_h^2$&$\lambda_h^3$&$\lambda_h^4$&$\lambda_h^5$&$\lambda_h^6$&$\lambda_h^7$&$\lambda_h^8$&$\lambda_h^9$&$\lambda_h^{10}$\\
    \hline
        1 & 2.918  & 14.590  & 14.590  & 14.590  & 14.590  & 14.590  & 14.590  & 14.590  & ---  & ---  \\ \hline
        2 & 2.366  & 7.274  & 7.274  & 11.672  & 19.454  & 19.454  & 29.531  & 29.531  & 43.615  & 58.361  \\ \hline
        3 & 2.087  & 5.505  & 5.505  & 9.466  & 11.963  & 11.963  & 16.852  & 16.852  & 22.973  & 22.973  \\ \hline
        4 & 2.022  & 5.121  & 5.121  & 8.349  & 10.447  & 10.447  & 13.893  & 13.893  & 18.258  & 18.258  \\ \hline
        5 & 2.005  & 5.030  & 5.030  & 8.086  & 10.109  & 10.109  & 13.218  & 13.218  & 17.301  & 17.301  \\ \hline
         & $\searrow$ & $\searrow$ & $\searrow$ & $\searrow$ & $\searrow$ & $\searrow$ & $\searrow$ & $\searrow$ & $\searrow$ & $\searrow$ \\ \hline
    \end{tabular}
    \caption{Computed eigenvalues by $\mathbb{RT}^{\rm nc}_h$ scheme on union Jack grids.}
    \label{tab:nrtevuj}
\end{table}

%
%
%
\section{Nonconforming finite element exterior calculus}
\label{sec:ncdcx}

\subsection{Nonconforming finite element spaces for $H\Lambda^k$ in $\rn$}

Let $\mathcal{G}_h$ be a simplicial subdivision of $\Omega$. For $0\leqslant k\leqslant n-1$, we define finite element spaces for $\hlk$ by
\begin{equation}
\mathbf{W}^{\rm nc}_h\Lambda^k:=\left\{\fomega_h\in  \mathcal{P}^-_1\Lambda^k(\mathcal{G}_h):
\langle\fomega_h,\odelta_{k+1}\feta_h\rangle_{L^2\Lambda^k}-\langle\od^k_h\fomega_h,\feta_h\rangle_{L^2\Lambda^{k+1}}=0,\ \forall\,\feta_h\in\fW^*_{h0}\Lambda^{k+1}\right\},
\end{equation}
and, for $H_0\Lambda^k$, 
\begin{equation} 
\mathbf{W}^{\rm nc}_{h0}\Lambda^k:=\left\{\fomega_h\in  \mathcal{P}^-_1\Lambda^k(\mathcal{G}_h): \langle\fomega_h,\odelta_{k+1}\feta_h\rangle_{L^2\Lambda^k}-\langle\od^k_h\fomega_h,\feta_h\rangle_{L^2\Lambda^{k+1}}=0,\ \forall\,\feta_h\in\fW^*_h\Lambda^{k+1}\right\}.
\end{equation}
Set
\begin{equation}
\fW^{\rm nc}_h\Lambda^n:=\mathcal{P}_0\Lambda^n(\mathcal{G}_h),\quad \mbox{and}\quad \fW^{\rm nc}_{h0}\Lambda^n:=\fW^{\rm nc}_h\Lambda^n\cap L^2_0\Lambda^n(\Omega).
\end{equation}

\begin{remark}
Note that $\fW^{\rm nc}_h\Lambda^0$ and $\fW^{\rm nc}_{h0}\Lambda^0$ are the lowest-degree Crouzeix-Raviart element spaces. If further $n=1$, $\fW^{\rm nc}_h\Lambda^0$ and $\fW^{\rm nc}_{h0}\Lambda^0$ coincide with the respective continuous linear element spaces. 
\end{remark}
\begin{remark}\label{rem:mutual}
Associated with the definitions, by \eqref{eq:n=r=c}, it holds that, for example, 
$$
\fW^*_h\Lambda^k=\left\{\fmu_h\in\mathcal{P}^{*,-}_1\Lambda^k(\mathcal{G}_h): \langle\odelta_k\fmu_h,\ftau_h\rangle_{L^2\Lambda^{k-1}}-\langle\fmu_h,\od^{k-1}_h\ftau_h\rangle_{L^2\Lambda^k}=0,\ \forall\,\ftau_h\in \fW^{\rm nc}_{h0}\Lambda^{k-1}\right\}.
$$
\end{remark}

By the same virtue of Theorem \ref{thm:localbasisrtabc}, noting \eqref{eq:n=r=c}, we can prove theorem below. 
\begin{theorem}
The space $\mathbf{W}^{\rm nc}_{h0}\Lambda^k$ admits a set of linear independent basis functions, which are each supported on two adjacent simplices.  

The space $\mathbf{W}^{\rm nc}_h\Lambda^k$ admits a set of linear independent basis functions; they consist of two types of functions, Type I and Type II. The functions of Type I are each supported on two adjacent simplices, and the functions of Type II are each supported on one simplex.
\end{theorem}

In the sequel, we use $\mathcal{F}^\mathcal{G}$ for a family of shape regular subdivisions of $\Omega$.

\subsubsection{Locally defined interpolator and optimal approximation}
\label{subsubsec:inthlk}
Similar to \eqref{eq:irtlocal}, we define a local interpolator $\mathbb{I}^{\od^k}_T:H\Lambda^k(T)\to \mathcal{P}^-_1\Lambda^k(T)$, $0\leqslant k\leqslant n-1$, such that, 
$$
\langle \mathbb{I}^{\od^k}_T\fomega,\odelta_{k+1}\feta\rangle_{L^2\Lambda^k(T)}-\langle\od^k\mathbb{I}^{\od^k}_T\fomega,\feta\rangle_{L^2\Lambda^{k+1}(T)}=\langle \fomega,\odelta_{k+1}\feta\rangle_{L^2\Lambda^k(T)}-\langle\od^k\fomega,\feta\rangle_{L^2\Lambda^{k+1}(T)},
$$ 
for any $\feta\in\mathcal{P}^{*,-}_1\Lambda^{k+1}(T)$, and, following \eqref{eq:globalintp}, define a global interpolator 
$$
\displaystyle\mathbb{I}^{\od^k}_h:\bigoplus_{T\in\mathcal{G}_h} E_T^\Omega H\Lambda^k(T)\to \mathcal{P}^-_1\Lambda^k(\mathcal{G}_h),\ \mbox{by}\ (\mathbb{I}^{\od^k}_h\fomega)|_T=\mathbb{I}^{\od^k}_T(\fomega|_T),\ \forall\,T\in\mathcal{G}_h.
$$ 
Set $\mathbb{I}^{\od^n}_T$ the $L^2(T)$ projection to $\mathcal{P}_0\Lambda^n$ on $T$, and $\mathbb{I}^{\od^n}_h$ the $L^2(\Omega)$ projection to $\mathcal{P}_0\Lambda^n(\mathcal{G}_h)$. 

Denote $\|\fmu_h\|_{\od^k_h}:=(\|\od^k_h\fmu_h\|_{L^2\Lambda^{k+1}}^2+\|\fmu_h\|_{L^2\Lambda^k}^2)^{1/2}$. The proofs of the two lemmas below are the same as that of Lemma \ref{lem:intdivinrt} and Lemma \ref{lem:intdiv}, and are omitted here. 
\begin{lemma}
$\R(\mathbb{I}^{\od^k}_h,\hlk)\subset \fW^{\rm nc}_h\Lambda^k$ and $\R(\mathbb{I}^{\od^k}_h,H_0\Lambda^k)\subset \fW^{\rm nc}_{h0}\Lambda^k$.
\end{lemma}

\begin{lemma}
With $C_{k,n}$ uniform for $\mathcal{F}^\mathcal{G}$, for $\mathcal{G}_h\in\mathcal{F}^\mathcal{G}$ and $\displaystyle\fomega\in \bigoplus_{T\in\mathcal{G}_h} E_T^\Omega H\Lambda^k(T)$, 
$$
\|\fomega-\mathbb{I}^{\rm \od^k}_h\fomega\|_{\od^k_h}\leqslant C_{k,n}\inf_{\feta_h\in \mathcal{P}^-_1\Lambda^k(\mathcal{G}_h)}\|\fomega-\feta_h\|_{\od^k_h}. 
$$ 
\end{lemma}

\subsubsection{Uniform discrete Poincar\'e inequalities}
\label{subsec:ubpi}
As generally $\R(\od^k_h,\fW^{\rm nc}_h\Lambda^k)\not\subset H\Lambda^{k+1}(\Omega)$, we cannot simply repeat the proof of Lemma \ref{lem:stabrtabc}. We adopt an indirect approach, which can be viewed a finite-dimensional analogue of the closed range theorem. 

Let $\xX$ and $\yY$ be two Hilbert spaces. For $(\oT,\xD):\xX\to \yY$ a closed  operator, denote 
$$
\xD^{\boldsymbol \lrcorner_{\oT}}:=\left\{\xv\in \xD:\langle \xv,\xw\rangle_\xX=0,\ \forall\,\xw\in \mathcal{N}(\oT,\xD)\right\}.
$$ 
Define the {\bf Poincar\'e inequality's criterion} of $(\oT,\xD)$ as
\begin{equation}\label{eq:deficr}
\mathsf{pic}(\oT,\xD):=\left\{\begin{array}{rl}
\displaystyle \sup_{0\neq\xv\in \xD^{\boldsymbol \lrcorner_{\oT}}}\frac{\|\xv\|_\xX}{\|\oT\xv\|_\yY},&\mbox{if}\ \xD^{\boldsymbol \lrcorner_{\oT}}\neq\left\{0\right\};
\\
0,&\mbox{if}\ \xD^{\boldsymbol \lrcorner_{\oT}}=\left\{0\right\}.
\end{array}\right.
\end{equation}

If $\icr(\oT,\xD)$ is finite, then the Poincar\'e inequality holds for $(\oT,\xD)$. It is further indeed the best constant of the Poincar\'e inequality. The index can be used for a criterion for closed range. We refer to, e.g., \cite[Lemma 3.6]{Arnold.D2018feec} for a proof of Lemma \ref{lem:pivscr} up to little technical modification. 
\begin{lemma}\label{lem:pivscr}
For $(\oT,\xD):\xX\to \yY$ a closed operator, $\R(\oT,\xD)$ is closed if and only if $\icr(\xT,\xD)<+\infty$.
\end{lemma}

The main estimation is the theorem below. 
\begin{theorem}\label{thm:piwabc}
With a constant $C_{k,n}$ uniform for $\mathcal{F}^\mathcal{G}$, 
$$
\icr(\od_h^k,\fW^{\rm nc}_h\Lambda^k)\leqslant C_{k,n}.
$$
\end{theorem}
We firstly present three lemmas below, and postpone their proofs to appendix Section \ref{sec:localwhitney}.
\begin{lemma}\label{lem:discontbpi}
$\icr(\od_h^k,\fW^{\rm nc}_h\Lambda^k)\leqslant \icr(\odelta_{k+1},\fW^*_{h0}\Lambda^{k+1})+2\icr(\od_h^k,\mathcal{P}^-_1\Lambda^k(\mathcal{G}_h)).$
\end{lemma}
\begin{lemma}\label{lem:localhot}
$\icr(\od_h^k,\mathcal{P}^-_1\Lambda^k(\mathcal{G}_h))=\mathcal{O}(h)$.
\end{lemma}

\begin{lemma}\label{lem:discrt}
$\left|\icr(\odelta_{k+1},\fW^*_{h0}\Lambda^{k+1})-\icr(\od_h^k,\fW^{\rm nc}_h\Lambda^k)\right|=\mathcal{O}(h)$.
\end{lemma}

\paragraph{\bf Proof of Theorem \ref{thm:piwabc}}
It is well known that (c.f., e.g., \cite{Arnold.D;Falk.R;Winther.R2006acta}), there exists a constant $C_{k,n}$ such that $\icr(\od^k,\mathbf{W}_h\Lambda^k)\leqslant C_{k,n}$, and, $\icr(\od^k,\mathbf{W}_{h0}\Lambda^k)\leqslant C_{k,n}$, which implies immediately that $\icr(\odelta_{k+1},\fW^*_{h0}\Lambda^{k+1})$ and $\icr(\odelta_{k+1},\fW^*_h\Lambda^{k+1})$ are uniformly bounded. It follows then $\icr(\od_h^k,\fW^{\rm nc}_h\Lambda^k)\leqslant C_{k,n}$. \qed

Similarly, $\icr(\od_{h0}^k,\fW^{\rm nc}_h\Lambda^k)\leqslant C_{k,n}$.

\subsubsection{Finite element schemes for elliptic variational problems}


Consider the elliptic variational problem: given $\ff\in L^2\Lambda^k$, find $\fomega\in H\Lambda^k$, such that 
\begin{equation}\label{eq:evpdlk}
\langle\od^k \fomega,\od^k \fmu\rangle_{L^2\Lambda^{k+1}}+\langle \fomega,\fmu\rangle_{L^2\Lambda^k}=\langle \ff,\fmu\rangle_{L^2\Lambda^k},\ \ \forall\,\fmu\in L^2\Lambda^k.
\end{equation}
It follows that $\od^k\fomega \in H^*_0\Lambda^{k+1}$, and $\odelta_{k+1}\od^k\fomega+\fomega=\ff$.

We consider its finite element discretization: find $\fomega\in\fW^{\rm nc}_h\Lambda^k$, such that 
\begin{equation}\label{eq:disprolk}
\langle\od^k_h \fomega_h,\od^k_h \fmu_h\rangle_{L^2\Lambda^{k+1}}+\langle \fomega_h,\fmu_h\rangle_{L^2\Lambda^k}=\langle \ff,\fmu_h\rangle_{L^2\Lambda^k},\ \ \forall\,\fmu_h\in \fW^{\rm nc}_h\Lambda^k.
\end{equation}

Immediately \eqref{eq:evpdlk} and \eqref{eq:disprolk} are well-posed. 

\begin{theorem}\label{thm:basicestrn}
Let $\fomega$ and $\fomega_h$ be the solutions of \eqref{eq:evpdlk} and \eqref{eq:disprolk}, respectively. 
\begin{equation*}
\|\fomega-\fomega_h\|_{\od^k_h}\leqslant 2\inf_{\fmu_h\in\fW^{\rm nc}_h}\|\fomega-\fmu_h\|_{\od^k_h}+\inf_{\ftau_h\in \fW^*_{h0}\Lambda^{k+1}}\|\od^k\fomega-\ftau_h\|_{\odelta_{k+1}}.
\end{equation*}
\end{theorem}
The proof is the same as that of Theorem \ref{thm:basicest}, and is omitted here.

\subsection{Discrete Helmholtz-Hodge decompositions of $\mathcal{P}_0\Lambda^k(\mathcal{G}_h)$}
\label{subsec:decoms}
%

\begin{theorem}[Discrete Helmholtz decomposition]  \label{thm:dishd}
Orthogonal in $L^2\Lambda^k(\Omega)$, for $1\leqslant k\leqslant n$,
$$
\mathcal{P}_0\Lambda^k(\mathcal{G}_h)=\mathcal{R}(\od^{k-1}_h,\mathbf{W}^{\rm nc}_h\Lambda^{k-1}) \opp \mathcal{N}(\odelta_k,\mathbf{W}_{h0}^*\Lambda^k)=\mathcal{R}(\od^{k-1}_h,\mathbf{W}^{\rm nc}_{h0}\Lambda^{k-1}) \opp \mathcal{N}(\odelta_k,\mathbf{W}_h^*\Lambda^k);
$$
for $0\leqslant k\leqslant n-1$,
$$
\mathcal{P}_0\Lambda^k(\mathcal{G}_h)=\mathcal{N}(\od^k_h,\mathbf{W}^{\rm nc}_h\Lambda^k) \opp \mathcal{R}(\odelta_{k+1},\mathbf{W}_{h0}^*\Lambda^{k+1})=\mathcal{N}(\od^k_h,\mathbf{W}^{\rm nc}_{h0}\Lambda^k) \opp \mathcal{R}(\odelta_{k+1},\mathbf{W}_h^*\Lambda^{k+1}).
$$
\end{theorem}
\begin{proof}
~~We are going to show, for $1\leqslant k\leqslant n$,
$$
\mathcal{P}_0\Lambda^k(\mathcal{G}_h)=\mathcal{R}(\od^{k-1}_h,\mathbf{W}^{\rm nc}_h\Lambda^{k-1}) \opp \mathcal{N}(\odelta_k,\mathbf{W}_{h0}^*\Lambda^k),$$
and other assertions follow the same way. 

By construction, $\mathcal{P}_0\Lambda^k(\mathcal{G}_h)$ contains $\mathcal{R}(\od^{k-1}_h,\mathbf{W}^{\rm nc}_h\Lambda^{k-1}) \opp \mathcal{N}(\odelta_k,\mathbf{W}_{h0}^*\Lambda^k)$. Conversely, let $\fsigma_h\in \mathcal{P}_0\Lambda^k(\mathcal{G}_h)\omp \mathcal{R}(\od^{k-1}_h,\fW^{\rm nc}_h\Lambda^{k-1})$. Then for any $\fmu_h\in \mathbf{W}^{\rm nc}_h\Lambda^{k-1}$, 
$$
\sum_{T\in\mathcal{G}_h}\langle\fsigma_h,\od^{k-1}\fmu_h\rangle_{\fL^2\Lambda^k(T)}+\langle\odelta_k\fsigma_h,\fmu_h\rangle_{\fL^2\Lambda^{k-1}(T)}=\langle\fsigma_h,\od^{k-1}_h\fmu_h\rangle_{\fL^2\Lambda^k}=0.
$$
Namely $\fsigma_h\in \fW^*_{h0}\Lambda^k$ and further $\fsigma_h\in \mathcal{N}(\odelta_k,\mathbf{W}_{h0}^*\Lambda^k)$. This completes the proof. 
\end{proof}

By noting that, for $1\leqslant k\leqslant n-1$, 
$$
\mathcal{P}_0\Lambda^k(\mathcal{G}_h)=\mathcal{R}(\od^{k-1}_h,\mathbf{W}^{\rm nc}_h\Lambda^{k-1}) \opp \mathcal{N}(\odelta_k,\mathbf{W}_{h0}^*\Lambda^k) = \mathcal{N}(\od^k_h,\mathbf{W}^{\rm nc}_h\Lambda^k) \opp \mathcal{R}(\odelta_{k+1},\mathbf{W}_{h0}^*\Lambda^{k+1}),
$$
we have immediately that, for $1\leqslant k\leqslant n-1$, 
\begin{equation}\label{eq:complexduality}
\mathcal{R}(\od^{k-1}_h,\mathbf{W}^{\rm nc}_h\Lambda^{k-1})\subset \mathcal{N}(\od^k_h,\mathbf{W}^{\rm nc}_h\Lambda^k) \Longleftrightarrow     \mathcal{R}(\odelta_{k+1},\mathbf{W}_{h0}^*\Lambda^{k+1}) \subset \mathcal{N}(\odelta_k,\mathbf{W}_{h0}^*\Lambda^k).
\end{equation}

Further, we can construct the discrete Poincar\'e-Lefschetz duality identities below.
\begin{theorem} [Discrete Poincar\'e-Lefschetz duality] \label{thm:displd}
For $1\leqslant k\leqslant n-1$, 
$$
\mathcal{N}(\od^k_h,\mathbf{W}^{\rm nc}_h\Lambda^k)\omp \mathcal{R}(\od^{k-1}_h,\mathbf{W}^{\rm nc}_h\Lambda^{k-1}) 
= 
 \mathcal{N}(\odelta_k,\mathbf{W}_{h0}^*\Lambda^k) \omp \mathcal{R}(\odelta_{k+1},\mathbf{W}_{h0}^*\Lambda^{k+1})
$$
and
$$
\mathcal{N}(\od^k_h,\mathbf{W}^{\rm nc}_{h0}\Lambda^k)\omp \mathcal{R}(\od^{k-1}_h,\mathbf{W}^{\rm nc}_{h0}\Lambda^{k-1}) 
=
\mathcal{N}(\odelta_k,\mathbf{W}_h^*\Lambda^k) \omp \mathcal{R}(\odelta_{k+1},\mathbf{W}_h^*\Lambda^{k+1}).
$$
\end{theorem}

Denote $\hf_h^{\rm nc}\Lambda^k:=\mathcal{N}(\od^k_h,\mathbf{W}^{\rm nc}_h\Lambda^k)\omp \mathcal{R}(\od^{k-1}_h,\mathbf{W}^{\rm nc}_h\Lambda^{k-1})$ and 
$\hf^{\rm nc}_{h0}\Lambda^k:=\mathcal{N}(\od^k_h,\mathbf{W}^{\rm nc}_{h0}\Lambda^k)\omp \mathcal{R}(\od^{k-1}_h,\mathbf{W}^{\rm nc}_{h0}\Lambda^{k-1}).$ We have the discrete orthogonal decomposition of $\mathcal{P}_0\Lambda^k(\mathcal{G}_h)$ below. 
\begin{theorem}[Discrete Hodge decomposition]\label{thm:dhgd}
For $1\leqslant k\leqslant n-1$,
\begin{multline}\label{eq:dishgd}
\displaystyle
\qquad\mathcal{P}_0\Lambda^k(\mathcal{G}_h)=\mathcal{R}(\od^{k-1}_h,\mathbf{W}^{\rm nc}_h\Lambda^{k-1}) \opp \hf^*_{h0}\Lambda^k(= \hf_h^{\rm nc}\Lambda^k) \opp \mathcal{R}(\odelta_{k+1},\mathbf{W}_{h0}^*\Lambda^{k+1}) 
\\
\displaystyle =\mathcal{R}(\od^{k-1}_h,\mathbf{W}^{\rm nc}_{h0}\Lambda^{k-1}) \opp \hf_{h0}^{\rm nc}\Lambda^k (= \hf^*_h\Lambda^k) \opp \mathcal{R}(\odelta_{k+1},\mathbf{W}_h^*\Lambda^{k+1}).\qquad
\end{multline}
\end{theorem}

\begin{remark}
Existing discrete Hodge decompositions in literature are of discretized spaces of $H\Lambda^k$; see \cite[(5.6)]{Arnold.D2018feec} for example. Contrastly, Theorem \ref{thm:dhgd} is of a discretization of $L^2\Lambda^k$. 
\end{remark}

\subsection{Commutative diagrams}
\begin{lemma}
For any $\fmu\in H\Lambda^k(T)$, $0\leqslant k\leqslant n-1$,
$\mathbb{I}_T^{\od^{k+1}}\od^k\fmu=\od^k\mathbb{I}_T^{\od^k}\fmu$.
\end{lemma}
\begin{proof}
~~Since $\od^{k+1}\od^k\fmu=0$, $\od^{k+1}\mathbb{I}^{\od^{k+1}}_T\od^k\fmu=0$. Further, $\mathbb{I}^{\od^{k+1}}_T\od^k\fmu\in\mathcal{P}_0\Lambda^{k+1}$. Then, 
\begin{multline*}
\langle \mathbb{I}_T^{\od^{k+1}}\od^k\fmu,\odelta_{k+2}\feta\rangle_{L^2\Lambda^{k+1}(T)}-\underline{\langle \od^{k+1}\mathbb{I}^{\od^{k+1}}_T\od^k\fmu,\feta \rangle_{L^2\Lambda^{k+2}(T)}}
\\
=\langle \od^k\fmu,\odelta_{k+2}\feta\rangle_{L^2\Lambda^{k+1}(T)}-\underline{\langle \od^{k+1}\od^k\fmu,\feta \rangle_{L^2\Lambda^{k+2}(T)}}
= \langle \od^k\fmu,\odelta_{k+2}\feta\rangle_{L^2\Lambda^{k+1}(T)}-\underline{\langle\fmu,\odelta_{k+1}\odelta_{k+2}\feta\rangle_{L^2\Lambda^k(T)}}
\\
=\langle \od^k\mathbb{I}^{\od^k}_T\fmu,\odelta_{k+2}\feta\rangle_{L^2\Lambda^{k+1}(T)}-\underline{\langle\mathbb{I}^{\od^k}_T\fmu,\odelta_{k+1}\odelta_{k+2}\feta\rangle_{L^2\Lambda^k(T)}},\ \ \forall\,\feta\in \mathcal{P}^{*,-}_1\Lambda^{k+2}(T).
\end{multline*}
Here we use underline to label the vanishing terms. Therefore, $\mathbb{I}_T^{\od^{k+1}}\od^k\fmu=\od^k\mathbb{I}_T^{\od^k}\fmu$. 
\end{proof}

Immediately we have, for any $\fmu\in H\Lambda^k(\Omega)$, $\mathbb{I}_h^{\od^{k+1}}\od^k\fmu=\od^k_h\mathbb{I}_h^{\od^k}\fmu$, $0\leqslant k\leqslant n-1$.

We summarize all above to theorem below. 
\begin{theorem} 
The following de Rham complexes commute:
\begin{equation}\label{eq:cdnb}
	\begin{array}{ccccccccc}
	\mathbb{R} & ~~~\xrightarrow{inc}~~~ & H\Lambda^0 & ~~~\xrightarrow{\od^0}~~~ & H\Lambda^1 & ~~~\xrightarrow{\od^1}~~~ &  ... & ~~~\xrightarrow{\od^{n-1}}~~~ &  H\Lambda^n \\
	& & \downarrow \mathbb{I}^{\od^0}_h & & \downarrow \mathbb{I}^{\od^1}_h & &&& \downarrow \mathbb{I}^{\od^n}_h  \\
	\mathbb{R} & \xrightarrow{inc} & \mathbf{W}^{\rm nc}_h\Lambda^0 & \xrightarrow{\od^0_h} & \mathbf{W}^{\rm nc}_h\Lambda^1 & \xrightarrow{\od^1_h} & ... & ~~~\xrightarrow{\od^{n-1}_h}~~~ & \mathbf{W}^{\rm nc}_h\Lambda^n 
	\end{array};
\end{equation}

\begin{equation}\label{eq:cdwb}
	\begin{array}{ccccccccc}
 0 & ~~~\longrightarrow~~~ & H_0\Lambda^0 & ~~~\xrightarrow{\od^0}~~~ & H_0\Lambda^1 & ~~~\xrightarrow{\od^1}~~~ &  ... & ~~~\xrightarrow{\od^{n-1}}~~~ &  H_0\Lambda^n \\
	& & \downarrow \mathbb{I}^{\od^0}_h & & \downarrow \mathbb{I}^{\od^1}_h & &&& \downarrow \mathbb{I}^{\od^n}_h  \\
	0& \longrightarrow & \mathbf{W}^{\rm nc}_{h0}\Lambda^0 & \xrightarrow{\od^0_h} & \mathbf{W}^{\rm nc}_{h0}\Lambda^1 & \xrightarrow{\od^1_h} & ... & ~~~\xrightarrow{\od^{n-1}_h}~~~ & \mathbf{W}^{\rm nc}_{h0}\Lambda^n 
	\end{array}.
\end{equation}
\end{theorem}

\begin{remark}
Given Theorem \ref{thm:displd} the discrete Poincar\'e-Lefschetz duality, we are actually led to that, once one of the four complexes in \eqref{eq:cdnb} and \eqref{eq:cdwb} is exact, so are the three others. 
\end{remark}


%
%
%
\section{Discretization of the Hodge Laplace problem with nonconforming spaces for $H\Lambda^k$}
\label{sec:dishl}

In this section, we study the discretizations of the Hodge Laplace problem: given $\ff\in L^2\Lambda^k$, with $\mathbf{P}^k_{\hf}$ the $L^2$ projection to $\hf\Lambda^k$, find $\fomega\in H\Lambda^k(\Omega)\cap H^*_0\Lambda^k(\Omega)$ with $\od^k\fomega\in H^*_0\Lambda^{k+1}(\Omega)$, such that 
\begin{equation}
\fomega\perp\hf\Lambda^k(\Omega),\quad\mbox{and}\quad
\odelta_{k+1}\od^k\fomega+\od^{k-1}\odelta_k\fomega=\ff-\mathbf{P}^k_{\hf}\ff.
\end{equation}
The primal weak formulation is: find $\fomega\in H\Lambda^k\cap H^*_0\Lambda^k$, such that 
\begin{equation}\label{eq:modelhlori}
\left\{
\begin{array}{rll}
\langle\fomega,\fvarsigma\rangle_{L^2\Lambda^k}&=0, & \forall\,\fvarsigma\in \hf\Lambda^k,\ 
\\  
\langle\od^k\fomega,\od^k\fmu\rangle_{L^2\Lambda^{k+1}}+\langle\odelta_k\fomega,\odelta_k\fmu\rangle_{L^2\Lambda^{k-1}}&=\langle\ff-\mathbf{P}^k_{\hf}\ff,\fmu\rangle_{L^2\Lambda^k}, & \forall\,\fmu\in H\Lambda^k(\Omega)\cap H^*_0\Lambda^k(\Omega).
\end{array}\right.
\end{equation}
A standard mixed formulation based on $\fomega\in H\Lambda^k$ is generally used (\cite{Arnold.D2018feec}), which seeks $(\fomega^{\bf p},\fsigma^{\bf p},\fvartheta^{\bf p})\in H\Lambda^k \times H\Lambda^{k-1}\times\hf\Lambda^k$, such that, for $(\fmu,\ftau,\fvarsigma)\in H\Lambda^k \times H\Lambda^{k-1}\times\hf\Lambda^k$, 
\begin{equation}\label{eq:mixedhodgeclassicalp}
\left\{
\begin{array}{cccll}
&&\langle\fomega^{\bf p},\fvarsigma\rangle_{L^2\Lambda^k}&=0
\\
&\langle \fsigma^{\bf p},\ftau\rangle_{L^2\Lambda^{k+1}}&-\langle\fomega^{\bf p},\od^{k-1}\ftau\rangle_{L^2\Lambda^k}&=0
\\
\langle\fvartheta^{\bf p},\fmu\rangle_{L^2\Lambda^k}&+\langle \od^{k-1}\fsigma^{\bf p},\fmu\rangle_{L^2\Lambda^k}&+\langle\od^k\fomega^{\bf p},\od^k\fmu\rangle_{L^2\Lambda^{k-1}}&=\langle \ff,\fmu\rangle_{L^2\Lambda^k}
\end{array}
\right..
\end{equation}
In this section, we investigate the application of the nonconforming finite element spaces to the discretizations of this classical formulation and to a new ``completely" mixed formulation.
\begin{remark}\label{rem:confdualmix}
Here we call \eqref{eq:mixedhodgeclassicalp} ``primal" mixed formulation, and use the supscript ${}^{\bf p}$ to label that. 

Actually, it is natural to set an auxiliary mixed formulation, which seeks $(\fomega^{\bf d},\fzeta^{\bf d},\fvartheta^{\bf d})\in H^*_0\Lambda^k \times H^*_0\Lambda^{k+1}\times\hf^*_0\Lambda^k$, such that, for $(\fmu,\feta,\fvarsigma)\in H^*_0\Lambda^k \times H^*_0\Lambda^{k+1}\times\hf^*_0\Lambda^k$,
\begin{equation}\label{eq:mixedhodgeclassicald}
\left\{
\begin{array}{cccll}
&&\langle\fomega^{\bf d},\fvarsigma\rangle_{L^2\Lambda^k}&=0
\\
&\langle \fzeta^{\bf d},\feta\rangle_{L^2\Lambda^{k+1}}&-\langle\fomega^{\bf d},\odelta_{k+1}\feta\rangle_{L^2\Lambda^k}&=0
\\
\langle\fvartheta^{\bf d},\fmu\rangle_{L^2\Lambda^k}&+\langle \odelta_{k+1}\fzeta^{\bf d},\fmu\rangle_{L^2\Lambda^k}&+\langle\odelta_k\fomega^{\bf d},\odelta_k\fmu\rangle_{L^2\Lambda^{k-1}}&=\langle \ff,\fmu\rangle_{L^2\Lambda^k}
\end{array}.
\right.
\end{equation}
This can be viewed as a mixed formulation as the dual of \eqref{eq:mixedhodgeclassicalp}.

Conforming finite elements have been used for discretization of \eqref{eq:mixedhodgeclassicalp}; they are naturally used for \eqref{eq:mixedhodgeclassicald}. For example, we can consider the discretization for \eqref{eq:mixedhodgeclassicald}: to find $(\fomega^{\bf d}_h,\fzeta^{\bf d}_h,\fvartheta^{\bf d}_h)\in \fW^*_{h0}\Lambda^k \times \fW^*_{h0}\Lambda^{k+1}\times\hf^*_{h0}\Lambda^k$, such that, for $(\fmu_h,\feta_h,\fvarsigma_h)\in \fW^*_{h0}\Lambda^k\times \fW^*_{h0}\Lambda^{k+1}\times\hf^*_{h0}\Lambda^k$,
\begin{equation}\label{eq:mixedhodgeclassicalddis}
\left\{
\begin{array}{cccll}
&&\langle\fomega^{\bf d},\fvarsigma\rangle_{L^2\Lambda^k}&=0
\\
&\langle \mathbf{P}_h^{k+1}\fzeta^{\bf d}_h, \mathbf{P}_h^{k+1}\feta_h\rangle_{L^2\Lambda^{k+1}}&-\langle\fomega^{\bf d}_h,\odelta_{k+1}\feta_h\rangle_{L^2\Lambda^k}&=0
\\
\langle\fvartheta^{\bf d}_h,\fmu_h\rangle_{L^2\Lambda^k}&+\langle \odelta_{k+1}\fzeta^{\bf d}_h,\fmu_h\rangle_{L^2\Lambda^k}&+\langle\odelta_k\fomega^{\bf d}_h,\odelta_k\fmu_h\rangle_{L^2\Lambda^{k-1}}&=\langle \ff,\mathbf{P}_h^k\fmu_h\rangle_{L^2\Lambda^k}
\end{array}.
\right.
\end{equation}
The well-posedness of \eqref{eq:mixedhodgeclassicalddis} is the same as that of \eqref{eq:mixedhodgeclassicalpdis} below. The convergence analysis of \eqref{eq:mixedhodgeclassicalddis} can be done in a classical way; precisely, denote by $(\bar\fomega^{\bf d}_h,\bar\fzeta^{\bf d}_h,\bar\fvartheta^{\bf d}_h)\in \fW^*_{h0}\Lambda^k \times \fW^*_{h0}\Lambda^{k+1}\times\hf^*_{h0}\Lambda^k$ and $(\tilde\fomega^{\bf d}_h,\tilde\fzeta^{\bf d}_h,\tilde\fvartheta^{\bf d}_h)\in \fW^*_{h0}\Lambda^k \times \fW^*_{h0}\Lambda^{k+1}\times\hf^*_{h0}\Lambda^k$ the respective solutions of the auxiliary problems
\begin{equation}
\left\{
\begin{array}{cccll}
&&\langle\bar\fomega^{\bf d},\fvarsigma\rangle_{L^2\Lambda^k}&=0
\\
&\langle \bar\fzeta^{\bf d}_h, \feta_h\rangle_{L^2\Lambda^{k+1}}&-\langle\bar\fomega^{\bf d}_h,\odelta_{k+1}\feta_h\rangle_{L^2\Lambda^k}&=0
\\
\langle\bar\fvartheta^{\bf d}_h,\fmu_h\rangle_{L^2\Lambda^k}&+\langle \odelta_{k+1}\bar\fzeta^{\bf d}_h,\fmu_h\rangle_{L^2\Lambda^k}&+\langle\odelta_k\bar\fomega^{\bf d}_h,\odelta_k\fmu_h\rangle_{L^2\Lambda^{k-1}}&=\langle \mathbf{P}_h^k\ff,\fmu_h\rangle_{L^2\Lambda^k}
\end{array},
\right.
\end{equation}
and
\begin{equation}
\left\{
\begin{array}{cccll}
&&\langle\tilde\fomega^{\bf d},\fvarsigma\rangle_{L^2\Lambda^k}&=0
\\
&\langle \tilde\fzeta^{\bf d}_h, \feta_h\rangle_{L^2\Lambda^{k+1}}&-\langle\tilde\fomega^{\bf d}_h,\odelta_{k+1}\feta_h\rangle_{L^2\Lambda^k}&=0
\\
\langle\tilde\fvartheta^{\bf d}_h,\fmu_h\rangle_{L^2\Lambda^k}&+\langle \odelta_{k+1}\tilde\fzeta^{\bf d}_h,\fmu_h\rangle_{L^2\Lambda^k}&+\langle\odelta_k\tilde\fomega^{\bf d}_h,\odelta_k\fmu_h\rangle_{L^2\Lambda^{k-1}}&=\langle \ff,\fmu_h\rangle_{L^2\Lambda^k}
\end{array}.
\right.
\end{equation}
It follows by standard procedure that 
$$
\|(\fomega^{\bf d}_h,\fzeta^{\bf d}_h,\fvartheta^{\bf d}_h)-(\bar\fomega^{\bf d}_h,\bar\fzeta^{\bf d}_h,\bar\fvartheta^{\bf d}_h)\|_{H^*\Lambda^k\times H^*\Lambda^{k+1}\times L^2\Lambda^k}\leqslant Ch\|\mathbf{P}_h^k\ff\|_{L^2\Lambda^k}\leqslant Ch\|\ff\|_{L^2\Lambda^k},
$$
and
$$
\|(\bar\fomega^{\bf d}_h,\bar\fzeta^{\bf d}_h,\bar\fvartheta^{\bf d}_h)-(\tilde\fomega^{\bf d}_h,\tilde\fzeta^{\bf d}_h,\tilde\fvartheta^{\bf d}_h)\|_{H^*\Lambda^k\times H^*\Lambda^{k+1}\times L^2\Lambda^k}\leqslant Ch\|\ff\|_{L^2\Lambda^k}.
$$
Meanwhile, the classical analysis (cf. \cite[Theorem 7.10 and its proof]{Arnold.D;Falk.R;Winther.R2006acta}) holds as
$$
\|(\fomega^{\bf d},\fzeta^{\bf d},\fvartheta^{\bf d})-(\tilde\fomega^{\bf d}_h,\tilde\fzeta^{\bf d}_h,\tilde\fvartheta^{\bf d}_h)\|_{H^*\Lambda^k\times H^*\Lambda^{k+1}\times L^2\Lambda^k}\leqslant Ch^s\|\ff\|_{L^2\Lambda^k},
$$
if the domain $\Omega$ is s-regular. The convergence analysis of \eqref{eq:mixedhodgeclassicalddis} then follows. 
\end{remark}

%
%
\subsection{Nonconforming discretization of \eqref{eq:mixedhodgeclassicalp}}
By the newly designed nonconforming finite element spaces, the discrete problem is: to find $(\fomega^{\bf p}_h,\fsigma^{\bf p}_h,\fvartheta^{\bf p}_h)\in \mathbf{W}^{\rm nc}_h\Lambda^k \times \mathbf{W}^{\rm nc}_h\Lambda^{k-1}\times\hf_h^{\rm nc}\Lambda^k$, such that, for $(\fmu_h,\ftau_h,\fvarsigma_h)\in\mathbf{W}^{\rm nc}_h\Lambda^k \times \mathbf{W}^{\rm nc}_h\Lambda^{k-1}\times\hf_h^{\rm nc}\Lambda^k$,
\begin{equation}\label{eq:mixedhodgeclassicalpdis}
\left\{
\begin{array}{cccll}
&&\langle\fomega^{\bf p}_h,\fvarsigma_h\rangle_{L^2\Lambda^k}&=0
\\
&\langle \mathbf{P}_h^{k-1}\fsigma^{\bf p}_h,\mathbf{P}_h^{k-1}\ftau_h\rangle_{L^2\Lambda^{k+1}}&-\langle\fomega^{\bf p}_h,\od^{k-1}_h\ftau_h\rangle_{L^2\Lambda^k}&=0
\\
\langle\fvartheta^{\bf p}_h,\fmu_h\rangle_{L^2\Lambda^k}&+\langle \od^{k-1}_h\fsigma^{\bf p}_h,\fmu_h\rangle_{L^2\Lambda^k}&+\langle\od^k_h\fomega^{\bf p}_h,\od^k_h\fmu_h\rangle_{L^2\Lambda^{k-1}}&=\langle \ff,\mathbf{P}_h^k\fmu_h\rangle_{L^2\Lambda^k}
\end{array}.
\right.
\end{equation}

To verify the well-posedness of \eqref{eq:mixedhodgeclassicalpdis}, following \cite[Section 4.2.2]{Arnold.D2018feec}, writing $X_h:=\mathbf{W}^{\rm nc}_h\Lambda^k \times \mathbf{W}^{\rm nc}_h\Lambda^{k-1}\times\hf_h^{\rm nc}\Lambda^k$, with $\|(\fmu_h,\ftau_h,\fvarsigma_h)\|_{X_h}:=\|\mu_h\|_{\od^k_h}+\|\ftau_h\|_{\od^{k-1}_h}+\|\fvarsigma_h\|_{L^2\Lambda^k}$, denoting on $X_h\times X_h$
\begin{multline}
B_h((\fomega_h,\fsigma_h,\fvartheta_h), (\fmu_h,\ftau_h,\fvarsigma_h)):=\langle \mathbf{P}_h^{k-1}\fsigma^{\bf p}_h,\mathbf{P}_h^{k-1}\ftau_h\rangle_{L^2\Lambda^{k+1}}-\langle\fomega^{\bf p}_h,\od^{k-1}_h\ftau_h\rangle_{L^2\Lambda^k}
\\
-\langle\fvartheta^{\bf p}_h,\fmu_h\rangle_{L^2\Lambda^k}-\langle \od^{k-1}_h\fsigma^{\bf p}_h,\fmu_h\rangle_{L^2\Lambda^k}-\langle\od^k_h\fomega^{\bf p}_h,\od^k_h\fmu_h\rangle_{L^2\Lambda^{k-1}}-\langle\fomega^{\bf p}_h,\fvarsigma_h\rangle_{L^2\Lambda^k},
\end{multline}
we show the uniform inf-sup condition that 
\begin{equation}\label{eq:infsupB}
\inf_{0\neq (\fomega_h,\fsigma_h,\fvartheta_h)\in X_h}\sup_{0\neq (\fmu_h,\ftau_h,\fvarsigma_h)\in X_h}\frac{B_h((\fomega_h,\fsigma_h,\fvartheta_h), (\fmu_h,\ftau_h,\fvarsigma_h))}{\|(\fomega_h,\fsigma_h,\fvartheta_h)\|_{X_h}\|(\fmu_h,\ftau_h,\fvarsigma_h)\|_{X_h}}\geqslant\gamma>0.
\end{equation}

Given $(\fomega_h,\fsigma_h,\fvartheta_h)\in X_h$, we can decompose orthogonally $\fomega_h=\od^{k-1}_h\fvarrho_h+\fomega^{\hf}_h+\fomega_h^{\lrcorner}$, with $\fvarrho_h\in \fW^{\rm nc}_h\Lambda^{k-1}$, $\fomega^{\hf}_h\in \hf^{\rm nc}_h\Lambda^k$, and $\fomega_h^{\lrcorner}$ orthogonal to $\N(\od^k_h,\fW^{\rm nc}_h\Lambda^k)$, such that, by the discrete Poincar\'e inequality (Theorem \ref{thm:piwabc}), $\|\fvarrho_h\|_{\od^{k-1}_h}\leqslant c_P\|\od^{k-1}_h\fvarrho_h\|_{L^2\Lambda^k}$. 

Now, set $\ftau_h=\fsigma_h-\frac{1}{c_P^2}\fvarrho_h$, $\fmu_h=-\fomega_h-\od^{k-1}_h\fsigma_h-\fvartheta_h$, and $\fvarsigma_h=-\fvartheta_h+\fomega_h^{\hf}$, then 
$$
\|(\fmu_h,\ftau_h,\fvarsigma_h)\|_{X_h}\leqslant C\|(\fomega_h,\fsigma_h,\fvartheta_h)\|_{X_h},
$$
and
\begin{multline*}
B_h((\fomega_h,\fsigma_h,\fvartheta_h), (\fmu_h,\ftau_h,\fvarsigma_h))=\|\mathbf{P}_h^{k-1}\fsigma_h\|_{L^2\Lambda^{k-1}}^2+\|\od^{k-1}_h\fsigma_h\|_{L^2\Lambda^k}^2+\|\od^k_h\fomega_h\|_{L^2\Lambda^{k+1}}^2
\\
+\|\fvartheta_h\|_{L^2\Lambda^k}^2+\|\fomega_h^{\hf}\|_{L^2\Lambda^k}^2+\frac{1}{c_P^2}\|\od^{k-1}_h\fvarrho_h\|_{L^2\Lambda^k}-\frac{1}{c_P^2}\langle\fsigma_h,\fvarrho_h\rangle_{L^2\Lambda^{k-1}}.
\end{multline*}
Note further that 
\begin{multline*}
\langle\fsigma_h,\fvarrho_h\rangle_{L^2\Lambda^{k-1}}\leqslant \|\fsigma_h\|_{L^2\Lambda^{k-1}}\|\fvarrho_h\|_{L^2\Lambda^{k-1}}
\\
\leqslant \frac{c_P^2}{2}\|\fsigma_h\|_{L^2\Lambda^{k-1}}^2+\frac{1}{2c_P^2}\|\fvarrho_h\|_{L^2\Lambda^{k-1}}^2\leqslant \frac{c_P^2}{2}\|\fsigma_h\|_{L^2\Lambda^{k-1}}^2+\frac{1}{2}\|\od^{k-1}_h\fvarrho_h\|_{L^2\Lambda^k}^2
\end{multline*}
Thus
\begin{multline*}
B_h((\fomega_h,\fsigma_h,\fvartheta_h), (\fmu_h,\ftau_h,\fvarsigma_h))\geqslant \|\mathbf{P}_h^{k-1}\fsigma_h\|_{L^2\Lambda^{k-1}}^2+\|\od^{k-1}_h\fsigma_h\|_{L^2\Lambda^k}^2-\frac{1}{2}\|\fsigma_h\|_{L^2\Lambda^{k-1}}^2+\|\od^k_h\fomega_h\|_{L^2\Lambda^{k+1}}^2
\\
+\|\fvartheta_h\|_{L^2\Lambda^k}^2+\|\fomega_h^{\hf}\|_{L^2\Lambda^k}^2+\frac{1}{c_P^2}\|\od^{k-1}_h\fvarrho_h\|_{L^2\Lambda^k}
\\
\geqslant  \frac{1}{2}\|\fsigma_h\|_{L^2\Lambda^{k-1}}^2+(1-Ch^2)\|\od^{k-1}_h\fsigma_h\|_{L^2\Lambda^k}^2+\|\od^k_h\fomega_h\|_{L^2\Lambda^{k+1}}^2
+\|\fvartheta_h\|_{L^2\Lambda^k}^2+\|\fomega_h^{\hf}\|_{L^2\Lambda^k}^2+\frac{1}{c_P^2}\|\od^{k-1}_h\fvarrho_h\|_{L^2\Lambda^k}.
\end{multline*}
Note that $\od^k_h\fomega^{\lrcorner}_h=\od^k_h\fomega_h$, and, by Theorem \ref{thm:piwabc}, $\|\fomega^{\lrcorner}_h\|_{L^2\Lambda^k}\leqslant C\|\od^k_h\fomega_h\|_{L^2\Lambda^{k+1}}$. It follows then
$$
B_h((\fomega_h,\fsigma_h,\fvartheta_h), (\fmu_h,\ftau_h,\fvarsigma_h))\geqslant C\|(\fomega_h,\fsigma_h,\fvartheta_h)\|_{X_h}^2,
$$
with $C$ depending on the Poincar\'e inequality only. The inf-sup condition \eqref{eq:infsupB} is then proved and the well-posedness of \eqref{eq:mixedhodgeclassicalpdis} is verified.

\subsection{A novel mixed element scheme}

It is natural to consider an approach where both $\od^k$ and $\odelta_k$ are operated in a dual way, and we begin with this ``completely" mixed formulation: to find $(\fomega^{\rm c},\fzeta^{\rm c},\fsigma^{\rm c},\fvartheta^{\rm c})\in L^2\Lambda^k\times H^*_0\Lambda^{k+1}\times H\Lambda^{k-1}\times\hf\Lambda^k$, such that, for $(\fmu,\feta,\ftau,\fvarsigma)\in L^2\Lambda^k\times H^*_0\Lambda^{k+1}\times H\Lambda^{k-1}\times\hf\Lambda^k$,
\begin{equation}\label{eq:mixedhodge}
\left\{
\begin{array}{ccccll}
&&&\langle\fomega^{\rm c},\fvarsigma\rangle_{L^2\Lambda^k}&=0
\\
&\langle \fzeta^{\rm c},\feta\rangle_{L^2\Lambda^{k+1}}&&-\langle\fomega^{\rm c},\odelta_{k+1}\feta\rangle_{L^2\Lambda^k}&=0
\\
&&\langle \fsigma^{\rm c},\ftau\rangle_{L^2\Lambda^{k-1}}&-\langle\fomega^{\rm c},\od^{k-1}\ftau\rangle_{L^2\Lambda^k}&=0
\\
\langle\fvartheta^{\rm c},\fmu\rangle_{L^2\Lambda^k}&+\langle \odelta_{k+1}\fzeta^{\rm c},\fmu\rangle_{L^2\Lambda^k}&+\langle \od^{k-1}\fsigma^{\rm c},\fmu \rangle_{L^2\Lambda^k}&&=\langle \ff,\fmu\rangle_{L^2\Lambda^k}
\end{array}.
\right.
\end{equation}
\begin{lemma}
For $\ff\in L^2\Lambda^k$, the problem \eqref{eq:mixedhodge} admits a unique solution $(\fomega^{\rm c},\fzeta^{\rm c},\fsigma^{\rm c},\fvartheta^{\rm c})$, and 
\begin{equation}\label{eq:stabcmhl}
\|\fomega^{\rm c}\|_{L^2\Lambda^k}+\|\fzeta^{\rm c}\|_{\odelta_{k+1}}+\|\fsigma^{\rm c}\|_{\od^{k-1}}+\|\fvartheta^{\rm c}\|_{L^2\Lambda^k}\leqslant C\|\ff\|_{L^2\Lambda^k}.
\end{equation}
Further, $\fzeta^{\rm c}=\od^k\fomega^{\rm c}$, $\fsigma^{\rm c}=\odelta_k\fomega^{\rm c}$, and $\fomega^{\rm c}$ solves \eqref{eq:modelhlori}.
\end{lemma}
\begin{proof}
~~For \eqref{eq:stabcmhl}, we only have to verify Brezzi's conditions, which hold by the orthogonal Hodge decomposition 
$$
L^2\Lambda^k=\R(\od^{k-1},H\Lambda^{k-1})\opp\hf\Lambda^k\opp\R(\odelta_{k+1},H^*_0\Lambda^{k+1}),
$$
together with the closeness of $\mathcal{R}(\od^{k-1},H\Lambda^{k-1})$ and $\mathcal{R}(\odelta_{k+1},H^*_0\Lambda^{k+1})$. The remaining assertions are straightforward. The proof is completed. 
\end{proof}

A lowest-degree stable discretization of \eqref{eq:mixedhodge} is: find $(\fomega^{\rm c}_h,\fzeta^{\rm c}_h,\fsigma^{\rm c}_h,\fvartheta^{\rm c}_h)\in \mathcal{P}_0\Lambda^k(\mathcal{G}_h)\times \mathbf{W}^*_{h0}\Lambda^{k+1}\times \mathbf{W}^{\rm nc}_h\Lambda^{k-1}\times\hf^{\rm nc}_h\Lambda^k$, such that, for $(\fmu_h,\feta_h,\ftau_h,\fvarsigma_h)\in \mathcal{P}_0\Lambda^k(\mathcal{G}_h)\times \mathbf{W}^*_{h0}\Lambda^{k+1}\times \mathbf{W}^{\rm nc}_h\Lambda^{k-1}\times\hf^{\rm nc}_h\Lambda^k$, 
\begin{equation}\label{eq:mixedhodgedis}
\left\{
\begin{array}{ccccl}
&&&\langle\fomega^{\rm c}_h,\fvarsigma_h\rangle_{L^2\Lambda^k}&=0
\\
&\langle \mathbf{P}_h^{k+1}\fzeta^{\rm c}_h,\mathbf{P}_h^{k+1}\feta_h\rangle_{L^2\Lambda^{k+1}}&&-\langle\fomega^{\rm c}_h,\odelta_{k+1}\feta_h\rangle_{L^2\Lambda^k}&=0
\\
&&\langle \mathbf{P}_h^{k-1}\fsigma^{\rm c}_h,\mathbf{P}_h^{k-1}\ftau_h\rangle_{L^2\Lambda^{k-1}}&-\langle\fomega^{\rm c}_h,\od^{k-1}_h\ftau_h\rangle_{L^2\Lambda^k}&=0
\\
\langle\fvartheta^{\rm c}_h,\fmu_h\rangle_{L^2\Lambda^k}&+\langle \odelta_{k+1}\fzeta^{\rm c}_h,\fmu_h\rangle_{L^2\Lambda^k}&+\langle \od^{k-1}_h\fsigma^{\rm c}_h,\fmu_h \rangle_{L^2\Lambda^k}&&=\langle \ff,\fmu_h\rangle_{L^2\Lambda^k}
\end{array}.
\right.
\end{equation}

\begin{lemma}\label{lem:stabcomphodgelap}
Given $\ff\in L^2\Lambda^k$, the problem \eqref{eq:mixedhodgedis} admits a unique solution $(\fomega^{\rm c}_h,\fzeta^{\rm c}_h,\fsigma^{\rm c}_h,\fvartheta^{\rm c}_h)$, and 
$$
\|\fomega^{\rm c}_h\|_{L^2\Lambda^k}+\|\fzeta^{\rm c}_h\|_{\odelta_{k+1}}+\|\fsigma^{\rm c}_h\|_{\od^{k-1}_h}+\|\fvartheta^{\rm c}_h\|_{L^2\Lambda^k}\leqslant C\|f\|_{L^2\Lambda^k}.
$$
The constant $C$ depends on $\icr(\odelta_{k+1},\mathbf{W}_{h0}^*\Lambda^{k+1})$ and $\icr(\od^{k-1}_h,\mathbf{W}^{\rm nc}_h\Lambda^{k-1})$. 
\end{lemma}
Again, for the well-posedness of \eqref{eq:mixedhodgedis}, we only have to verify Brezzi's conditions, which holds by the discrete Hodge decomposition \eqref{eq:dishgd}. The stable decompositions \eqref{eq:dishgd} comes true by the aid of the nonconforming space $\fW^{\rm nc}_h\Lambda^k$. Hence \eqref{eq:mixedhodgedis} is a new scheme hinted in nonconforming finite element exterior calculus. 

%
%
\subsection{Equivalences among lowest-degree mixed element schemes}

\begin{lemma}\label{lem:hlmprimaldual}
Let $(\fomega^{\rm c}_h,\fzeta^{\rm c}_h,\fsigma^{\rm c}_h,\fvartheta^{\rm c}_h)$, $(\fomega^{\bf p}_h,\fsigma^{\bf p}_h,\fvartheta^{\bf p}_h)$ and $(\fomega^{\bf d}_h,\fzeta^{\bf d}_h,\fvartheta^{\bf d}_h)$ be the solutions of \eqref{eq:mixedhodgedis}, \eqref{eq:mixedhodgeclassicalpdis} and \eqref{eq:mixedhodgeclassicalddis}, respectively. Then
\begin{equation}\label{eq:completevsdual}
\fvartheta^{\bf d}_h=\fvartheta^{\rm c}_h,\ \fzeta^{\bf d}_h=\fzeta^{\rm c}_h,\ \mathbf{P}_h^k\fomega^{\bf d}_h=\fomega^{\rm c}_h,\ \odelta_k\fomega^{\bf d}_h=\mathbf{P}_h^{k-1}\fsigma^{\rm c}_h,\ \odelta_{k+1} \fzeta^{\bf d}_h=\mathbf{P}_h^k\ff-\od^{k-1}_h \fsigma^{\rm c}_h-\fvartheta^{\rm c}_h,
\end{equation}
\begin{equation}\label{eq:completevsprimal}
\fvartheta^{\bf p}_h=\fvartheta^{\rm c}_h,\  \fsigma^{\bf p}_h=\fsigma^{\rm c}_h,\ \mathbf{P}_h^k\fomega^{\bf p}_h=\fomega^{\rm c}_h,\  \od^k_h\fomega^{\bf p}_h=\mathbf{P}_h^{k+1}\fzeta^{\rm c}_h, \ \od^{k-1}_h\fsigma^{\bf p}_h=\mathbf{P}_h^k\ff-\odelta_{k+1} \fzeta^{\rm c}_h-\fvartheta^{\rm c}_h,
\end{equation}
\begin{equation}\label{eq:primalvsdual}
\fvartheta^{\bf d}_h=\fvartheta^{\bf p}_h,\ \mathbf{P}_h^{k+1}\fzeta^{\bf d}_h=\od^k_h\fomega^{\bf p}_h,\ \mathbf{P}_h^k\fomega^{\bf d}_h= \mathbf{P}_h^k\fomega^{\bf p}_h,\ \odelta_k\fomega^{\bf d}_h=\mathbf{P}_h^{k-1}\fsigma^{\bf p}_h,\ \odelta_{k+1} \fzeta^{\bf d}_h+\od^{k-1}_h\fsigma^{\bf p}_h=\mathbf{P}_h^k\ff-\fvartheta^{\rm c}_h.
\end{equation}
\end{lemma}

\begin{proof}
~~Let $(\fomega^{\bf d}_h, \fzeta^{\bf d}_h, \fvartheta^{\bf d}_h)$ be the solution of \eqref{eq:mixedhodgeclassicalddis}. Then, with a $ \overline{\fsigma}_h\in \mathbf{W}_h^{\rm nc}\Lambda^{k-1}$, 
\begin{multline*}
\qquad
\langle \fvartheta^{\bf d}_h,\fmu_h\rangle_{L^2\Lambda^k}+\langle \odelta_{k+1} \fzeta^{\bf d}_h,\fmu_h\rangle_{L^2\Lambda^k}+\langle\odelta_k \fomega^{\bf d}_h,\odelta_k\fmu_h\rangle_{L^2\Lambda^{k-1}}
\\
+\langle\od^{k-1}_h  \overline{\fsigma}_h,\fmu_h\rangle_{L^2\Lambda^k}-\langle  \overline{\fsigma}_h,\odelta_k\fmu_h\rangle_{L^2\Lambda^{k-1}}=\langle \ff,\mathbf{P}_h^k\fmu_h\rangle_{L^2\Lambda^k},\qquad
\end{multline*}
for any $\fmu_h\in \mathcal{P}^{*,-}_1\Lambda^{k}(\mathcal{G}_h)$. Choosing arbitrarily $\fmu_h\in \mathcal{P}_0\Lambda^k(\mathcal{G}_h)$, we have
\begin{equation}\label{eq:geometricequation}
\fvartheta^{\bf d}_h+\odelta_{k+1} \fzeta^{\bf d}_h+\od^{k-1}_h\overline{\fsigma}_h=\mathbf{P}_h^k\ff,
\end{equation}
and
\begin{equation*}
\langle\odelta_k \fomega^{\bf d}_h,\odelta_k\fmu_h\rangle_{L^2\Lambda^{k-1}}-\langle  \overline{\fsigma}_h,\odelta_k\fmu_h\rangle_{L^2\Lambda^{k-1}}=0, \ \ \forall\,\fmu_h\in \mathcal{P}^{*,-}_1\Lambda^k(\mathcal{G}_h),
\end{equation*}
which leads to that $\odelta_k \fomega^{\bf d}_h=\mathbf{P}_h^{k-1} \overline{\fsigma}_h$. Further, noting that $\langle\odelta_k \fomega^{\bf d}_h,\ftau_h\rangle_{L^2\Lambda^{k-1}}=\langle  \fomega^{\bf d}_h, \od^k_h\ftau_h\rangle_{L^2\Lambda^k}$ for $\ftau_h\in \mathbf{W}_h^{\rm nc}\Lambda^{k-1}$, we obtain $\langle \mathbf{P}_h^{k-1} \overline{\fsigma}_h, \mathbf{P}_h^{k-1} \ftau_h^{\bf d}\rangle_{L^2\Lambda^{k-1}}-\langle  \fomega^{\bf d}_h, \od^{k-1}_h\ftau_h\rangle_{L^2\Lambda^k}=0$
for $\ftau_h\in \mathbf{W}_h^{\rm nc}\Lambda^{k-1}$.

In all, $(\mathbf{P}_h^k \fomega^{\bf d}_h, \fzeta^{\bf d}_h, \overline{\fsigma}_h, \fvartheta^{\bf d}_h)\in \mathcal{P}_0\Lambda^k(\mathcal{G}_h)\times \fW^*_{h0}\Lambda^{k+1}\times \mathbf{W}^{\rm nc}_h\Lambda^{k-1}\times\hf^{\rm nc}_h\Lambda^k$ satisfies the system \eqref{eq:mixedhodgedis}, and thus $(\mathbf{P}_h^k \fomega^{\bf d}_h, \fzeta^{\bf d}_h, \overline{\fsigma}_h, \fvartheta^{\bf d}_h)=(\fomega^{\rm c}_h,\fzeta^{\rm c}_h,\fsigma^{\rm c}_h,\fvartheta^{\rm c}_h)$. This proves \eqref{eq:completevsdual}. Similarly can \eqref{eq:completevsprimal} be proved, and \eqref{eq:primalvsdual} follows by \eqref{eq:completevsdual} and \eqref{eq:completevsprimal}. The proof is completed. 
\end{proof}
The convergence analysis of \eqref{eq:mixedhodgedis} and \eqref{eq:mixedhodgeclassicalpdis} follow directly by Remark \ref{rem:confdualmix} and Lemma \ref{lem:hlmprimaldual}, and we omit the details here.  

\subsection{A decomposition processes for solving \eqref{eq:mixedhodgedis}}

Firstly, we decomposition \eqref{eq:mixedhodgedis} to two subsystems.

\begin{lemma}\label{lem:nohf}
Let $(\fomega^{\rm c}_h,\fzeta^{\rm c}_h,\fsigma^{\rm c}_h,\fvartheta^{\rm c}_h)$ be the solution of \eqref{eq:mixedhodgedis}, let $\fzeta_h\ \mbox{and}\ \fvarphi_h\in \fW^*_{h0}\Lambda^{k+1}$ be such that, for any $\feta_h\ \mbox{and}\ \fpsi_h\in \fW^*_{h0}\Lambda^{k+1}$,
\begin{equation}\label{eq:cm-4-sub-1}
\left\{
\begin{array}{ccll}
\langle \mathbf{P}_h^{k+1}\fzeta^{\rm c}_h,\mathbf{P}_h^{k+1}\feta_h\rangle_{L^2\Lambda^{k+1}}&-\langle\odelta_{k+1}\fvarphi_h,\odelta_{k+1}\feta_h\rangle_{L^2\Lambda^k}&=0
\\
\langle \odelta_{k+1}\fzeta^{\rm c}_h,\odelta_{k+1}\fpsi_h\rangle_{L^2\Lambda^k}&&=\langle \ff,\odelta_{k+1}\fpsi_h\rangle_{L^2\Lambda^k}
\end{array},
\right.
\end{equation}
and let $\fsigma_h\ \mbox{and}\ \fvarrho_h\in \fW^{\rm nc}_h\Lambda^{k-1}$ be such that, for any $\ftau_h\ \mbox{and}\ \fvarpi_h\in \fW^{\rm nc}_h\Lambda^{k-1}$,
\begin{equation}\label{eq:cm-4-sub-2}
\left\{
\begin{array}{ccll}
\langle \mathbf{P}_h^{k-1}\fsigma^{\rm c}_h,\mathbf{P}_h^{k-1}\ftau_h\rangle_{L^2\Lambda^{k-1}}&-\langle\od^{k-1}_h\fvarrho_h,\od^{k-1}_h\ftau_h\rangle_{L^2\Lambda^k}&=0
\\
\langle \od^{k-1}_h\fsigma^{\rm c}_h,\od^{k-1}_h\fvarpi_h \rangle_{L^2\Lambda^k}&&=\langle \ff,\od^{k-1}_h\fvarpi_h\rangle_{L^2\Lambda^k}
\end{array}.
\right.
\end{equation}
Then 
\begin{equation}\label{eq:decomcompmix}
\fzeta^{\rm c}_h=\fzeta_h, \ \fsigma^{\rm c}_h=\fsigma_h,\ \mbox{and}\ \fomega^{\rm c}_h=\od^{k-1}_h\fvarrho_h+\odelta_{k+1}\fvarphi_h. 
\end{equation}
\end{lemma}
\begin{proof}
~~The existence of solutions to \eqref{eq:cm-4-sub-1} and \eqref{eq:cm-4-sub-2} is easy to verify, where $\fzeta_h$ and $\fsigma_h$ are uniquely determined, $\fvarphi_h$ is uniquely determined up to $\N(\odelta_{k+1},\fW^*_{h0}\Lambda^{k+1})$, and $\fvarrho_h$ is uniquely determined up to $\N(\od^{k-1}_h,\fW^{\rm nc}_h\Lambda^{k-1})$. By the Hodge decomposition of $\mathcal{P}_0\Lambda^k(\mathcal{G}_h)$, we can decompose $\fomega_h\in \mathcal{P}_0\Lambda^k(\mathcal{G}_h)$ to $\fomega^{\rm c}_h=\fiota^{\rm c}_h+\odelta_{k+1}\fvarphi^{\rm c}_h+\od^{k-1}_h\fvarrho^{\rm c}_h$ with $\fiota^{\rm c}_h\in \hf_h\Lambda^k$, $\fvarphi^{\rm c}_h\in \fW^*_{h0}\Lambda^{k+1}$ and $\fvarrho^{\rm c}_h\in\fW^{\rm nc}_h\Lambda^{k-1}$, and $\od^{k-1}_h\fvarrho^{\rm c}_h$ and $\odelta_{k+1}\fvarphi^{\rm c}_h$ are uniquely determined. We can similarly write $\fmu_h=\fchi_h+\odelta_{k+1}\fpsi_h+\od^{k-1}_h\fvarpi_h$. Substituting the decompositions of $\fomega^{\rm c}_h$ and $\fmu_h$ into \eqref{eq:mixedhodgedis} leads to subsystems \eqref{eq:cm-4-sub-1} and \eqref{eq:cm-4-sub-2}, and further \eqref{eq:decomcompmix}.  
\end{proof}

Noting that both \eqref{eq:cm-4-sub-1} and \eqref{eq:cm-4-sub-2} are each a saddle problem whose solution is not unique, we are now to further decompose them to series of semi positive definite problems to solve. 

\begin{lemma}
Let $(\fzeta_h^{\lrcorner},\fxi^{\lrcorner}_h,\fvarphi^{\lrcorner}_h)$ be a solution of the sequence of problems below:
\begin{enumerate}
\item find $\fzeta_h^{\lrcorner}\in \fW^*_{h0}\Lambda^{k+1}$, such that 
\begin{equation}
\langle \odelta_{k+1}\fzeta^{\lrcorner}_h,\odelta_{k+1}\fpsi_h\rangle_{L^2\Lambda^k}=\langle \ff,\odelta_{k+1}\fpsi_h\rangle_{L^2\Lambda^k},\ \ \forall\,\fpsi_h\in \fW^*_{h0}\Lambda^{k+1};
\end{equation}
\item find $\fxi_h^{\lrcorner}\in \fW^{\rm nc}_h\Lambda^k$, such that 
\begin{equation}
\langle\od^k_h\fxi^{\lrcorner}_h,\od^k_h\fnu_h\rangle_{L^2\Lambda^{k+1}}=\langle\odelta_{k+1}\fzeta^{\lrcorner}_h,\fnu_h\rangle_{L^2\Lambda^k},\ \ \forall\,\fnu_h\in \fW^{\rm nc}_h\Lambda^k;
\end{equation}
\item find $\fvarphi^{\lrcorner}_h\in\fW^*_{h0}\Lambda^{k+1}$, such that 
\begin{equation}
\langle \odelta_{k+1}\fvarphi^{\lrcorner}_h,\odelta_{k+1}\feta_h\rangle_{L^2\Lambda^k}= \langle \od^k_h\fxi^{\lrcorner}_h,\feta_h\rangle_{L^2\Lambda^{k+1}},\ \ \forall\,\feta_h\in \fW^*_{h0}\Lambda^{k+1}.
\end{equation}
\end{enumerate}
Let $(\fzeta_h,\fvarphi_h)$ be a solution of \eqref{eq:cm-4-sub-1}. Then
\begin{equation}
\fzeta_h=\frac{(-1)^{nk}}{n-k}\okappa_h^{\odelta}(\odelta_{k+1}\fzeta_h^{\lrcorner})+\od^k_h\fxi_h^{\lrcorner},\ \ \ \mbox{and}\ \ \odelta_{k+1}\fvarphi_h=\odelta_{k+1}\fvarphi_h^{\lrcorner}. 
\end{equation}
\end{lemma}
\begin{proof}
~~ Evidently, $(\fzeta_h^{\lrcorner},\fxi^{\lrcorner}_h,\fvarphi^{\lrcorner}_h)$ exists and is unique up to $\N(\odelta_{k+1},\fW^*_{h0}\Lambda^{k+1})\times \N(\od^k_h,\fW^{\rm nc}\Lambda^k)\times \N(\odelta_{k+1},\fW^*_{h0}\Lambda^{k+1})$, further $\odelta_{k+1}\fzeta_h^{\lrcorner}=\odelta_{k+1}\fzeta_h$. Since $\langle\mathbf{P}_h^{k+1}\fzeta_h,\mathbf{P}_h^{k+1}\feta_h\rangle_{L^2\Lambda^{k+1}}=\langle \odelta_{k+1}\fxi_h,\odelta_{k+1}\feta_h\rangle_{L^2\Lambda^k}$ for any $\feta_h\in\fW^*_h\Lambda^{k+1},$ it holds that $\mathbf{P}_h^{k+1}\fzeta_h$ is orthogonal to $\N(\odelta_{k+1},\fW^*_{h0}\Lambda^{k+1})$, and thus $\mathbf{P}_h^{k+1}\fzeta_h\in\R(\od^k_h,\fW^{\rm nc}_h\Lambda^k)$. Namely, there exists a $\fxi^{\lrcorner}_h\in \fW^{\rm nc}_h\Lambda^k$, such that $\fzeta_h=(\fzeta_h-\mathbf{P}_h^{k+1}\fzeta_h)+\od^k_h\fxi^{\lrcorner}_h$. As for any $\fnu_h\in \fW^{\rm nc}_h\Lambda^k$, $\langle\odelta_{k+1}\fzeta_h,\fnu_h\rangle_{L^2\Lambda^k}=\langle\fzeta_h,\od^k_h\fnu_h\rangle_{L^2\Lambda^{k+1}}$, it holds  further that, with $\od^k_h\fnu_h$ being piecewise constant, $\langle\od^k_h\fxi^{\lrcorner}_h,\od^k_h\fnu_h\rangle_{L^2\Lambda^{k+1}}=\langle\odelta_{k+1}\fzeta^{\lrcorner}_h,\fnu_h\rangle_{L^2\Lambda^k}$. 
It follows by the homotopy formula that $\fzeta_h=\frac{(-1)^{nk}}{n-k}\okappa_h^{\odelta}(\odelta_{k+1}\fzeta_h^{\lrcorner})+\od^k_h\fxi_h^{\lrcorner}.$ 
Then $\langle\odelta_{k+1}\fvarphi_h,\odelta_{k+1}\feta_h\rangle_{L^2\Lambda^k}=\langle \mathbf{P}_h^{k+1}\fzeta^{\rm c}_h,\mathbf{P}_h^{k+1}\feta_h\rangle_{L^2\Lambda^{k+1}}=\langle \od^k_h\fxi^{\lrcorner}_h,\feta_h\rangle_{L^2\Lambda^{k+1}}$ for $\feta_h\in \fW^*_{h0}\Lambda^{k+1}$, and it thus follows that $\odelta_{k+1}\fvarphi_h=\odelta_{k+1}\fvarphi_h^{\lrcorner}.$ The proof is completed. 
\end{proof}
Similarly we have the decomposition of \eqref{eq:cm-4-sub-2}.
\begin{lemma}
Let $(\fsigma_h^{\lrcorner},\fiota^{\lrcorner}_h,\fvarrho^{\lrcorner}_h)$ be a solution of the sequence of problems below:
\begin{enumerate}
\item find $\fsigma_h^{\lrcorner}\in \fW^{\rm nc}_h\Lambda^{k-1}$, such that 
\begin{equation}
\langle \od^{k-1}_h\fsigma^{\lrcorner}_h,\od^{k-1}_h\fvarpi_h\rangle_{L^2\Lambda^k}=\langle \ff,\od^{k-1}_h\fvarpi_h\rangle_{L^2\Lambda^k},\ \ \forall\,\fvarpi_h\in \fW^{\rm nc}_h\Lambda^{k-1};
\end{equation}
\item find $\fiota_h^{\lrcorner}\in \fW^*_{h0}\Lambda^k$, such that 
\begin{equation}
\langle\odelta_k\fiota^{\lrcorner}_h, \odelta_k\fchi_h\rangle_{L^2\Lambda^{k-1}}=\langle\od^{k-1}\fsigma^{\lrcorner}_h,\fchi_h\rangle_{L^2\Lambda^k},\ \ \forall\,\fchi_h\in \fW^*_{h0}\Lambda^k;
\end{equation}
\item find $\fvarrho^{\lrcorner}_h\in\fW^{\rm nc}_h\Lambda^{k-1}$, such that 
\begin{equation}
\langle \od^{k-1}_h\fvarrho^{\lrcorner}_h,\od^{k-1}_h\ftau_h\rangle_{L^2\Lambda^k}= \langle \odelta_k\fiota^{\lrcorner}_h,\ftau_h\rangle_{L^2\Lambda^{k-1}},\ \ \forall\,\ftau_h\in \fW^{\rm nc}_h\Lambda^{k-1}.
\end{equation}
\end{enumerate}
Let $(\fsigma_h,\fvarrho_h)$ be one solution of \eqref{eq:cm-4-sub-2}. Then
\begin{equation}
\fsigma_h=\frac{1}{k}\okappa_h(\od^{k-1}_h\fsigma_h^{\lrcorner})+\odelta_k\fiota_h^{\lrcorner},\ \ \ \mbox{and}\ \ \od^{k-1}_h\fsigma_h=\od^{k-1}_h\fsigma_h^{\lrcorner}. 
\end{equation}
\end{lemma}

\begin{remark}
It is illustrated that the system \eqref{eq:mixedhodgedis}, as well as \eqref{eq:mixedhodgeclassicalddis} and \eqref{eq:mixedhodgeclassicalp}, can be transferred to a series of semi positive definite problems to solve. Particularly, these systems can be solved without knowledge of $\hf_h\Lambda^k$, which consists of globally supported functions and which cannot generally be figured out. A decomposition similar to Lemma \ref{lem:nohf} can be carried out onto \eqref{eq:mixedhodgeclassicalddis} without the aid of $\fW^{\rm nc}_h\Lambda^k$ and onto \eqref{eq:mixedhodgeclassicalpdis} without the aid of $\fW^*_{h0}\Lambda^k$. However, the further decomposition of \eqref{eq:cm-4-sub-1} and \eqref{eq:cm-4-sub-2} will rely on the combinational utilization of $\fW^{\rm nc}_h\Lambda^k$ and $\fW^*_{h0}\Lambda^k$ together.
\end{remark}


%
%
%
\section{Concluding remarks}
\label{sec:conc}

This paper presents a unified construction of finite element spaces for $H\Lambda^k$ in $\rn$, extending the Crouzeix-Raviart paradigm to differential forms. Beyond error estimation as usual, differences from existing classical schemes are preliminarily demonstrated using eigenvalue problems as examples, and can be further investigated through additional applications, for instance where a locally defined stable interpolator matters. Actually, the role of a locally-defined stable interpolator used to be illustrated by the correct computation of the convex variational problems \cite{Ortner.C2011ima}. A new way to impose inter-cell continuity is indicated, and finite element spaces can be constructed in future for various problems by this new approach, as well as on non-simplicial meshes. This approach suggests potential extensions to nonstandard and nonconforming meshes which will be discussed in future. This paper focuses on pure Dirichlet and pure Neumann boundary conditions. It is noteworthy that mixed boundary conditions have recently been investigated in \cite{Licht.M2019mc,christiansen2020poincare,Licht.M2017FoCM}. The new approach also works for that and can be discussed in future. Relevant to the equivalences established in \cite{Marini.L1985sinum} between the Crouzeix-Raviart element discretization and the Raviart-Thomas element discretization for Poisson equations, the equivalence between the conforming and nonconforming finite element schemes on the Hodge Laplace problem in Section \ref{sec:dishl} is the generalization of \cite{Marini.L1985sinum} with new interpretations. 
~\\

Within classical FEEC theory, discrete Hodge decompositions for $H\Lambda^k$ spaces are established, allowing for in-contractible domains, as demonstrated in (5.6) of \cite{Arnold.D2018feec}, which reads $V^k_h=\mathcal{B}^k_h\omp \mathfrak{H}^k_h\omp\mathcal{B}^*_{k,h}$, where $\mathcal{B}^*_{k,h}$ is the range of a {\bf globally} defined operator $d^*_{jh}$. These decompositions can be rebuilt based on $\fW^{\rm nc}_h\Lambda^k$. Beyond this, the theory of nonconforming finite element exterior calculus contains discrete Helmholtz decompositions and Hodge decompositions for piecewise constant $k$-forms corresponding to $L^2\Lambda^k$ for both contractible and in-contractible domains. Notably, the Hodge decompositions presented in this paper differ from those in \cite{Arnold.D2018feec} in that all discrete operators involved are {\bf locally} defined, i.e., cell by cell. In other words, the discrete derivative and coderivative operators are both local. Inspired by \cite{Lee.J;Winther.R2018local}, discretization scheme for the Hodge Laplace problem with local derivatives and local coderivatives will be studied in future. Recently, in two and three dimensions, discrete Helmholtz decompositions have been explored not only for piecewise constant but also for piecewise affine vector and tensor fields \cite{Bringmann.P;Ketteler.J;Schedensack.M2024}; it is intriguing to observe that the non-Ciarlet type finite element spaces of \cite{Fortin.M;Soulie.M1983,Zhang.S2021SCM} have been utilized as a basis therein. The generalization of the results presented in this manuscript to higher-degree vector and tensor fields in higher dimensions will be discussed in future.
~\\

In Section \ref{subsec:decoms}, a reciprocal causation \eqref{eq:complexduality} between two discrete complexes is presented. This can be viewed a discrete analogue of the dual complexes composed by adjoint operator pairs (namely $\od$ and $\odelta$), defined in Section 4.1.2 of \cite{Arnold.D2018feec}. We note that kinds of dualities used to be studied in, e.g., \cite{Arnold.D;Falk.R;Winther.R2009CMAME,Berchenko.Y2021duality,Buffa.A;Christiansen.S2007dual,Dlotko.P;Specogna.R2013physics,Jain.V;Zhang.Y;Palha.A;Gerristma.M2021CMA,Nakata.Y;Urade.Y;Nakanishi.T2019,Oden.J1974,Schoberl.J2008mc,Wieners.C;Barbara.W2011sinum,Licht.M2017FoCM}. Prior works primarily address dual representations of finite element spaces; dual grids have usually been used for the construction of discretized dual complexes. In this paper, the discrete dual complexes by function spaces are both constructed on a same grid. Therefore, the duality argument can be designed to derive uniform discrete Poincar\'e inequalities leveraging the adjoint relationship between $\od^k$ and $\odelta_{k+1}$, formulating some analogue of the closed range theorem; see Section \ref{sec:qcrt} for a quantifiable version of the closed range theorem. Further, our approach avoids nonlocal operators when establishing dual connections; all discrete operators involved are local.  The discrete complex duality can be expected further studied in future. Particularly, the validity of the structure of complex and the commutative diagram may not necessarily depend on the dualities \eqref{eq:n=r=c} or \eqref{eq:nabladivdual}.


\appendix


%
%
%
\section{Proofs of Lemmas \ref{lem:discontbpi}, \ref{lem:localhot} and \ref{lem:discrt}}
\label{sec:localwhitney}


\paragraph{\bf Proof of Lemma \ref{lem:discontbpi}}
Decompose $\fW^{\rm nc}_h\Lambda^k=\N(\od^k_h,\fW^{\rm nc}_h\Lambda^k)\opp (\fW^{\rm nc}_h\Lambda^k)^{\boldsymbol \lrcorner}$, orthogonal in $L^2\Lambda^k(\Omega)$. Given $\fsigma_h\in (\fW^{\rm nc}_h\Lambda^k)^{\boldsymbol \lrcorner}$, decompose orthogonally $\fsigma_h=\mathring{\fsigma}_h+\fsigma_h^{\boldsymbol \lrcorner}$, such that $\mathring{\fsigma}_h\in\mathcal{P}_0\Lambda^k(\mathcal{G}_h)$ and $\displaystyle \fsigma_h^{\boldsymbol \lrcorner}\in \bigoplus_{T\in\mathcal{G}_h}E_T^\Omega\okappa_T(\mathcal{P}_0\Lambda^{k+1}(T))$. As $\N(\od^k_h,\fW^{\rm nc}_h\Lambda^k)\subset\mathcal{P}_0\Lambda^k(T)$, we have further $\fsigma_h^{\boldsymbol \lrcorner}$ is orthogonal to $\N(\od^k_h,\fW^{\rm nc}_h\Lambda^k)$; therefore, $\mathring{\fsigma}_h$ is orthogonal to $\N(\od^k_h,\fW^{\rm nc}_h\Lambda^k)$, and further $\mathring{\fsigma}_h\in\R(\odelta_{k+1},\fW^*_{h0}\Lambda^{k+1})$ by Theorem \ref{thm:dishd}. Decompose 
$\fW^*_{h0}\Lambda^{k+1}=\N(\odelta_{k+1},\fW^*_{h0}\Lambda^{k+1})\opp (\fW^*_{h0}\Lambda^{k+1})^{\boldsymbol \lrcorner}$. Then $\R(\odelta_{k+1},\fW^*_{h0}\Lambda^{k+1})=\R(\odelta_{k+1},(\fW^*_{h0}\Lambda^{k+1})^{\boldsymbol \lrcorner})$. Therefore,
\begin{multline*}
\|\mathring{\fsigma}_h\|_{L^2\Lambda^k}=\sup_{\fmu_h\in (\fW^*_{h0}\Lambda^{k+1})^{\boldsymbol \lrcorner}}\frac{\langle \mathring{\fsigma}_h,\odelta_{k+1}\fmu_h\rangle_{L^2\Lambda^k}}{\|\odelta_{k+1}\fmu_h\|_{L^2\Lambda^k}} 
=
\sup_{\fmu_h\in (\fW^*_{h0}\Lambda^{k+1})^{\boldsymbol \lrcorner}}\frac{\langle \fsigma_h^{\boldsymbol \lrcorner},\odelta_{k+1}\fmu_h\rangle_{L^2\Lambda^k}+\langle \od^k_h\fsigma_h^{\boldsymbol \lrcorner},\fmu_h\rangle_{L^2\Lambda^{k+1}}}{\|\odelta_{k+1}\fmu_h\|_{L^2\Lambda^k}}
\\
\leqslant  \|\fsigma_h^{\boldsymbol \lrcorner}\|_{L^2\Lambda^k}+\|\od^k_h\fsigma_h^{\boldsymbol \lrcorner}\|_{L^2\Lambda^{k+1}}\sup_{\fmu_h\in (\fW^*_{h0}\Lambda^{k+1})^{\boldsymbol \lrcorner}}\frac{\|\fmu_h\|_{L^2\Lambda^{k+1}}}{\|\odelta_{k+1}\fmu_h\|_{L^2\Lambda^k}}
\\ 
\leqslant \|\od^k_h\fsigma_h\|_{L^2\Lambda^{k+1}} \icr(\od_h^k,\mathcal{P}^-_1\Lambda^k(\mathcal{G}_h)) +\|\od^k_h\fsigma_h^{\boldsymbol \lrcorner}\|_{L^2\Lambda^{k+1}}\icr(\odelta_{k+1},\fW^*_{h0}\Lambda^{k+1}).
\end{multline*}
Then $\displaystyle \|\fsigma_h\|_{L^2\Lambda^k}\leqslant \|\mathring{\fsigma}_h\|_{L^2\Lambda^k}+\|\fsigma_h^{\boldsymbol \lrcorner}\|_{L^2\Lambda^k} \leqslant \|\od^k_h\fsigma_h\|_{L^2\Lambda^{k+1}}(2\icr(\od_h^k,\mathcal{P}^-_1\Lambda^k(\mathcal{G}_h))+\icr(\odelta_{k+1},\fW^*_{h0}\Lambda^{k+1})).$ 
This completes the proof. \qed
\begin{remark}
No continuous problem or Sobolev space is used as a bridge here, and this is a direct relation based on the discrete adjoint connection between $\fW^*_{h0}\Lambda^{k+1}$ and $\fW^{\rm nc}_h\Lambda^k$. 
\end{remark}

\begin{lemma}\label{lem:pikappa}
There exists a constant $C_{k,n}$, depending on the regularity of $T$, such that 
\begin{equation}\label{eq:pimud}
\|\fmu\|_{L^2\Lambda^k(T)}\leqslant C_{k,n}h_T\|\od^k\fmu\|_{L^2\Lambda^{k+1}(T)},\ \ \mbox{for}\ \ \fmu\in\okappa_T(\mathcal{P}_0\Lambda^{k+1}(T)).
\end{equation}
\end{lemma}
\begin{proof}
~~Given  $\displaystyle \fmu=\sum_{\ixalpha\in\mathbb{IX}_{k+1,n}}C_{\ixalpha} \left(\sum_{j=1}^{k+1}(-1)^{j+1}\tilde x^{\ixalpha_j}\dx^{\ixalpha_1}\wedge\dx^{\ixalpha_2}\wedge\dots\wedge\dx^{\ixalpha_{j-1}}\wedge\dx^{\ixalpha_{j+1}}\wedge\dots\wedge \dx^{\ixalpha_{k+1}}\right),$ 
\begin{multline*}
|\fmu|^2_{H^1\Lambda^k(T)}
=\left\|\sum_{\ixalpha\in\mathbb{IX}_{k,n}}C_{\ixalpha}\sum_{j=1}^{k+1}(-1)^{j+1}\nabla \tilde{x}^{\ixalpha_j}\dx^{\ixalpha_1}\wedge\dots\wedge\dx^{\ixalpha_{j-1}}\wedge\dx^{\ixalpha_{j+1}}\wedge\dots\wedge \dx^{\ixalpha_{k+1}}\right\|_{L^2\Lambda^k(T)}^2
\\
=\left\langle \sum_{\ixalpha\in\mathbb{IX}_{k,n}}C_{\ixalpha}\sum_{j=1}^{k+1}(-1)^{j+1}\nabla \tilde{x}^{\ixalpha_j}\dx^{\ixalpha_1}\wedge\dots\wedge\dx^{\ixalpha_{j-1}}\wedge\dx^{\ixalpha_{j+1}}\wedge\dots\wedge \dx^{\ixalpha_{k+1}}, \right.\qquad
\\
\qquad\qquad\left.\sum_{\ixalpha'\in\mathbb{IX}_{k,n}}C_{\ixalpha'}\sum_{i=1}^{k+1}(-1)^{i+1}\nabla \tilde{x}^{\ixalpha'_i}\dx^{\ixalpha'_1}\wedge\dots\wedge\dx^{\ixalpha'_{i-1}}\wedge\dx^{\ixalpha'_{i+1}}\wedge\dots\wedge \dx^{\ixalpha'_{k+1}}\right\rangle_{L^2\Lambda^k(T)}
\\
=\sum_{\ixalpha\in\mathbb{IX}_{k,n}}\sum_{\ixalpha'\in\mathbb{IX}_{k,n}} C_{\ixalpha}C_{\ixalpha'}\sum_{j=1}^{k+1}\sum_{i=1}^{k+1} (-1)^{j+i}e^{\ixalpha_j}\cdot e^{\ixalpha_i} \Bigg\langle \dx^{\ixalpha_1}\wedge\dots\wedge\dx^{\ixalpha_{j-1}}\wedge\dx^{\ixalpha_{j+1}}\wedge\dots\wedge \dx^{\ixalpha_{k+1}}, 
\\
\dx^{\ixalpha'_1}\wedge\dots\wedge\dx^{\ixalpha'_{i-1}}\wedge\dx^{\ixalpha'_{i+1}}\wedge\dots\wedge \dx^{\ixalpha'_{k+1}}\Bigg\rangle_{L^2\Lambda^k(T)}=(k+1)|T|\sum_\alpha C_\alpha^2,
\end{multline*}
and $\displaystyle\left\|\od^k\mu\right\|_{L^2\Lambda^{k+1}(T)}^2=(k+1)^2\left\|\sum_\alpha C_\alpha \dx^{\ixalpha_1}\wedge\dx^{\ixalpha_2}\wedge\dots\wedge\dx^{\ixalpha_{k+1}}\right\|_{L^2\Lambda^{k+1}(T)}^2=(k+1)^2|T|\sum_{\ixalpha}C_{\ixalpha}^2.$
Namely 
$$
\displaystyle
\|\od^k\mu\|_{L^2\Lambda^{k+1}(T)}=\sqrt{k+1}|\fmu|_{H^1\Lambda^k(T)}.
$$ Therefore, by noting that $\int_T\tilde{x}^j=0$, with a constant $C_n$ depending on the regularity of $T$, we obtain
$$
\|\fmu\|_{L^2\Lambda^k(T)}\leqslant C_nh_T|\fmu|_{H^1\Lambda^k(T)}=C_n(k+1)^{-1/2}h_T\|\od^k\mu\|_{L^2\Lambda^{k+1}(T)}.
$$
This completes the proof. 
\end{proof}

%

\paragraph{\bf Proof of Lemma \ref{lem:localhot}} Evidently, 
\begin{multline}\label{eq:lcctr}
\icr(\od_h^k,\mathcal{P}^-_1\Lambda^k(\mathcal{G}_h))
=\sup_{\displaystyle\ftau_h\in \bigoplus_{T\in\mathcal{G}_h}E_T^\Omega\okappa_T(\mathcal{P}_0\Lambda^{k+1}(T))}\frac{\|\ftau_h\|_{L^2\Lambda^k}}{\|\od^k_h\ftau_h\|_{L^2\Lambda^{k+1}}}
\\
=\max_{T\in\mathcal{G}_h}\sup_{\ftau\in \okappa_T(\mathcal{P}_0\Lambda^{k+1}(T))}\frac{\|\ftau\|_{L^2\Lambda(T)}}{\|\od^k\ftau\|_{L^2\Lambda^{k+1}(T)}}.
\end{multline}
By Lemma \ref{lem:pikappa} and \eqref{eq:lcctr}, $\icr(\od_h^k,\mathcal{P}^-_1\Lambda^k(\mathcal{G}_h))$ is of $\mathcal{O}(h)$ order. 
\qed

\paragraph{\bf Proof of Lemma \ref{lem:discrt}} By virtue of Lemma \ref{lem:discontbpi} and Remark \ref{rem:mutual}, $\icr(\odelta_{k+1},\fW^*_{h0}\Lambda^{k+1})$ is controlled by $\icr(\od_h^k,\fW^{\rm nc}_h\Lambda^k)$ the same way. Further by Lemma \ref{lem:localhot}, we obtain Lemma \ref{lem:discrt}. \qed

\section{A quantifiable closed range theorem}
\label{sec:qcrt}

In this part, we establish a quantifiable version of the classical closed range theorem, in order to show how Lemma \ref{lem:discrt} can be viewed as a discrete analogue of the closed range theorem. 

Let $\xX$ and $\yY$ be two Hilbert spaces with respective inner products $\langle\cdot,\cdot\rangle_\xX$ and $\langle\cdot,\cdot\rangle_\yY$, and let $(\oT,\xD):\xX\to \yY$ be an unbounded linear operator, $\xD$ being the domain dense in $\xX$. The adjoint operator of $(\oT,\xD)$, denoted by $(\oT^*,\xD^*)$, is defined by
\begin{equation}\label{eq:adjointdef}
\langle \oT^*\yw,\xv\rangle_\xX=\langle \yw,\oT\xv\rangle_\yY,\ \ \forall\,\xv\in \xD,
\end{equation}
and the domain $\xD^*$ consists of such $\yw\in\yY$ that there exists an element in $\xX$ taken as $\oT^*\yw$ to satisfy \eqref{eq:adjointdef}. The closed range theorem (cf. \cite{Arnold.D2018feec,Yosida.K1965,Kato.T1980,Brezis.H2010} and other textbooks) asserts that 
\begin{equation}
\R(\oT,\xD)\ \mbox{is\ closed}\ \ \Longleftrightarrow \ \ \R(\oT^*,\xD^*)\ \mbox{is\ closed}. 
\end{equation}
It further follows by Lemma \ref{lem:pivscr} that 
\begin{equation}
\icr(\oT,\xD)<\infty\ \ \Longleftrightarrow\ \  \icr(\oT^*,\xD^*)<\infty. 
\end{equation}

The theorem below further gives a preciser quantification of the closed range theorem. 

\begin{theorem}\label{thm:cohpi}
For $(\oT,\xD):\xX\to \yY$ and $(\yT,\yD):\yY\to\xX$ a pair of closed densely defined adjoint operators,
\begin{equation}\label{eq:idicr}
\icr(\oT,\xD)= \icr(\yT,\yD). 
\end{equation}
\end{theorem}
\begin{proof}
~~Recalling the Helmholtz decomposition $\xX=\N(\oT,\xD)\opp\overline{\R(\yT,\yD)}$, we have
\begin{equation}
\xD^{\boldsymbol \lrcorner}=\xD\cap (\N(\oT,\xD))^\perp=\xD\cap \overline{\R(\yT,\yD)}.
\end{equation}
Therefore, provided that $0<\icr(\yT,\yD)<\infty$ and thus $\overline{\R(\yT,\yD)}=\R(\yT,\yD)$, given $\xv\in\xD^{\boldsymbol \lrcorner}$, there exists a $\yw\in \yD^{\boldsymbol \lrcorner}$, such that $\xv=\yT\yw$, then $\|\yw\|_{\yY}\leqslant \icr(\yT,\yD) \|\xv\|_{\xX}$ and
$$
\|\xv\|_\xX^2=\langle\xv,\xv\rangle_{\xX}=\langle \xv,\yT\yw\rangle_{\xX}=\langle\oT\xv,\yw\rangle_{\yY}\leqslant \|\oT\xv\|_\yY\|\yw\|_{\yY}\leqslant \icr(\yT,\yD) \|\oT\xv\|_\yY\|\xv\|_{\xX}.
$$
Therefore, $\|\xv\|_\xX\leqslant \icr(\yT,\yD)\|\oT\xv\|_\xX$ for any $\xv\in \xD^{\boldsymbol \lrcorner}$ and $\icr(\oT,\xD)\leqslant \icr(\yT,\yD)<\infty$. Similarly, $\infty>\icr(\oT,\xD)\geqslant \icr(\yT,\yD)$; note that $(\oT,\xD)$ is the adjoint operator of $(\yT,\yD)$. Namely, if one of $\icr(\oT,\xD)$ and $\icr(\yT,\yD)$ is finitely positive, then $\icr(\oT,\xD)=\icr(\yT,\yD)$.

If $\icr(\yT,\yD)=0$, then $\R(\yT,\yD)=\left\{0\right\}$ and $\xD^{\boldsymbol\lrcorner}=\left\{0\right\}$. It follows then $\icr(\oT,\xD)=0$. Finally, if one of $\icr(\oT,\xD)$ and $\icr(\yT,\yD)$ is $+\infty$, then so is the other. The proof is completed. 
\end{proof}



\begin{thebibliography}{10}

\bibitem{Arbogast.T;Correa.M2016}
Todd Arbogast and Maicon~R Correa.
\newblock Two families of ${H}({\rm div})$ mixed finite elements on
  quadrilaterals of minimal dimension.
\newblock {\em SIAM Journal on Numerical Analysis}, 54(6):3332--3356, 2016.

\bibitem{Armentano.M;Duran.R2004asymptotic}
Mar{\'\i}a~G Armentano and Ricardo~G Dur{\'a}n.
\newblock Asymptotic lower bounds for eigenvalues by nonconforming finite
  element methods.
\newblock {\em Electron. Trans. Numer. Anal}, 17(2):93--101, 2004.

\bibitem{Arnold.D;Falk.R;Winther.R2010bams}
Douglas Arnold, Richard Falk, and Ragnar Winther.
\newblock Finite element exterior calculus: from {H}odge theory to numerical
  stability.
\newblock {\em Bulletin of the American mathematical society}, 47(2):281--354,
  2010.

\bibitem{Arnold.D2018feec}
Douglas~N Arnold.
\newblock {\em Finite element exterior calculus}.
\newblock SIAM, 2018.

\bibitem{Arnold.D;Falk.R1989}
Douglas~N Arnold and Richard~S Falk.
\newblock A uniformly accurate finite element method for the
  {R}eissner--{M}indlin plate.
\newblock {\em SIAM Journal on Numerical Analysis}, 26(6):1276--1290, 1989.

\bibitem{Arnold.D;Falk.R;Winther.R2006acta}
Douglas~N Arnold, Richard~S Falk, and Ragnar Winther.
\newblock Finite element exterior calculus, homological techniques, and
  applications.
\newblock {\em Acta Numerica}, 15:1--155, 2006.

\bibitem{Arnold.D;Falk.R;Winther.R2009CMAME}
Douglas~N Arnold, Richard~S Falk, and Ragnar Winther.
\newblock Geometric decompositions and local bases for spaces of finite element
  differential forms.
\newblock {\em Computer Methods in Applied Mechanics and Engineering},
  198(21-26):1660--1672, 2009.

\bibitem{Berchenko.Y2021duality}
Yakov Berchenko-Kogan.
\newblock Duality in finite element exterior calculus and {H}odge duality on
  the sphere.
\newblock {\em Foundations of Computational Mathematics}, 21(5):1153--1180,
  2021.

\bibitem{Boffi.D;Brezzi.F;Fortin.M2013}
Daniele Boffi, Franco Brezzi, and Michel Fortin.
\newblock {\em Mixed finite element methods and applications}, volume~44.
\newblock Springer, 2013.

\bibitem{Brezis.H2010}
Haim Brezis.
\newblock {\em Functional Analysis, Sobolev Spaces and Partial Differential
  Equations}.
\newblock Springer Science \& Business Media, 2010.

\bibitem{Bringmann.P;Ketteler.J;Schedensack.M2024}
Philipp Bringmann, Jonas~W Ketteler, and Mira Schedensack.
\newblock Discrete {H}elmholtz decompositions of piecewise constant and
  piecewise affine vector and tensor fields.
\newblock {\em Foundations of Computational Mathematics}, pages 1--45, 2024.

\bibitem{Buffa.A;Christiansen.S2007dual}
Annalisa Buffa and Snorre Christiansen.
\newblock A dual finite element complex on the barycentric refinement.
\newblock {\em Mathematics of Computation}, 76(260):1743--1769, 2007.

\bibitem{Christiansen.S;Winther.R2008smoothed}
Snorre Christiansen and Ragnar Winther.
\newblock Smoothed projections in finite element exterior calculus.
\newblock {\em Mathematics of Computation}, 77(262):813--829, 2008.

\bibitem{christiansen2020poincare}
Snorre~H Christiansen and Martin~W Licht.
\newblock Poincar{\'e}--{F}riedrichs inequalities of complexes of discrete
  distributional differential forms.
\newblock {\em BIT Numerical Mathematics}, 60(2):345--371, 2020.

\bibitem{Ciarlet.P1978book}
Philippe~G Ciarlet.
\newblock {\em The finite element method for elliptic problems}.
\newblock North-Holland, Amsterdam, 1978.

\bibitem{Clement.P1975}
Ph~Cl{\'e}ment.
\newblock Approximation by finite element functions using local regularization.
\newblock {\em Revue fran{\c{c}}aise d'automatique, informatique, recherche
  op{\'e}rationnelle. Analyse num{\'e}rique}, 9(2):77--84, 1975.

\bibitem{Crouzeix.M;Raviart.P1973}
Michel Crouzeix and P-A Raviart.
\newblock Conforming and nonconforming finite element methods for solving the
  stationary {S}tokes equations {I}.
\newblock {\em Revue fran{\c{c}}aise d'automatique informatique recherche
  op{\'e}rationnelle. Math{\'e}matique}, 7(R3):33--75, 1973.

\bibitem{Dlotko.P;Specogna.R2013physics}
Pawe{\l} D{\l}otko and Ruben Specogna.
\newblock Physics inspired algorithms for (co) homology computations of
  three-dimensional combinatorial manifolds with boundary.
\newblock {\em Computer Physics Communications}, 184(10):2257--2266, 2013.

\bibitem{Ern.A;Guermond.J2017M2AN}
Alexandre Ern and Jean-Luc Guermond.
\newblock Finite element quasi-interpolation and best approximation.
\newblock {\em ESAIM: Mathematical Modelling and Numerical Analysis},
  51(4):1367--1385, 2017.

\bibitem{Falk.R;Winther.R2014}
Richard Falk and Ragnar Winther.
\newblock Local bounded cochain projections.
\newblock {\em Mathematics of Computation}, 83(290):2631--2656, 2014.

\bibitem{Fortin.M;Soulie.M1983}
M~Fortin and M~Soulie.
\newblock A non-conforming piecewise quadratic finite element on triangles.
\newblock {\em International Journal for Numerical Methods in Engineering},
  19(4):505--520, 1983.

\bibitem{Gawlik.E;Holst.M;Licht.M2021}
Evan Gawlik, Michael~J Holst, and Martin~W Licht.
\newblock Local finite element approximation of {S}obolev differential forms.
\newblock {\em ESAIM: Mathematical Modelling and Numerical Analysis},
  55(5):2075--2099, 2021.

\bibitem{Hiptmair.R2002acta}
Ralf Hiptmair.
\newblock Finite elements in computational electromagnetism.
\newblock {\em Acta Numerica}, 11:237--339, 2002.

\bibitem{Jain.V;Zhang.Y;Palha.A;Gerristma.M2021CMA}
Varun Jain, Yi~Zhang, Artur Palha, and Marc Gerritsma.
\newblock Construction and application of algebraic dual polynomial
  representations for finite element methods on quadrilateral and hexahedral
  meshes.
\newblock {\em Computers \& Mathematics with Applications}, 95:101--142, 2021.

\bibitem{Kato.T1980}
Tosio Kato.
\newblock {\em Perturbation theory for linear operators}, volume 132.
\newblock Springer Science \& Business Media, 2013.

\bibitem{Lee.J;Winther.R2018local}
Jeonghun Lee and Ragnar Winther.
\newblock Local coderivatives and approximation of {H}odge {L}aplace problems.
\newblock {\em Mathematics of Computation}, 87(314):2709--2735, 2018.

\bibitem{Licht.M2019mc}
Martin Licht.
\newblock Smoothed projections and mixed boundary conditions.
\newblock {\em Mathematics of Computation}, 88(316):607--635, 2019.

\bibitem{Licht.M2019mcweakly}
Martin Licht.
\newblock Smoothed projections over weakly {L}ipschitz domains.
\newblock {\em Mathematics of Computation}, 88(315):179--210, 2019.

\bibitem{Licht.M2017FoCM}
Martin~Werner Licht.
\newblock Complexes of discrete distributional differential forms and their
  homology theory.
\newblock {\em Foundations of Computational Mathematics}, 17(4):1085--1122,
  2017.

\bibitem{Marini.L1985sinum}
Luisa~Donatella Marini.
\newblock An inexpensive method for the evaluation of the solution of the
  lowest order {R}aviart--{T}homas mixed method.
\newblock {\em SIAM Journal on Numerical Analysis}, 22(3):493--496, 1985.

\bibitem{monk1991mixed}
Peter~B Monk.
\newblock A mixed method for approximating {M}axwell's equations.
\newblock {\em SIAM Journal on Numerical Analysis}, 28(6):1610--1634, 1991.

\bibitem{Nakata.Y;Urade.Y;Nakanishi.T2019}
Yosuke Nakata, Yoshiro Urade, and Toshihiro Nakanishi.
\newblock Geometric structure behind duality and manifestation of self-duality
  from electrical circuits to metamaterials.
\newblock {\em Symmetry}, 11(11):1336, 2019.

\bibitem{Oden.J1974}
JT~Oden.
\newblock Generalized conjugate functions for mixed finite element
  approximations of boundary value problems.
\newblock {\em The Mathematical Foundations of the Finite Element Method with
  Applications to Partial Differential Equations, Academic Press, New York},
  pages 629--667, 1972.

\bibitem{Ortner.C2011ima}
Christoph Ortner.
\newblock Nonconforming finite-element discretization of convex variational
  problems.
\newblock {\em IMA Journal of Numerical Analysis}, 31(3):847--864, 2011.

\bibitem{Quan.Q;Ji.X;Zhang.S2022}
Qimeng Quan, Xia Ji, and Shuo Zhang.
\newblock A lowest-degree quasi-conforming finite element de {R}ham complex on
  general quadrilateral grids by piecewise polynomials.
\newblock {\em Calcolo}, 59(1):5, 2022.

\bibitem{Raviart.P;Thomas.J1977}
Pierre-Arnaud Raviart and Jean-Marie Thomas.
\newblock A mixed finite element method for 2-nd order elliptic problems.
\newblock In {\em Mathematical aspects of finite element methods}, pages
  292--315. Springer, 1977.

\bibitem{Schoberl.J2008mc}
Joachim Sch{\"o}berl.
\newblock A posteriori error estimates for {M}axwell equations.
\newblock {\em Mathematics of Computation}, 77(262):633--649, 2008.

\bibitem{Scott.R;Zhang.S1990intp}
L~Ridgway Scott and Shangyou Zhang.
\newblock Finite element interpolation of nonsmooth functions satisfying
  boundary conditions.
\newblock {\em Mathematics of Computation}, 54(190):483--493, 1990.

\bibitem{Shi.D;Pei.L2008low}
Dongyang Shi and Lifang Pei.
\newblock Low order {C}rouzeix-{R}aviart type nonconforming finite element
  methods for approximating {M}axwell's equations.
\newblock {\em Int. J. Numer. Anal. Model}, 5(3):373--385, 2008.

\bibitem{Wieners.C;Barbara.W2011sinum}
Christian Wieners and Barbara Wohlmuth.
\newblock A primal-dual finite element approximation for a nonlocal model in
  plasticity.
\newblock {\em SIAM Journal on Numerical Analysis}, 49(2):692--710, 2011.

\bibitem{Yosida.K1965}
K{\"o}saku Yosida.
\newblock {\em Functional analysis}.
\newblock Springer Science \& Business Media, 2012.

\bibitem{Zeng.H;Zhang.C;Zhang.S2023existence}
Huilan Zeng, Chen-Song Zhang, and Shuo Zhang.
\newblock On the existence of locally-defined projective interpolations.
\newblock {\em Applied Mathematics Letters}, 146:108789, 2023.

\bibitem{Zhang.S2021SCM}
Shuo Zhang.
\newblock An optimal piecewise cubic nonconforming finite element scheme for
  the planar biharmonic equation on general triangulations.
\newblock {\em Science China Mathematics}, 64(11):2579--2602, 2021.

\end{thebibliography}
\end{document}